\documentclass[10pt,english]{smfart}

\usepackage[T1]{fontenc}
\usepackage[english,francais]{babel}
\usepackage{latexsym,amscd,color}
\usepackage{amsmath,amsfonts,amssymb,mathrsfs}
\usepackage{enumerate,euscript}
\usepackage{amssymb,url,xspace,smfthm}
\input xy
\xyoption{all}

\newcommand{\BibTeX}{{\scshape Bib}\kern-.08em\TeX}
\newcommand{\T}{\S\kern .15em\relax }
\newcommand{\AMS}{$\mathcal{A}$\kern-.1667em\lower.5ex\hbox
        {$\mathcal{M}$}\kern-.125em$\mathcal{S}$}

\DeclareMathOperator{\des}{des}
\DeclareMathOperator{\sudeg}{\mathrm{su}\widehat{\deg}_{\mathrm{n}}}
\DeclareMathOperator{\udeg}{\mathrm{u}\widehat{\deg}_{\mathrm{n}}}
\DeclareMathOperator{\ndeg}{\widehat{\deg}_{\mathrm{n}}}
\DeclareMathOperator{\Hom}{Hom}

\DeclareMathOperator{\rang}{rk}

\DeclareMathOperator{\Spec}{Spec}

\DeclareMathOperator{\rk}{rk}
\DeclareMathOperator{\Aut}{Aut}

\DeclareMathOperator{\dega}{\widehat{\deg}}

\DeclareMathOperator{\sudegg}{\mathrm{sudeg}}
\DeclareMathOperator{\udegg}{\mathrm{udeg}}

\def\degi#1{\widehat{\deg}{}_{\mathrm{n}}^{(#1)}}

\def\Z{\mathbb{Z}}
\def\PP{\mathbb{P}}
\def\Q{\mathbb{Q}}
\def\R{\mathbb{R}}
\def\C{\mathbb{C}}

\def\cX{\mathcal{X}}

\def\olE{\overline{E}}
\def\olF{\overline{F}}
\def\olG{\overline{G}}
\def\olK{\overline{K}}
\def\olV{\overline{V}}
\def\olL{\overline{L}}
\def\olM{\overline{M}}

\def\cO{\mathcal{O}}
\def\cQ{\mathcal{Q}}
\def\OK{{\mathcal{O}_K}}

\def\mua{\hat{\mu}}

\def\cf{\emph{cf.}\,}

\title[Tensor products in Arakelov geometry]{Concerning the semistability of \\ tensor products in Arakelov geometry}
\date{\today}
\author{Jean-Beno\^{i}t Bost}
\address{Universit\'e Paris Sud -- Orsay, D\'epartement de Math\'ematiques}
\email{jean-benoit.bost@math.u-psud.fr}
\author{Huayi Chen}
\address{Institut de Math\'ematiques de Jussieu, Universit\'e Paris Diderot and BICMR, Peking University}
\email{chenhuayi@math.jussieu.fr}

\begin{document}
\def\smfbyname{}

\begin{abstract}
We study the semistability of the tensor product of hermitian vector bundles by using the $\varepsilon$-tensor product and the geometric (semi)stability of vector subspaces in the tensor product of two vector spaces.

Notably, for any number field $K$ and any hermitian vector bundles $\olE$ and $\olF$ over $\Spec \OK$, we show that the maximal slopes of $\olE,$ $\olF$, and $\olE \otimes \olF$ satisfy the following inequality :
\[\hat{\mu}_{\max}(\overline E\otimes\overline F)
\leqslant\hat{\mu}_{\max}(\overline E)+\hat{\mu}_{\max}(\overline F)+\frac 12\min\big(\log (\rang E),\log(\rang F) \big).\]

We also prove that, for any $\OK$-submodule $V$ of $E \otimes F$ of rank $\leqslant 4,$ the slope of $\olV$ satisfies:
\[\hat{\mu}(\overline V)\leqslant\hat{\mu}_{\max}(\overline E)+\hat{\mu}_{\max}(\overline F).\] 
This shows that, if $\overline E$ and $\overline F$ are semistable and if $\rang E. \rang F \leqslant 9$, then $\overline E\otimes\overline F$  also is semistable.
\end{abstract}

\begin{altabstract}
Nous \'etudions la semi-stabilit\'e du produit tensoriel de fibr\'es vectoriels hermitiens en utlisant le produit $\varepsilon$-tensoriel et la (semi-)stabilit\'e g\'eom\'etrique des sous-espaces vectoriels dans le produit tensoriel de deux espaces vectoriels.

En particulier, si $K$ d\'esigne un corps de nombres et si  $\olE$ et $\olF$ sont deux fibr\'es vectoriels hermitiens sur $\Spec \OK$, nous montrons que les pentes maximales de  $\olE,$ $\olF$ et $\olE \otimes \olF$ satisfont \`a l'in\'egalit\'e :
\[\hat{\mu}_{\max}(\overline E\otimes\overline F)
\leqslant\hat{\mu}_{\max}(\overline E)+\hat{\mu}_{\max}(\overline F)+\frac 12\min\big(\log (\rang E),\log(\rang F)\big).\]

Nous prouvons aussi que, pour tout  sous $\OK$-module $V$ de $E \otimes F$ de rang  $\leqslant 4,$ la pente de $\olV$ v\'erifie:
\[\hat{\mu}(\overline V)\leqslant\hat{\mu}_{\max}(\overline E)+\hat{\mu}_{\max}(\overline F).\] 
Cela entra\^{\i}ne que, si $\rang E. \rang F \leqslant 9$ et que $\overline E$ et $\overline F$ sont semistables, le produit tensoriel   $\overline E\otimes\overline F$  l'est aussi.

\end{altabstract}

\maketitle

\tableofcontents

\section{Introduction}

\subsection{}\label{debut} This article is devoted to problems  in Arakelov geometry  motivated by classical results in the theory of vector bundles on a projective curve over a field. 

Notably we investigate the arithmetic analogue of the following:

\begin{theo}\label{Thm:NS} Let $k$ be a field of characteristic $0$, and let $E$ and $F$ be two vector bundles on a smooth projective integral curve $C$ over $k$. 

If $E$ and $F$ are semistable, then $E\otimes F$ also is semistable.
\end{theo}

When $k= \C,$ Theorem \ref{Thm:NS} is a straightforward consequence of the existence, first established by Narasimhan and Seshadri \cite{Nara_Se65}, of projectively flat metrics on stable vector bundles (see also \cite{Donaldson83} and \cite{Siu87}, notably Chapter I, Appendix A). This entails Theorem \ref{Thm:NS} for an arbitrary base field $k$ of characteristic zero,  by considering a subfield $k_0$ of $k$ of finite type over $\Q$ over which $C$ and $E,$ and $F$ are defined, and a field embedding of $k_0$ in $\C$   (see for instance \cite{HuybrechtsLehn2010}, Section 1.3, for the relevant invariance property of semistability under extension of the base field). 

Algebraic proofs of Theorem \ref{Thm:NS} that do not rely on the ``analytic crutch'' of special hermitian metrics have been devised by several authors.

Notably such as an algebraic proof may be easily derived from Bogomolov's work on stable vector bundles (\cite{Bogomolov78}; see also \cite{Raynaud81}). Actually the fact that, over curves, Bogomolov's notion of semistable vector bundles (the so-called tensorial semistability) is equivalent to the classical one (in terms of slopes of subbundles) directly implies Theorem \ref{Thm:NS}.

An extension of Theorem \ref{Thm:NS}, concerning semistable $G$-bundles for $G$ an arbitrary reductive group, has been established by Ramanan and Ramanathan \cite{Ramanan_Ramanathan}. Their technique of proof, which like Bogomolov's one relies on geometric invariant theory, turned out to be crucial in the study of  the preservation of semistability by tensor products in various contexts (see notably  \cite{Totaro96}), in particular in Arakelov geometry (\cite{Chen_pm}).

Besides, Theorem \ref{Thm:NS} may be derived from the basic theory of ample and nef line bundles on projective schemes over fields, as pointed out for instance in \cite{Miyaoka87} (see  \emph{loc. cit.}, Section 3, completed by a reference to \cite{Hart2}, Corollary 5.3; see also \cite{Gieseker79}, and the discussion in \cite{LazarsfeldII}, Chapter 6). We shall discuss related arguments in more details in paragraph \ref{geotens} below.

 \subsection{} Recall that Arakelov geometry is concerned with analogues in an arithmetic context of projective varieties over some base field $k$, of vector bundles $F$ over them, and of their $\Z$-valued invariants defined by means of characteristic classes and intersection theory. The objects of interest in Arakelov geometry are projective flat schemes $\cX$ over $\Z$, say regular  for simplicity, and \emph{hermitian vector bundles} over them: by definition, a hermitian vector bundle  over $\cX$ is a pair $\olE := (E, \Vert.\Vert)$  consisting of some vector bundle $E$ over $\cX$ and of some $C^\infty$ hermitian metric $\Vert.\Vert$, invariant under complex conjugation, on the analytic vector bundle $E^{an}_\C$ over the complex manifold $\cX(\C)$, that is deduced from $E$ by the base change $\Z \hookrightarrow \C$ and analytification. Arakelov geometric invariants attached to such data are real numbers.
 
 The simplest instance of schemes $\cX$ as above are ``arithmetic curves'' $\Spec \OK$, defined by the ring of integers $\OK$ of some number field $K$. A hermitian vector bundle $\olE$ over $\Spec \OK$ is precisely the data of some finitely generated projective $\OK$-module $E$, and of a family of hermitian  norms $(\Vert.\Vert_\sigma)_{\sigma: K\hookrightarrow \C}$ on the finite dimensional $\C$-vector spaces $E_\sigma:= E\otimes_{\OK, \sigma}\C$, invariant under complex conjugation.
 
Basic operations, such as $\oplus,$ $\otimes,$ $\Lambda^\bullet,$ may be defined on hermitian vector bundles in an obvious way. Moreover, if $f: \cX' \rightarrow \cX$ is a morphism between $\Z$-schemes as above, and $\olE$ some hermitian vector bundle over $\cX,$ we may form its pull-back $$f^\ast\olE:= (f^\ast E, f_\C^{an \ast}\Vert.\Vert),$$ which is a hermitian vector bundle over $\cX'.$
 
 When $K\hookrightarrow K'$ is an extension of number fields and when $f$ is the morphism
 \begin{equation*}\label{deff}
 f: \Spec \cO_{K'} \longrightarrow \Spec \OK
\end{equation*}
defined by the inclusion $\OK \hookrightarrow \cO_{K'},$ this pull-back construction maps the hermitian vector bundle
$$\olE:= (E, (\Vert.\Vert_\sigma)_{\sigma: K\hookrightarrow \C})$$
to
$$f^\ast\olE :=  (E', (\Vert.\Vert'_{\sigma'})_{\sigma': K'\hookrightarrow \C}),$$
where $E':=E\otimes_{\OK}\cO_{K'}$ and, for any $\sigma': K'\hookrightarrow \C$ of restriction $\sigma:=\sigma'_{\mid K},$ we let 
$$\Vert.\Vert'_{\sigma'} := \Vert.\Vert_\sigma \mbox{ on } E'_{\sigma'} \simeq E_\sigma.$$

The \emph{Arakelov degree} of the hermitian vector bundle $\olE$ over $\Spec \OK$ is defined as the real number 
$$\dega \olE := \log \vert E/\OK s\vert - \sum_{\sigma: K \hookrightarrow \C} \log \Vert s\Vert_\sigma$$
when $\rk E =1,$ where $s$ denotes any element of $E\setminus\{0\},$ and as 
$$\dega \olE := \dega \Lambda^{\rk E} \olE$$
in general.

It turns out to be convenient to introduce also the \emph{normalized} Arakelov degree :
$$\ndeg \olE := \frac{1}{[K:\Q]} \dega \olE.$$
It is indeed invariant by extension of the base field $K$. Namely, for any extension of number fields $K \hookrightarrow K',$ we have
\begin{equation}\label{ndeginv}
\ndeg \pi^\ast \olE = \ndeg \olE
\end{equation}
where 
$$\pi: \Spec \cO_{K'} \longrightarrow \Spec \OK$$
denotes the morphism defined by the inclusion $\OK \hookrightarrow \cO_{K'}.$

As pointed out by Stuhler 
(when $\OK = \Z$) and Grayson, 
the Arakelov degree satisfies formal properties similar to the ones of the degree of vector bundles on smooth projective curves over some base field $k$. This allowed them to define some arithmetic analogues   of \emph{semistability}, of \emph{successive slopes}, and of the \emph{Harder-Narasimhan filtration},  concerning hermitian vector bundles over arithmetic curves (\cite{Stuhler76}, \cite{Grayson84}. We refer to \cite{Andre09} and \cite{Chen_hn} for general discussions of the formalisms of slopes and associated filtrations).

Indeed, if $\olE$ is an arithmetic vector bundle over $\Spec \OK,$ of positive rank $\rk E,$ its \emph{slope} is defined as the quotient 
$$\mua(\olE) := \frac{\ndeg \olE}{\rk E}.$$
The hermitian vector bundle $\olE := (E, \Vert.\Vert_{\sigma: K \hookrightarrow \C})$ is said to be \emph{semistable} if, for every non-zero $\OK$-submodule $F$ of $E,$ the hermitian vector bundle $\olF$ (defined by $F$ equipped with the restriction to $F_\sigma$ of the hermitian metric $\Vert.\Vert_\sigma$) satisfies
$$\mua(\olF) \leqslant \mua (\olE).$$

To define the successive slopes and the Harder-Narasimhan filtration of $\olE$, we consider the set of points in $\R^2$ of the form $(\rk F,\ndeg \olF),$ where $F$ varies over the $\OK$-submodules of $E.$ The closed convex hull of this set may be also described as 
$$\mathcal{C}_{\olE} := \{ (x,y) \in [0, \rk E] \times \R \mid y \leqslant P_{\olE}(x) \}$$ for some (uniquely determined) concave function
$$P_{\olE} : [0,\rk E] \longrightarrow \R,$$
which is  affine on every interval $[i-1,i],$ $i \in \{1,\ldots,\rk E\}.$ The \emph{$i$-th slope} $\mua_i(\olE)$ of $\olE$ is the derivative of $P_{\olE}$ on this interval. These slopes satisfy
$$\mua_1(\olE) \geqslant \mua_2(\olE) \geqslant \ldots \geqslant \mua_{\rk E}(\olE),$$
and
$$\sum_{i=1}^{\rk E} \mua_i(\olE) = \ndeg \olE.$$

Any vertex of the so-called canonical polygon $\mathcal{C}_{\olE}$ --- that is, any point $(i, P_{\olE}(i))$ where $i$ is either $0,$ $\rk E,$ or an integer in $]0,\rk E[$ such that 
$\mua_i(\olE) > \mua_{i+1}(\olE)$ --- is of the form $(\rk F, \ndeg \olF)$ for some uniquely determined $\OK$-submodule $F$ of $E.$ Moreover these submodules are saturated (a $
\mathcal O_K$-submodule $F$ of $E$ is said to be saturated if $E/F$ is torsion-free) and fit into a filtration
$$0 = E_0 \subsetneq E_1 \subsetneq \ldots \subsetneq E_N = E,$$
the so-called \emph{Harder-Narasimhan}, or \emph{canonical, filtration} of $\olE.$ 
It is straightforward that
\begin{equation*}
\begin{split}
\mua_{\max} (\olE) & := \mua_1 (\olE) \\
& = \max_{0 \neq F \subset E} \mua (\olF),
\end{split}
\end{equation*}
and that $\olE$ is semistable iff $\mua_{\max} (\olE)  = \mua (\olE)$, or equivalently iff $P_{\olE}$ is linear, or iff the Harder-Narasimhan filtration of $E$ is trivial (that is : $E_0 =E,$ $E_1=E$).

It is natural to consider the joint data  of the Harder-Narasimhan filtration $(E_k)_{0 \leqslant k \leqslant N}$ and of the successive slopes $$\mua (\overline{E_k/E_{k-1}}) = \mua_i(\olE) \mbox{ for } \rk E_{k-1} < i \leqslant \rk E_k.$$ 
They define the weighted Harder-Narasimhan filtration, an instance of weighted filtration, or weighted flag, in the $K$-vector space $E_K$ (in the sense of Section \ref{SubSec:weighted} \emph{infra}), or equivalently, of a point in the real vectorial building $I(\mathrm{GL}(E_K))$ attached to the reductive group $\mathrm{GL}(E_K)$ over $K$ (see \cite{Rousseau81a}).

\subsection{} The successive slopes $(\mua_i(\olE))_{1\leqslant i \leqslant \rk E}$ of a hermitian vector bundle $\olE$ satisfy remarkable formal properties.

For instance, from the basic properties of the Arakelov degree, one easily derives the following relations between the slopes of some hermitian vector bundle $\olE$ over $\Spec \OK$ and of its dual $\olE^\vee$ :
\begin{equation}\label{transfer}
\mua_i(\olE^\vee) = - \mua_{\rk E + 1 -i} (\olE).
\end{equation}

Besides, the invariance of the normalized Arakelov degree under number field extensions (\ref{ndeginv}) and the canonical character of the Harder-Narasimhan filtration, together with a simple Galois descent argument, imply (\cite{BostBour96}, Proposition A.2; compare \cite{HuybrechtsLehn2010}, Lemma 3.2.2) :
\begin{prop}\label{arpullback} Let $K$ denote a number field, $K'$ a finite extension of $K$, and 
$$\pi : \Spec \cO_{K'} \longrightarrow \Spec \OK$$
the scheme morphism defined by the inclusion $\OK \hookrightarrow \cO_{K'}.$

If $\olE$ is a hermitian vector bundle over $\Spec \OK$, and if $(E_i)_{0\leqslant i \leqslant N}$ denotes its Harder-Narasimhan filtration, then the Harder-Narasimhan filtration of $\pi^\ast \olE$ is $(\pi^\ast E_i)_{0\leqslant i \leqslant N}$ $(:= (E_i \otimes_{\OK} \cO_{K'}))$.

Consequently, for any $i \in \{1, \ldots, \rk E\},$
\begin{equation}\label{arpullbackslopes}
\mua_i(\pi^\ast \olE) = \mua_i (\olE),
\end{equation}
and $\pi^\ast \olE$ is semistable if (and only if) $\olE$ is semistable.
\end{prop}

A hermitian vector bundle $\olE := (E, \Vert. \Vert)$ over $\Spec \Z$ is nothing else than a  Euclidean lattice. Its Arakelov degree may be expressed in terms of its covolume ${\rm covol}\,  \olE$:
$$\widehat{\deg}\, \olE = - \log {\rm covol}\, \olE,$$
and its successive slopes are related to its successive minima, classically defined in geometry of numbers (\cite{Cassels97}, Chapter VIII). For instance, if we consider the first of those :
$$\lambda_1(\olE) := \min_{v \in E \setminus \{0\}} \Vert v \Vert,$$
Minkowski's First Theorem easily implies (see  \cite{Borek05}, or \cite{Bost_Kunnemann}, 3.2 and 3.3) :
\begin{equation}\label{compfirstmax}
0 \leqslant \mua_{\max} (\olE) - \log \lambda_1(\olE)^{-1} \leqslant \frac{1}{2} \log \rk E.
\end{equation}

In brief, successive slopes may be seen as variants of the classical successive minima of Euclidean lattices which make sense for hermitian vector bundles, not only over $\Spec \Z,$ but over arbitrary arithmetic curves, and are compatible with extensions of number fields. 

Moreover they appear to be ``better behaved'' than the classical successive minima. For instance  relations (\ref{transfer}) may be seen as an avatar of the ``transference theorems'' relating successive minima of some euclidean lattice and of its dual lattice (cf. \cite{Cassels97}, VIII.5, and  \cite{Banaszczyk93}). However these theorems are not equalities, but involve error terms depending on the rank of the lattice, whereas the slope relations (\ref{transfer}) are exact equalities.

These properties of the slopes of hermitian vector bundles make them especially useful when using Diophantine approximation techniques in the framework of Arakelov geometry : working with slopes turns out to be an alternative to the classical use of Siegel's Lemma; their ``good behavior'' --- notably their invariance under number fields extensions --- make them a particularly flexible tool, leading to technical simplifications in diverse Diophantine geometry proofs (see for instance \cite{BostBour96} and \cite{Bost2001}, and \cite{BostICM} for references).  

\subsection{} Let $K$ be a number field and let $\olE$ and $\olF$ be two hermitian vector bundles of positive rank over $\Spec \OK.$

Recall that we have :
$$\mua(\olE \otimes \olF) = \mua(\olE)+ \mua(\olF).$$
Applied to submodules of $E$ and $F$, this implies the inequality:
$$\mua_{\max} (\olE \otimes \olF) \geqslant  \mua_{\max}(\olE) + \mua_{\max}(\olF)$$
(see for instance \cite{Chen_pm}, Proposition 2.4).

 In this article, we study the following:
 
 \begin{enonce}{Problem}\label{mumaxconj} Let $K$ be a number field and let $\olE$ and $\olF$ be two hermitian vector bundles of positive rank over $\Spec \OK.$

Does the equality
\begin{equation}\label{equamax} 
\mua_{\max} (\olE \otimes \olF) = \mua_{\max}(\olE) + \mua_{\max}(\olF)
\end{equation}
hold ?

In other words, is it true that, for any $\OK$-submodule $V$ of positive rank in  $E\otimes F,$ the following inequality
\begin{equation}\label{equamaxbis} 
\mua (\olV) \leqslant \mua_{\max}(\olE) + \mua_{\max}(\olF)
\end{equation}
holds ?

\end{enonce}

This problem is easily seen to admit the following equivalent formulation:

\begin{enonce}{Problem}\label{semistableconj}
Does the semistability of $\olE$ and $\olF$ imply the one of $\olE \otimes \olF$ ?
\end{enonce}
\noindent and  the following one as well: 

\begin{enonce}{Problem}\label{HNtens} Does the weighted Harder-Narasimhan filtration of $\olE \otimes \olF$ coincide with the tensor product\footnote{See  Section \ref{Filtr}, \emph{infra}, for the definition of weighted flags or filtrations, and 
of their tensor products.} of 
the ones  of $\olE$ and $\olF$ ?
\end{enonce}

\subsection{}\label{geotens} To put these problems in perspective, we present a simple proof showing that their geometric counterparts,  concerning vector bundles on smooth projective curves over a field of characteristic zero, have a positive answer. This will demonstrate how Theorem 1.1 may be derived from the basic theory of ample line bundles on projective varieties. (Our argument is in the spirit of the arguments in \cite{Miyaoka87} and \cite{Hart2} alluded to in \ref{debut}, \emph{supra}. However it avoids to resort to the representation theory of the linear group. We refer the reader to \cite{HuybrechtsLehn2010}, sections 3.1-2, and \cite{LazarsfeldII}, Chapter 6, for related results and references.)\footnote{The content of this section was discussed in \cite{Bost97Bis} and partially appears in   \cite{Bost_Kunnemann}, Section 3.8. Closely related arguments have been obtained independently by Y. Andr\'e (\cf \cite{Andre10}, Section 2).} 

Let $C$ be a smooth projective integral curve over some base field $k$, that we shall assume algebraically closed for simplicity.

The analogue of the invariant $\log \lambda_1(\olE)^{-1}$ (attached to some Euclidean lattice $\olE$) for a vector bundle $E$ of positive rank over $C$ is its \emph{upper degree}\footnote{This notion, and the notion of stable upper degree defined below, are variants of the \emph{lower degree} $\rm{ld} (E)$ and \emph{stable lower degree} $\mathrm{sld}(E)$ introduced in \cite{EsnaultViehweg90}: we have $\udegg E =- \mathrm{ld}(E^\vee),$ and $\sudegg E = - \mathrm{sld}(E^\vee).$}:
$$\udegg E := \max_{L \hookrightarrow E, \rk L = 1} \deg L.$$

Recall that, if $C'$ is another smooth projective integral curve over $k$ and $\pi: C' \rightarrow C$ a dominant (and consequently surjective and finite) $k$-morphism, then for any vector bundle $E$ over $C$, we have: 
$$\deg \pi^\ast E = \deg \pi . \deg E,$$
where $\deg \pi$ denotes the degree of the field extension $\pi^\ast : k(C) \rightarrow k(C').$ Consequently the slope of vector bundles, defined as
$$\mu(E) := \frac{\deg E}{\rk E},$$
satisfies 
$$\mu(\pi^\ast E)  = \deg \pi. \mu(E)$$ 
and therefore the maximal slope, defined as 
$$\mu_{\max}(E):=\max_{0 \neq F \hookrightarrow E} \mu(F),$$
satisfies
$$\mu_{\max}(\pi^\ast E)  \geqslant  \deg \pi. \mu_{\max}(E).$$
In general, when the characteristic of $k$ is positive, this inequality may be strict  (\cite{Hart4}, \cite{Gieseker73}). However, when $k$ is a field of characteristic zero or, more generally, when $\pi$ is generically separable, a Galois descent argument establishes the equality 
\begin{equation}\label{pullbackmumax}
\mu_{\max}(\pi^\ast E)  =  \deg \pi. \mu_{\max}(E).
\end{equation}
(See for instance \cite{HuybrechtsLehn2010}, Lemma 3.2.2. It is the geometric counterpart  of  Proposition \ref{arpullback}. The occurrence of the factor $\deg \pi$ in (\ref{pullbackmumax}), that does not appear in the arithmetic analogue (\ref{arpullbackslopes}), comes from the use of the \emph{normalized} degree in the arithmetic case.) 

Similarly we have :
$$\udegg \pi^\ast E \geqslant \deg \pi. \udegg E.$$
Here the inequality may be strict even in the characteristic zero case\footnote{
Situations where this arises may be easily constructed as follows, say when $k= \C$. Consider an \'etale Galois covering $\pi: C'\rightarrow C,$ of Galois group  the full symmetric group $\mathfrak{S}_d,$ $d:= \deg \pi,$ and define $E$ as the kernel of the trace map $\pi_\ast \cO_{C'}\rightarrow \cO_C.$ This vector bundle of rank $d-1$ and degree zero is stable. (This follows from the irreducibily of the permutation representation of $\mathfrak{S}_d$ on $\C^{d-1} \simeq \{(x_1,\ldots,x_d) \in \C^d \mid x_1+ \cdots +x_d =0 \}.$) Consequently, when  $d\geqslant 3,$ we have 
$\udegg E < 0.$
Besides,  $\pi^\ast E$ is a trivial vector bundle of rank $d-1$, and consequently
$\udegg \pi^\ast E = 0.$}.
However the \emph{stable upper degree} of $E$,
$$\sudegg E := \sup_{\pi : C' \twoheadrightarrow C} \frac{1}{\deg \pi} \udegg \pi^\ast E,$$
defined by considering all the finite coverings $\pi : C' \twoheadrightarrow C$ as above, becomes compatible with finite coverings by its very definition. Namely, for every finite covering $\pi : C' \twoheadrightarrow C$, we have:
$$
\sudegg \pi^\ast E = \deg \pi.\sudegg E.
$$

The following result is well-known (compare \cite{LazarsfeldII}, Section 6.4), but does not seem to be explicitly stated in the literature.

\begin{prop}\label{mumsudegg} If $k$ has characteristic zero, then, for any vector bundle of positive rank over $C$, we have:
\begin{equation}\label{sumax}
\sudegg E = \mu_{\max}(E).
\end{equation}
\end{prop}

Let us briefly recall how this Proposition follows from the basic properties of nef line bundles over projective varieties\footnote{
A line bundle $M$ on some projective integral scheme $X$ over $k$ is \emph{nef} if, for every smooth projective integral curve $C'$ over $k$ and any $k$-morphism $\nu: C' \rightarrow X,$ the line bundle $\nu^\ast L$ on $C'$ has non-negative degree. Then, for any closed integral subscheme $Y$ in $X,$ the intersection number $c_1(M)^{\dim Y}. Y $ is non-negative. This definition and this property clearly depend only on the class of $L$ modulo numerical equivalence and immediately extend to $\Q$-line bundles.}. 

\begin{proof}[Proof of Proposition \ref{mumsudegg}]
Together with the invariance of the maximum slope by pull-back (\ref{pullbackmumax}), the trivial inequality
$$\udegg E \leqslant \mu_{\max} (E)$$
applied over finite coverings of $C$ shows that:
$$
\sudegg E \leqslant \mu_{\max}(E).
$$

To establish the opposite inequality,
consider the projective bundle
$$\PP_{\rm sub} (E) = \PP(E^\vee) := \mathbf{Proj} ({\rm Sym}(E^\vee)),$$
its structural morphism
$$p: \PP_{\rm sub} (E)\longrightarrow C,$$
and the canonical quotient line bundle over $\PP_{\rm sub} (E)$
$$p^\ast E^\vee \stackrel{q}{\longrightarrow} \cO_{E^\vee}(1).$$
Consider also a point $c$ in $C(k)$. The line bundle $\cO_{E^\vee}(1)$ is ample relatively to $p$, and $\cO(c)$ is ample over $C$. Consequently the set of rational numbers
$$\left\{ \delta \in \Q \mid \cO_{E^\vee}(1) \otimes p^\ast \cO (\delta.c) \mbox{ is nef} \right \}$$
is of the form $\Q \cap [\lambda, + \infty[$ for some (uniquely defined) real number $\lambda.$ 

We claim that
\begin{equation}\label{lambdasud}
\lambda = \sudegg E.
\end{equation}
Indeed, for any smooth projective integral curve $C'$ over $k$, the datum of a $k$-morphism 
$$\nu : C' \longrightarrow \PP_{\rm sub}(E)$$
is equivalent to the data of a $k$-morphism
$$\pi : C' \longrightarrow C$$
and of a quotient line bundle over $C'$ of $\pi^\ast E^\vee$: 
$$\label{Mq}
\pi^\ast E^\vee \longrightarrow M.$$
(This follows from the very definition of $\PP_{\rm sub}(E) \simeq \PP(E^\vee).$ To $\nu$ is attached the morphism  $\pi := p\circ \nu$ and the quotient line bundle 
$$\pi^\ast E^\vee \simeq \nu^\ast p^\ast E^\vee \stackrel{\nu^\ast q}{\longrightarrow} M:= \nu^\ast \cO_{\vee}(1).)$$ In turn, giving  the quotient line bundle $\pi^*E^\vee\rightarrow M$ is equivalent to giving a (saturated) sub-line bundle:
$$M^\vee \hookrightarrow \pi^\ast E.$$
Moreover, with the above notation, we have for every $\delta \in \Q$ :
\begin{equation*}
\begin{split}
\deg_{C'} \nu^\ast (\cO_{E^\vee}(1) \otimes p^\ast \cO(\delta.c)) & = \deg_{C'} M + \deg_{C'} \pi^\ast \cO(\delta.c) \\
& = - \deg_{C'} M^\vee + \delta \deg \pi.
\end{split}
\end{equation*}
This degree is always non-negative when $\pi$ is constant (then $\nu$ factorizes through a fiber of $p$, over which $\cO_{E^\vee}(1)$ is ample). Consequently it is non-negative for every $\nu: C' \rightarrow \PP_{\rm sub} (E)$ iff $\delta \geqslant \sudegg E.$ This establishes (\ref{lambdasud}).

To complete the proof, we are left to show
\begin{equation}\label{maxlambda}
\mu_{\rm max} (E) \leqslant \lambda.
\end{equation}

Let $e:= \rk E.$ Recall that the direct image by $p$ of the class $c_1(\cO_{E^\vee}(1))^e$ in the Chow group $\mathrm{CH}^e (\PP_{\rm sub}(E))$ is the class
$$p_\ast c_1(\cO_{E^\vee}(1))^e = - c_1(E)$$
in $\mathrm{CH}^1(C)$  and that
$$p_\ast c_1(\cO_{E^\vee}(1))^{e-1} = [C]$$
in $\mathrm{CH}^0(C)$ (see for instance \cite{Fulton}, Section 3.2). Consequently
\begin{equation*}
\begin{split}
c_1(\cO_{E^\vee}(1)\otimes p^\ast \cO(\delta.c))^e.[\PP_{\rm sub}(E)] & = p_\ast(c_1(\cO_{E^\vee}(1)\otimes p^\ast \cO(\delta.c))^e).[C] \\
& = \sum_{0\leqslant i \leqslant e} \binom{e}{i} p_\ast(c_1(\cO_{E^\vee}(1))^i) c_1(\cO(\delta.c))^{e-i}. [C]\\
& = (-c_1(E) + e \delta c_1(\cO(c)). [C]\\
& = - \deg E + e \delta \\
& = e (\delta - \mu(E)).
\end{split}
\end{equation*}

More generally, for any (saturated) subvector bundle $F$ of $E$, of positive rank $f$, the projective bundle $\PP_{\rm sub}(F)$ embeds canonically into  $\PP_{\rm sub}(E)$
and the restriction of $\cO_{E^\vee}(1)$ to $\PP_{\rm sub}(F)$ may be identified with $\cO_{F^\vee}(1),$ and we get :
\begin{equation}
c_1(\cO_{E^\vee}(1)\otimes p^\ast \cO(\delta.c))^f.[\PP_{\rm sub}(F)] = f (\delta - \mu(F)).
\end{equation}

When $\cO_{E^\vee}(1)\otimes p^\ast \cO(\delta.c)$ is nef, this intersection number is non-negative. This establishes  the required inequality (\ref{maxlambda}).
\end{proof}

Proposition \ref{mumsudegg} will allow us to derive Theorem \ref{Thm:NS}
in the equivalent form:

\begin{theo}\label{Thm:NS'} If $k$ has characteristic zero, then for any two vector bundles $E$ and $F$ of positive rank over $C,$ we have:
$$\mu_{\max}(E\otimes F) = \mu_{\max}(E) + \mu_{\max}(F).$$
\end{theo}

 Observe that, as in the arithmetic case, the slope of vector bundles of positive rank over a curve is additive under tensor product:
 $$\mu(E\otimes F) = \mu(E) + \mu(F),$$
 and that this immediately yields the inequality:
 $$\mu_{\max}(E\otimes F) \geqslant \mu_{\max}(E) + \mu_{\max}(F).$$
 Our derivation of the opposite inequality will rely on the following observation (compare \cite{Bost_Kunnemann}, Prop. 3.4.1 and 3.8.1):
 \begin{lemm}\label{Lemm:udegmax} For any two vector bundles $E$ and $F$ of positive rank over $C$, we have
 \begin{equation}\label{udegmax}
\udegg E \otimes F \leqslant \mu_{\max}(E) + \mu_{\max}(F).
\end{equation}
\end{lemm}

Observe that  the inequality (\ref{udegmax}) holds even when the base field $k$ has positive characteristic.

\begin{proof} Let us denote $K:= k(C)$ the function field of $C$.

Consider $L$ a line bundle in $E\otimes F$. By restriction to the generic point of $C$, it defines a one-dimensional vector subspace $L_K$ in the $K$-vector space $(E\otimes F)_K \simeq E_K \otimes_K F_K.$ Let $l$ be a non-zero element of $L_K$, and let $V_K$ and $W_K$ the minimal $K$-vector subspaces of $E_K$ and $F_K$ such that $l \in V_K \otimes_KW_K$ : the element $l$ of $E_K\otimes_K F_K$ may be identified with a $K$-linear map $T$ in $\Hom_K(E_K^\vee, F_K)$ (resp. with $^t T$ in $\Hom(F_K^\vee,E_K)$), and $W_K$ (resp. $V_K$) is the image of $T$ (resp. $^t T$). Moreover the $K$-linear map $\tilde{T}:V_K^\vee \rightarrow W_K$ defined by $l$ seen as an element of $V_K\otimes_K W_K$ is invertible.

Clearly $V_K$ and $W_K$ do not depend of the choice of $l$ in $L_K\setminus\{0\}$, and if $V$ and $W$ denote the saturated coherent subsheaves of $E$ and $F$ whose restrictions to the generic point of $C$ are $V_K$ and $W_K$, then $L$ is included in $V\otimes W$, which is saturated in $E\otimes F$.

Let $r$ denote the common rank of $V$ and $W$. Let us denote
$${\rm{det}}_{L,K} := L_K^{\otimes r} \stackrel{\sim}{\longrightarrow} (\Lambda^r V \otimes \Lambda^r W)_K$$
the unique  $K$-isomorphism (of one-dimensional $K$-vector spaces) that maps $l^{\otimes r}$ to $\det \tilde{T}$, which is an element of 
$$\Hom_K( \Lambda^r V_K^\vee, \Lambda^r W_K) \simeq (\Lambda^r V \otimes \Lambda^r W)_K.$$
Observe that, if $L$, $V$, and $W$ may be trivialized over some open subscheme $U$ of $C$, and if $l$ is a regular section of $L$ over $U,$ then $\det \tilde{T}$ defines a morphism of line bundles over $U$, from $\Lambda^r V_U^\vee$ to $\Lambda^r W_U$. This shows that ${\rm{det}}_{L,K}$ is  the restriction to the generic point of $C$ of a morphism of line bundles over $C$:
$${\rm{det}}_L : L^{\otimes r} \longrightarrow \Lambda^r V \otimes \Lambda^r W.$$

Since ${\rm{det}}_{L,K}\neq 0,$ we obtain :
$$\deg L^{\otimes r} \leqslant \deg (\Lambda^r V \otimes \Lambda^r W),$$
or equivalently:
$$\deg L \leqslant \mu(V) +\mu(W).$$
This shows that
$$\deg L \leqslant \mu_{\max}(E) + \mu_{\max}(F).$$
\end{proof}
 
 \begin{proof}[Proof of Theorem \ref{Thm:NS'}] For any finite covering $\pi : C' \rightarrow C,$ we may apply Lemma \ref{Lemm:udegmax} to $\pi^\ast E$ and $\pi^\ast F$. So we have :
 $$\frac{1}{\deg \pi} \udegg \pi^\ast(E\otimes F) \leqslant \frac{1}{\deg \pi} \mu_{\max} (\pi^\ast E) + \frac{1}{\deg \pi} \mu_{\max}(\pi^\ast F).$$
 When $k$ is a field of characteristic zero, we may use the ``invariance'' of the maximal slope under finite coverings (\ref{pullbackmumax}), and we obtain
 \begin{equation}\label{sudeggmax}
 \sudegg E\otimes F \leqslant \mu_{\max}(E) +\mu_{\max}(F).
\end{equation}

Together with Proposition \ref{mumsudegg}, this establishes the inequality 
$$\mu_{\max}(E \otimes F)  \leqslant \mu_{\max}(E) +\mu_{\max}(F).$$
\end{proof}

\subsection{}\label{Subsec:sudegree}In the arithmetic situation, we may attach analogues of the upper and stable upper  degree to hermitian vector bundles over arbitrary arithmetic curves.

Namely, if $K$ is a number field and $\olE$ a hermitian vector bundle of positive rank over $\Spec \OK$, we define its \emph{upper arithmetic degree} $\udeg \olE$ as the maximum of the (normalized) Arakelov degree $\ndeg \olL$ of a $\OK$-submodule $L$ of rank $1$ in $E$ equipped with the hermitian structure induced by the one of $\olE$ (\cf \cite{Bost_Kunnemann}, 3.3), and its \emph{stable  upper arithmetic degree} as
$$\sudeg \olE := \sup_{K'} \udeg \pi^\ast \olE,$$
where $K'$ varies over all finite extensions of $K,$ and $\pi: \Spec \cO_{K'} \rightarrow \Spec \OK$ denotes the associated morphism. 

The analogue of Lemma \ref{Lemm:udegmax} still holds in this context. In other words, the conjectural estimate (\ref{equamaxbis}) is true when $\rk V = 1.$ (\cf \cite{Bost_Kunnemann}, Prop. 3.4.1). Together with the invariance of arithmetic slopes under extensions of number fields (Proposition \ref{arpullback}), this establishes:

\begin{prop}\label{sudegtens} For any two hermitian vector bundles of positive rank $\olE$ and $\olF$ over $\Spec \OK,$
\begin{equation}
\sudeg \olE \otimes \olF \leqslant \mua_{\max}(\olE) + \mua_{\max}(\olF).
\end{equation}
\end{prop} 

To establish the additivity of maximal slopes under tensor product for hermitian vector bundles (\ref{equamax}), it would be enough to know that, as in the geometric case dealt with in Proposition \ref{mumsudegg}, the stable upper degree and the maximal slope of a hermitain vector bundle coincide. Unfortunately, this is \emph{not} the case.

More specifically, the trivial inequality
$$\udeg \olE \leqslant \mua_{\max} (\olE),$$
together with Proposition \ref{arpullback} show that
\begin{equation}
\sudeg\olE \leqslant \mua_{\max} (\olE).
\end{equation}
 Hermitian vector bundles for which this inequality is strict are easily found\footnote{We refer to \cite{Gaudron_Remond}, Section 3.3, for further results on hermitian vector bundles $\olE$ for which  $ \mua_{\max} (\olE) -\sudeg\olE $ take ``large'' positive values. }:

 \begin{prop}\label{propA2}
Let $A_2$ be the euclidean lattice  $\Z^3 \cap A_{2,\R}$ in the hyperplane
$$A_{2,\R} := \left\{(x_1,x_2,x_3) \in \R^3 \mid x_1 +x_2 +x_3 =0\right\}$$
of $\R^3$ equipped with the standard euclidean norm.

Then, seen as a hermitian vector bundle over $\Spec \Z,$ $A_2$ is semi-stable and satisfies
\begin{equation}\label{Aprems}
\mua_{\max}(A_2) = \mua(A_2) = -\frac{1}{4} \log 3
\end{equation}
and
\begin{equation}\label{Abis}
\sudeg (A_2) = -\frac{1}{2} \log 2.
\end{equation}
\end{prop}

\begin{proof} The semistability of $A_2$ follows from the fact that the natural representation by permutation of the coordinates of the symmetric group $\frak{S}_3$ on $A_{2,\R}$ is (absolutely) irreducible and preserves $A_2$ (\cf \cite{BostBour96}, Proposition A.3, and Proposition \ref{PropirredAut} \emph{infra}).

The covolume of $A_2$ is $\sqrt{3}$, and consequently
$$\dega A_2 = - \log \sqrt{3}.$$
Together with the semistability of $A_2$, this establish (\ref{Aprems}).

To derive (\ref{Abis}), firstly we observe that, for any $i \in \{1,2,3\}$, the $\Z$-submodule $L_i$ of rank 1 in $A_2$ 
defined by the equation $X_i = 0$, 
satisfies
$$\dega \olL_i = -\frac{1}{2} \log 2.$$
(The line $X_1 =0$, for instance, admits as $\Z$-basis the vector $(0,1-1)$, of euclidean norm $\sqrt{2}$.) 

Beside we claim that, \emph{for any number field $K$ and any $\OK$-submodule $L$ of rank $1$ in $\pi^\ast A_2 := A_2 \otimes_\Z \OK$ distinct from the three modules $\pi^\ast L_i,$ we have :}
\begin{equation}\label{Ater}
\ndeg \olL \leqslant -\frac{1}{2} \log 3.
\end{equation}
This will complete the proof of (\ref{Abis}) (actually, of a more precise result). Observe also that (\ref{Ater}) may be an equality :  consider a number field $K$ containing a primitive third root of unity $\zeta_3$, and $L:= \OK (1, \zeta_3, \zeta_3^2).$

 To establish the upper bound (\ref{Ater}), consider the  $\OK-$linear morphisms
$$X_{i\mid L} : L \longrightarrow \cO_K,  \,\,\,\, i\in\{1,2,3\}.$$
By hypothesis they are non-zero, and therefore the following inequalities hold, as straightforward consequences of the definition of the Arakelov degree :
$$\ndeg \olL \leqslant -\frac{1}{[K:\Q]} \sum_{\sigma : K \hookrightarrow \C} \log \Vert X_{i \mid L}\Vert_\sigma.$$
This implies :
$$ 3 \ndeg \olL \leqslant -\frac{1}{[K:\Q]} \sum_{\sigma : K \hookrightarrow \C} \log \prod_{1\leqslant i \leqslant 3}\Vert X_{i \mid L}\Vert_\sigma.$$
To conclude, observe that, for any $(z_1, z_2, z_3)$ in $A_{2,\C}$, we have
$$\vert z_1 z_2 z_3 \vert \leqslant 3^{-3/2} (\vert z_1 \vert^2 + \vert z_2\vert^2+\vert z_3\vert^2)^{3/2},$$
and therefore, for every embedding $\sigma : K \hookrightarrow \C,$
$$\log \prod_{1\leqslant i \leqslant 3}\Vert X_{i \mid L}\Vert_\sigma \leqslant -\frac{3}{2} \log 3.$$
\end{proof}

Proposition \ref{propA2} and its proof may be  extended to higher dimensions. We leave as an exercise for the reader the proof of the following :

\begin{prop}\label{propAn} Let $n$ be a positive integer, and let $A_n$ be the euclidean lattice  $\Z^{n+1} \cap A_{n,\R}$ in the hyperplane
$$A_{n,\R} := \left\{(x_0,\ldots, x_n) \in \R^{n+1} \mid x_0 +\ldots +x_n =0\right\}$$
of $\R^{n+1}$ equipped with the standard euclidean norm.

The hermitian vector bundle $A_n$ over $\Spec \Z$  is semi-stable and its (maximal) slope is  $$\mua(A_n) = -\frac{1}{2n} \log (n+1).$$

Moreover,
$$\sudeg (A_2) = -\frac{1}{2} \log 2.
$$
More precisely,
for any number field $K$ and any $\OK$-submodule $L$ of rank 1 in $\pi^\ast A_n := A_n \otimes_\Z \OK,$ we have
$$\ndeg \olL \leqslant -\frac{1}{2} \log \alpha(L),
$$
where $\alpha(L)$ denotes the integer $\geqslant 2$ defined as the cardinality of
$$\left\{ i \in \{0,\ldots,n\} \mid X_{i,L} \neq 0 \right\}.$$
\end{prop}

Although the stable upper arithmetic degree and the maximal slope of a hermitian vector bundle $\olE$ over $\Spec \OK$ may not not coincide, they differ by  some additive error  bounded in terms of $\rk E$ only:

\begin{prop}\label{absMink} For any positive integer $r$, there exists $C(r)$ in $\R_+$ such that, for any number field $K$ and any hermitian vector bundle $\olE$ over $\Spec \OK$,
\begin{equation}\label{eq:absMink}
\mua_{\max} (\olE) \leqslant \sudeg (\olE) + C(r).
\end{equation}
\end{prop}

As pointed out in 
\cite{BostBour96}, Appendix, this follows from Zhang's theory of ample hermitian line bundles in Arakelov geometry (\cite{Zhang95}) applied to projective spaces, and actually appears as an arithmetic analogue of the above derivation of Proposition  \ref{mumsudegg} from the classical theory of ample and nef line bundles applied to projective bundles over curves. Proposition \ref{absMink} has been independently established by Roy and Thunder 
\cite{RoyThunder96}, who used direct arguments of geometry of numbers. Their proof shows that Proposition \ref{absMink} holds with
\begin{equation}
\label{RT} C(r) = \frac{\log 2}{2} (r-1).
\end{equation}

Actually, by using Zhang's theory carefully, one obtains that it actually holds with 
$$C(r) = \frac{1}{2}\ell(r),$$
where $$\ell(r) := \sum_{2\leqslant i \leqslant r} \frac{1}{i}$$
(see for instance  \cite{DavidPhilippon99}, section 4.2, or \cite{Gaudron_Remond}, section 3.2). This improves on (\ref{RT}), since $\ell(r) \leqslant \log r$, and is actually optimal, up to some bounded error term, as shown by Gaudron and R\'emond in \cite{Gaudron_Remond}.

Observe that, from Propositions \ref{sudegtens} and \ref{absMink}, we immediately derive:

\begin{coro}\label{tenserr} For any two hermitian vector bundles $\olE$ and $\olF$ over some arithmetic curve $\Spec \OK,$
we have:
\begin{equation}\label{eq:tenserr}
\mua_{\max} (\olE \otimes \olF) \leqslant \mua_{\max} (\olE) + \mua_{\max}(\olF) + C(\rk E . \rk F).
\end{equation}
\end{coro}
Here $C(\rk E . \rk F)$ denotes the value of the constant $C(r)$ occurring in Proposition \ref{absMink} when $r = \rk E . \rk F.$ According to the above discussion of the possible value of $C(r)$, we obtain (see also \cite{Andre10}, Theorem 0.5):
\begin{equation}\label{eq:tenserrbis}
\begin{split}
\mua_{\max} (\olE \otimes \olF) & \leqslant \mua_{\max} (\olE) + \mua_{\max}(\olF) + \ell(\rk E . \rk F) \\
& \leqslant \mua_{\max} (\olE) + \mua_{\max}(\olF) + \frac{1}{2} \log\rk E + \frac{1}{2}\log  \rk F.
\end{split}
\end{equation}

\subsection{} In 1997, the first author discussed Problems \ref{mumaxconj}--\ref{HNtens} above during a conference in Oberwolfach (\cite{Bost97Bis}), and presented some evidence for a \emph{positive} answer to these problems, namely (i) the upper bound in Proposition \ref{sudegtens} and the consequent validity of (\ref{equamax}) up to an error term stated in Corollary \ref{tenserr} above, (ii) a positive answer to  Problem  \ref{semistableconj} when $\olE$ or $\olF$ possesses an automorphism group acting irreducibly (\cf  Proposition \ref{PropirredAut} \emph{infra}),  and (iii) when $\olE$ and $\olF$ have rank two (\cf \,paragraph \ref{rk2} below).

He also discussed the geometric approach in Section \ref{geotens} and the examples in Propositions \ref{propA2} and \ref{propAn}  showing that  $\sudeg \olE$ can be smaller than $\mua_{\max}(\olE)$. This discussion, showing that, concerning vector bundles and projective bundles over curves, the analogy between number fields and function fields is not complete\footnote{a point made already clear by the study of projective geometry in the Arakelov context in \cite{BGS94}.}, made especially intriguing the available evidence for the ``good behavior''  under tensor product of slopes of hermitian vector bundles.

In the next two paragraphs, we briefly discuss the positive evidence for this preservation mentioned in points (ii) and (iii) above.

\subsection{}\label{irredAut} If $\olE := (E, \Vert.\Vert)$ is a hermitian vector bundle over  $\Spec \OK$, we shall denote $\Aut \olE$ its automorphism group, namely the finite subgroup of $\mathrm{GL}_{\OK}(E)$ consisting of elements the image of which in $\mathrm{GL}_\C (E_\sigma)$ is unitary with respect to $\Vert.\Vert_\sigma$ for every embedding $\sigma: K \hookrightarrow \C$.

Observe that, if $\olF$ denotes another hermitian vector bundle over $\Spec \OK,$ of positive rank, there is a natural injection of automorphism groups:
$$
\begin{array}{rcl}
  \Aut \olE & \longrightarrow   & \Aut (\olE \otimes \olF)  \\
  \phi & \longmapsto  & \phi \otimes \mathrm{Id}_F.   
\end{array}
$$ 

The following Proposition is a straightforward consequence of the invariance of Harder-Narasimhan filtrations under automorphism groups (compare \cite{BostBour96}, Proposition A.3):

\begin{prop}\label{PropirredAut} Let $\olE$ be a hermitian vector bundle of positive rank over $\Spec \OK$ over such that the natural representation
$$\Aut \olE \hookrightarrow \mathrm{GL}_{\OK} (E) \hookrightarrow \mathrm{GL}_K (E_K)$$
of $\Aut \olE$ in $E_K$ is absolutely irreducible.

For any  hermitian vector bundle of positive rank $\olF$ over $\Spec \OK,$ of    canonical filtration $(F_i)_{0\leqslant i \leqslant N}$, the canonical filtration  of $\olE \otimes \olF$ is $(E \otimes F_i)_{0\leqslant i \leqslant N}$.

In particular $\olE$ is semistable, and for any semistable  hermitian vector bundle $\olF$ over $\Spec \OK$, $\olE \otimes \olF$ is semistable.

\end{prop}

Hermitian vector bundles $\olE$ such that the natural representation of $\Aut \olE$ in $E_K$ is absolutely irreducible naturally arise. For instance, diverse remarkable classical euclidean lattices --- seen as hermitian vector bundles over $\Spec \Z$ --- satisfy this condition, for instance the root lattices  and the Leech lattice (\cf \cite{ConwaySloane99}). Other examples of such hermitian vector bundles are given by the hermitian vector bundles defined by sections of ample line bundles over (semi-)abelian schemes over $\Spec \OK$ (\cf \cite{BostBour96}, \cite{Pazuki09}). 

\subsection{}\label{rk2} We now turn to Problems \ref{mumaxconj}--\ref{HNtens} 
for rank two vector bundles. We want to explain why, in this case, they have positive answers --- that is, explicitly:

\begin{prop}\label{proprk2}Let $K$ be a number field and let $\olE$ and $\olF$ be two hermitian vector bundles of rank two over $\Spec \OK$.

For any $\OK$-submodule of $E\otimes F$ of positive rank, we have:
\begin{equation}\label{murk2}
\mua(\olV) \leqslant \mua_{\max}(\olE) + \mua_{\max}(\olF).
\end{equation}
\end{prop}

This result is actually contained in the more general results established in the sequel. However some basic, but important, geometric ideas are already displayed in the proof of this simple case, that avoids the intricacies of the more general situations considered below.

We need to establish (\ref{murk2}) for $V$ saturated of rank 1, 2, or 3. The case $\rk V = 1$  has been dealt with in Proposition \ref{sudegtens}). The case $\rk V = 3$  reduces to  the case $\rk V=1$ by means of the following simple duality result applied to $\olG = \olE \otimes \olF,$ together with the isomorphism $(\olE \otimes \olF)^\vee \simeq \olE^\vee \otimes \olF^\vee$ (compare \cite{deShalit_Parzan}, section 1.5):

\begin{lemm}\label{Lem:dualite} Let $\olG := (G, \Vert.\Vert)$ be a hermitian vector bundle over some arithmetic curve $\Spec \OK$, and let $\olG^\vee := (G^\vee, \Vert. \Vert^\vee)$ denote its dual. 

Let $V$ be some saturated $\OK$-submodule of $G,$ and let $V^\perp$ its orthogonal (saturated)  submodule in $G^\vee$:
$$V^\perp := \left\{ \xi \in G^\vee \mid \xi_{\mid V} = 0\right\}.$$ Equipped with the restrictions of $\Vert.\Vert$ and $\Vert.\Vert^\vee$, they define hermitian vector bundles  $\olV$ and $\olV^\perp$ that fit into a natural short exact sequence which is compatible with metrics :
$$0 \longrightarrow \olV^\perp \longrightarrow \olG^\vee \longrightarrow \olV^\vee \longrightarrow 0.$$

Consequently
$$\dega \olV^\perp = \dega \olV - \dega \olG,$$
and, if $\rk V$ (and consequently $\rk V^\perp$) is $\neq 0$ or $\rk G$,
$$\mua (\olV) \leqslant \mua(\olG) \mbox{ iff } \mua(\olV^\perp) \leqslant \mua(\olG^\vee).$$
\end{lemm}

To handle the case $\rk V =2,$ we consider the quadratic form $\det_{E,F}$ on $E\otimes F$ with values in the line $\Lambda^2 E \otimes \Lambda^2 F$ --- that is the element of $S^2(E\otimes F)^\vee \otimes \Lambda^2 E \otimes \Lambda^2 F$  which maps an element $\phi$ of $E \otimes F \simeq \Hom (E^\vee, F)$ to its determinant $\det \phi$ in $\Hom(\Lambda^2 E^\vee, \Lambda^2 F)\simeq \Lambda^2 E \otimes \Lambda^2 F$ --- and the relative position of the projective line $\PP_{\rm sub}(V)$ and of the quadric $\cQ$ of equation $\det_{E,F} = 0,$ that parametrizes ``simple tensors'' in $E \otimes F$, in the projective space $\PP_{\rm sub}(E \otimes F)$.

One easily checks that $\PP_{\rm sub}(V)$ is contained in $\cQ$ precisely when $V$ is of the form $V = L \otimes F$ for some saturated $\OK$--submodule $L$ of rank 1 in $E$, or of the form $V = E \otimes M$ for some saturated $\OK$--submodule $M$ of rank 1 in  $F$. When this holds, we have :
$$\mua(\olV) = \mua(\olL) + \mua(\olF) \leqslant \mua_{\max}(\olE) + \mua (\olF)$$
or 
$$\mua(\olV) = \mua(\olF) + \mua(\olM) \leqslant \mua(\olE) + \mua_{\max}(\olF).$$

Suppose now that $\PP_{\rm sub}(V)$ is not contained in $\cQ$. There exists some $\overline{K}$-rational point in the intersection of the line $\PP_{\rm sub}(V)$ and of the hypersurface $\cQ$ (the scheme 
$(\PP_{\rm sub}(V)\cap\cQ)_K$ is actually a zero dimensional scheme of length 2 over $K$). After possibly replacing the field $K$ by some finite extension --- which is allowed according to the invariance of slopes under  extension of the base field --- we may assume that there exist some rational point $P$ in this intersection. The saturated $\OK$-submodule of $V$ defined by the $K$-line attached to $P$ is of the form $L \otimes M$ where $L$ (resp. $M$) denotes a saturated $\OK$-submodule of rank 1 in $E$ (resp. in $F$). The composite morphism
$$p : V \hookrightarrow E \otimes F \twoheadrightarrow E/L \otimes F$$
is non-zero (indeed $V \not\subset L \otimes F,$ since $\PP_{\rm sub}(V) \not\subset \cQ$) and its kernel contains $L\otimes M$. Consequently its image ${\rm im\,} p$ is a $\OK$-submodule of rank 1 in $E/L \otimes F$ and its kernel $\ker p$ is $L\otimes M.$ Moreover $p : \olV \rightarrow \overline{E/L} \otimes \olF$ is of archimedean norms $\leqslant 1$. This implies:
$$\ndeg \overline{{\rm im\,} p}  \leqslant \mua_{\max}(\overline{E/L}\otimes \olF) = \ndeg \overline{E/L} + \mua_{\max} (\olF)$$
and
$$\ndeg \olV \leqslant \ndeg \olL \otimes \olM  + \ndeg \overline{{\rm im\,}p}.$$ 
As $\ndeg \olL \otimes \olM = \ndeg \olL  + \ndeg  \olM$ and $\ndeg \overline{E/L} = \ndeg \overline{E} - \ndeg \overline{L},$ we finally obtain:
$$\ndeg \olV \leqslant \ndeg \olM + \mua_{\max}(F) \leqslant \mua_{\max}(\olE) + \mua_{\max}(\olF).$$
This completes the proof of (\ref{murk2}) and of Proposition \ref{proprk2}.

Observe that, even if one wants to establishes the upper bound  (\ref{murk2}) for euclidean lattices only (that is, when $K=\Q$), the above arguments uses the formalism of hermitian vector bundles over $\Spec \OK$ for larger number fields : the introduction of the point $P$ requires in general to introduce a quadratic extension of the initial base field $K$.

It is also worth noting that a refinement of these arguments  shows that, when 
$(\PP_{\rm sub}(V)\cap\cQ)_K$ is a \emph{reduced} zero-dimensional scheme  --- or equivalently when \emph{the line $\PP_{\rm sub}(V)$ and the quadric $\cQ$ meet in precisely two points over} $\overline{K}$, or when the \emph{Pl\"ucker points} $\Lambda^2 V_{\olK}\setminus \{0\}$ of $V_{\olK}$ in $\Lambda^2(E\otimes F)_{\olK}$ are \emph{semi-stable} with respect to the action of $\mathrm{SL}(E_{\olK}) \times \mathrm{SL}(F_{\olK})$ in the sense of geometric invariant theory --- then the following  stronger inequality holds:
\begin{equation}\label{2stable}
\ndeg \olV \leqslant  \mua(\olE) + \mua(\olF).
\end{equation}

Indeed, in this situation there exists a ${\olK}$-basis $(e_1,e_2)$ (resp. $(f_1,f_2)$) of $E_{\olK}$ (resp. $F_{\olK}$) such that $(e_1\otimes f_1, e_2 \otimes f_2)$ is a $\olK$-basis of $V_{\olK}$. Actually this holds precisely when
$$\left\{[e_1\otimes f_1], [e_2\otimes f_2]\right\} = (\PP_{\rm sub}(V)\cap\cQ)(\olK).$$
After a possible finite extension of $K$, we may assume that $(e_1,e_2)$ and $(f_1,f_2)$) are actually $K$-bases of $E_{K}$ and $F_{K}$. Then the inequality (\ref{2stable}) is a consequence of the following local inequalities, valid with obvious notations for every place $v$ of $K$, $p$-adic or archimedean:
\begin{equation}\label{rk2loc}
\Vert e_1\otimes f_1 \wedge e_2 \otimes f_2 \Vert_v \geqslant \Vert e_1 \wedge e_2\Vert_v . \Vert f_1 \otimes f_2\Vert_v. 
\end{equation}

To establish (\ref{rk2loc}), by homogeneity, we may assume that $\Vert e_1 \Vert_v = \Vert f_1 \Vert_v =1$ and choose linear forms of norm $1$, 
$$\eta : E_{K_v} \longrightarrow K_v \mbox{ and } \phi :E_{K_v} \longrightarrow K_v$$
such that
$$\eta(e_1) = \phi(f_1) = 0.$$
Then 
\begin{equation}
\label{rk2encore}
\vert \eta(e_2)\vert_v = \Vert e_1 \wedge e_2\Vert_v
\mbox{ and } \vert \phi(f_2) \vert_v = \Vert f_1 \otimes f_2\Vert_v.
\end{equation}
Moreover $\eta \otimes \phi$ is a linear form of norm $1$ on $(E\otimes F)_{K_v}$ and vanishes on the vector $e_1 \otimes f_1$, which satisfies $\Vert e_1 \otimes f_1 \Vert_v = 1.$ Consequently the restriction of $\eta \otimes \phi$ to $V_{K_v}$ has norm $\leqslant 1$, and
\begin{equation}
\label{rk2enfin}
\vert (\eta \otimes \phi) (e_2 \otimes f_2 )\vert_v \leqslant  \Vert e_1\otimes f_1  \wedge e_2\otimes f_2\Vert_v.
\end{equation}
Finally (\ref{rk2loc}) follows from (\ref{rk2encore}) and (\ref{rk2enfin}).

\subsection{}\label{ElsdSP}

Additional positive evidence for a positive answer to Problems Problems \ref{mumaxconj}--\ref{HNtens}
has been obtained by Elsenhans (in his G\"ottingen Diplomarbeit \cite{Elsenhans01} under the supervision of U. Stuhler) and by   de Shalit and Parzanchevski (\cite{deShalit_Parzan}). 

They show that the inequality
\begin{equation}\label{Basic}
\mua(\olV) \leqslant \mua_{\max}(\olE) + \mua_{\max}(\olF)
\end{equation}
holds for any sublattice $V$ in the tensor product of two euclidean lattices $\olE$ and $\olF$ under the following conditions on their ranks :
$\rk E =2$ and $\rk V \leqslant 4$ (\cite{Elsenhans01}), or $\rk V \leqslant 3$ (\cite{deShalit_Parzan}). 

The arguments in  \cite{Elsenhans01} and \cite{deShalit_Parzan} rely on some specific features of the theory of euclidean lattices and of their reduction theory in low ranks, and consequently do not generalize to hermitian vector bundles over general arithmetic curves $\Spec \OK$. 

Besides,  these arguments make conspicuous the role of the \emph{tensorial rank} of elements of $E\otimes F$ when investigating the inequality (\ref{Basic}).
(The tensorial rank of a vector $\alpha$ in $E\otimes F$ is defined as the minimal integer $m\geqslant 0$ such that $\alpha$ can be written in the form
$x_1\otimes y_1+\cdots+x_m\otimes y_m$
with $x_1,\ldots,x_m\in E$ and $y_1,\ldots,y_m\in F$. It is also the rank of $\alpha$ considered as a homomorphism of $\mathbb Z$-modules from $E^\vee$ to $F$.  Note that split tensor vectors --- namely tensor vectors of tensorial rank $1$ --- have also been considered by Kitaoka  to study the first minimum of tensor product lattices; see \cite{Kitaoka93} for a survey.)

\subsection{}\label{approx}Concerning the ``approximate validity'' of the additivity of maximal slopes under tensor product (that is, of equality (\ref{equamax}) in Problem \ref{mumaxconj}) established by means of Zhang's version of the ``absolute Siegel's Lemma'' stated in Proposition \ref{absMink}, the best result available in the literature seems to be the one in  \cite{Gaudron_Remond}:
\begin{equation}\label{Equ:Gaudron-Remond}\hat{\mu}_{\max}(\overline E_1\otimes\cdots\otimes\overline E_n)\leqslant\sum_{i=1}^n\Big(\hat{\mu}_{\max}(\overline E_i)+\frac 12\log(\rang(E_i))\Big).\end{equation}
(Observe that \eqref{Equ:Gaudron-Remond} does not follow formally from the special case when $n=2$; a similar observation applies to  \eqref{Equ:Chen} \emph{infra}.) 

Another method for establishing that type of inequality has been developed by the second author (\cite{Chen_pm}). This method, which involves the geometric invariant theory of (the geometric generic fibers of ) subbundles in tensor products,  may be considered as an analogue in Arakelov geometry of the method of Ramanan and Ramanathan \cite{Ramanan_Ramanathan} (based on previous results of Kempf \cite{Kempf78}), and of its elaboration by Totaro in \cite{Totaro96}. It  is directly related to the study of geometric invariant theory in the context of Arakelov geometry and of related height inequalities by Burnol \cite{Burnol92}, Bost \cite{Bost94}, Zhang \cite{Zhang96}, and Gasbarri \cite{Gasbarri00}. 

The key point of the method consists of estimating the Arakelov degree of some hermitian line subbundle of a tensor product under some additional geometric semistability condition. By using the classical theory of invariants   (cf. \cite[Ch. III]{Weyl} and \cite[Appendix 1]{Atiyah_Bott_Patodi}) to make ``explicit'' the geometric semistabiltiy condition, one derives an upper bound for the Arakelov degree of the given hermitian line subbundle. The general case can be reduced to the former one (where some geometric semistability condition is satisfied) by means of Kempf's instability flag. By this technique, the upper bound
\begin{equation}\label{Equ:Chen}\hat{\mu}_{\max}(\overline E_1\otimes\cdots\otimes\overline E_n)\leqslant\sum_{i=1}^n\Big(\hat{\mu}_{\max}(\overline E_i)+\log(\rang(E_i))\Big)
\end{equation}
has been proved in \cite{Chen_pm}
for any family hermitian vector bundles $(\overline E_i)_{i=1}^n$ of positive rank over an arbitrary arithmetic curve $\Spec \OK$.

Beside avoiding the use of higher dimensional Arakelov geometry, implicit via Zhang's version of the ``absolute Siegel's Lemma''  in the previous approach, this method also shows that a positive answer to Problems \ref{mumaxconj}-\ref{HNtens} is equivalent to a positive answer to:  

\begin{enonce}{Problem}\label{mustableconj} Let $K$ be a number field and $(\overline E_i)_{i=1}^n$ a family of  hermitian vector bundles of positive rank over $\Spec \OK.$

Is it true that, for any $\OK$-submodule $V$ of positive rank in $E_1\otimes\cdots\otimes E_n$ such that the Pl\"ucker point of $V_{\olK}$ is semistable under the action of 
$\mathrm{SL}(E_{1,\olK})\times\cdots\times \mathrm{SL}(E_{n,\olK})$, the inequality 
\begin{equation}\label{equamustableconj} 
\mua (\olV) \leqslant\sum_{i=1}^n\hat{\mu}(\overline E_i)
\end{equation}
hold ?
\end{enonce}

\subsection{}The three lines of approach to the behavior of slopes of hermitian vector bundles discussed in paragraphs \ref{approx} and \ref{ElsdSP} may seem to be rather distinct at first sight. They are however closely related. For instance the first two approaches rely on some version of reduction theory  and of   Minkowski's first theorem. The tensorial rank, essential in the arguments of 
\cite{Elsenhans01} and \cite{deShalit_Parzan}, occurs naturally in the geometric invariant theory of tensor products, as it precisely distinguishes the orbits of $\mathrm{SL}(E_{\olK})\times \mathrm{SL}(F_{\olK})$ acting on $E_{\olK} \otimes F_{\olK}$ . It also appears in the proof of (\ref{2stable}) above and in 
 in  \cite[\S 5.2]{Gaudron_Remond}.  
 
 Moreover, as we show in the sequel, a suitable combination of these approaches leads to new evidences for  a positive answer to Problems \ref{mumaxconj}-\ref{HNtens} and \ref{mustableconj}, encompassing previously known results. It  also provides new insights on possible refinements of the conjectural slopes inequalities formulated in these problems.
 
Our  results may be summarized  as follows:

\begin{enonce*}{Theorem A} For any two hermitian vector bundles of positive rank  $\overline E$ and $\overline F$ 
over some arithmetic curve $\Spec \OK$, we have :
\[\hat{\mu}_{\max}(\overline E\otimes\overline F)
\leqslant\hat{\mu}_{\max}(\overline E)+\hat{\mu}_{\max}(\overline F)+\frac 12\min\big(\ell(\rang(E)),\ell(\rang(F))\big).\]
\end{enonce*}
(Recall that, for any integer $r\geqslant 1$, $\ell(r):= \sum_{2\leqslant j\leqslant r} \frac{1}{j} \leqslant \log r$.)

\begin{enonce*}{Theorem B}
Let $\overline E$ and $\overline F$ be two  hermitian vector bundles of positive rank over some arithmetic curve $\Spec \OK$.
For any hermitian vector subbundle $\overline V$ of $\overline E\otimes\overline F$ such that $\rang(V)\leqslant 4$, one has
\[\hat{\mu}(\overline V)\leqslant\hat{\mu}_{\max}(\overline E)+\hat{\mu}_{\max}(\overline F).\] 

Moreover, if $\overline E$ and $\overline F$ are semistable and if $\rang(E)\rang(F)\leqslant 9$, then $\overline E\otimes\overline F$ also issemistable.
\end{enonce*}

Theorem A implies (by induction on $n$) that the following inequality
\[\hat{\mu}_{\max}(\overline E_1\otimes\cdots\otimes\overline E_n)
\leqslant\sum_{i=1}^n\hat{\mu}_{\max}(\overline E_i)+\sum_{i=2}^n\frac 12\log\rang(E_i)\]
holds for any finite collection $(\overline E_i)_{i=1}^n$ of  hermitian vector bundles of positive ranks on $\Spec\mathcal O_K$.
This result has been announced in  2007 (cf. \cite{Bost07Bis}), together with the related  upper bound for the Arakelov degree of line subbundles in the tensor product of several normed vector bundles stated to Corollary \ref{Cor:tensor} {\it infra}. 

The proof of these results (and the actual formulation of Corollary \ref{Cor:tensor}) uses $\varepsilon$-tensor products of hermitian vector bundles. For any two hermitian vector bundles $\overline E$ and $\overline F$, the $\varepsilon$-tensor product $\overline E\otimes_{\varepsilon}\overline F$ is an adelic vector bundle (in the sense of Gaudron \cite{Gaudron07}). Its construction is  similar to the one of the usual hermitian tensor product $\overline E\otimes\overline F$, except that, at archimedean places,  the hermitian tensor product norms (or equivalently,  the Hilbert-Schmidt norms when we identify $E\otimes F$ to $\Hom(E^\vee,F)$) are replaced by the $\varepsilon$-tensor product norms (that is, the operator norms on
$E\otimes F \simeq \Hom(E^\vee,F)$).

It happens that various slope inequalities, already known  for hermitian tensor products, still hold for $\varepsilon$-tensor product. For example, for any hermitian line subbundle $\overline M$ of $\overline E\otimes_{\varepsilon}\overline F$, we have the following variant of the inequality in Proposition  \ref{sudegtens}.
\begin{equation}\label{Equ:rg1}\ndeg(\overline M)\leqslant\hat{\mu}_{\max}(\overline E)+\hat{\mu}_{\max}(\overline F).\end{equation}
This formula can actually be considered as a slope inequality, concerning morphisms between hermitian vector bundles (see the proof of Proposition \ref{Pro:upp2}). Moreover it is stronger than the initial version in Proposition  \ref{sudegtens} (where  $\overline M$ would be replaced by the  hermitian line subbundle of $\overline E\otimes\overline F$ with the same underlying $\OK$-module $M$) since the operator (or $\varepsilon$-) norm is not larger than the Hilbert-Schmidt (hermitian) norm. A refinement of the inequality (\ref{Equ:rg1}) leads to a short proof of Theorem A.

As explained above, the inequality \eqref{Equ:rg1} implies a particular case of Theorem B where the rank of $V$ is $1$. It seems delicate to study the slope of hermitian vector subbundles of a tensor product by using the technic of $\varepsilon$-tensor products, since a vector subbundle of rank $\geqslant 2$ of an $\varepsilon$-tensor product is not necessarily hermitian. Note however that the ratio of the Hilbert-Schmidt (hermitian) norm and the operator ($\varepsilon$-) norm of a vector in a tensor product of two Hermitian spaces is controlled by the tensorial rank of the vector. Based on this observation, we introduce a version of the inequality \eqref{Equ:rg1} concerning again the  hermitian tensor product,   stronger than what would be \emph{a priori} expected from  the conjectural estimates in Problems \ref{mumaxconj}-\ref{HNtens}. Namely we show that, for any hermitian line subbundle $\overline M$ of $\overline E\otimes\overline F$, one has
\begin{equation}\label{Equ:lM}
\ndeg(\overline M)\leqslant\hat{\mu}_{\max}(\overline E)+\hat{\mu}_{\max}(\overline F)-\frac 12\log\rho(M),\end{equation}
where $\rho(M)$ is the tensorial rank of any non-zero vector in $M$. 

Actually in the ``extreme case'' where $\rho(M)=\rang(E)=\rang(F)$ --- this condition is easily seen to be equivalent to the semistability of $M_{\olK}$ under the action of $\mathrm{SL}(E_{\olK}) \times\mathrm{SL}(F_{\olK})$ in the sense of the geometric invariant theory --- we prove that 
\begin{equation*}
\ndeg(\overline M)\leqslant\hat{\mu}(\overline E)+\hat{\mu}(\overline F)-\frac 12\log\rho(M).\end{equation*}
This observation suggests that a geometric semistability condition could lead to stronger upper bounds for the slopes of hermitian vector subbundles in the tensor product of  hermitian vector bundles than the ones conjectured in Problem \ref{mustableconj}.

 As emphasized in Totaro \cite{Totaro96},  geometric semistability conditions (for general linear groups and their products) may be interpreted by means of auxiliary filtrations in a purely numerical way, without explicit mention of group actions. This interpretation becomes even  clearer from a probabilistic point of view. Given a non-zero finite dimensional vector space $W$ equipped with a decreasing $\mathbb R$-filtration $\mathcal F$, we define a random variable $Z_{\mathcal F}$ on the probability space $\{1,\ldots,\rang(W)\}$ (equipped with the equidistributed probability measure) as
\[Z_{\mathcal F}(i)=\sup\{t\in\mathbb R\,|\,\rang(\mathcal F^tW)\geqslant i\}.\]
The expectation of $\mathcal F$, denoted $\mathbb E[\mathcal F]$, is then defined as the expectation of the random variable $Z_{\mathcal F}$. If $E$ and $F$ are two vector spaces and if $V$ is a non-zero vector subspace of $E\otimes F$, we say that $V$ is \emph{both-sided semistable} if for any filtration $\mathcal F$ of $E$ and any filtration $\mathcal G$ of $F$, one has $\mathbb E[(\mathcal F\otimes\mathcal G)|_V]\leqslant \mathbb E[\mathcal F]+\mathbb E[\mathcal G]$, where $(\mathcal F\otimes\mathcal G)|_V$ is the filtration of $V$ such that \[(\mathcal F\otimes\mathcal G)|_V^t(V)=V\cap\sum_{a+b=t}\mathcal F^a(E)\otimes\mathcal G^b(F).\]
The conjectural slope inequality (\ref{equamaxbis}) implies that, if $\overline E$ and $\overline F$ are two hermitian vector bundles and if $\overline V$ is a non-zero hermitian vector subbundle of $\overline E\otimes\overline F$ such that $V$ is both-sided semistable, then:
\[\hat{\mu}(\overline V)\leqslant\hat{\mu}(\overline E)+\hat{\mu}(\overline F).\]
Actually a refinement of this implication where we only assume that the conjecture holds ``up to $\rang(V)$'' is established in Theorem \ref{Thm:semi}. The idea is similar to that in previous works such as \cite{Bost94,Zhang96,Gasbarri00}. However, the proof relies on Harder-Narasimhan filtrations indiced by $\mathbb R$ and is elementary, in the sense that it does not involve explicitly geometric invariant theory. By using Kempf's instability flag in this context, we show  that, in order to establish the inequality
\[\hat{\mu}(\overline V)\leqslant\hat{\mu}_{\max}(\overline E)+\hat{\mu}_{\max}(\overline F),\]
it suffices to prove the same inequality under the supplementary conditions that $\overline E$ and $\overline F$ are stable and that $V$ is both-sided stable (see Theorem \ref{Thm:recurrence}).

 In the case where $\rang(V)\leqslant 4$, the both-sided stability condition of $V$ implies some constraints on the successive tensorial ranks of vectors in $V$ and permits to conclude, by using inequalities similar to \eqref{Equ:lM}, together with a version of the ``absolute reduction theory'' for $V$, derived from Zhang's theory of arithmetic ampleness, that takes tensorial rank into account.

In view of the methods that we apply in this article (geometric
semistability argument and absolute Siegel's lemma), it is more
convenient to work with the absolute adelic viewpoint. For this
purpose, we introduce in the second section the notion of hermitian
vector bundles over $\overline{\mathbb Q}$. It is a natural
generalization of the classical notion of hermitian vector bundles
over the spectrum of an algebraic integer ring and more generally
the notion of hermitian adelic vector bundles (non-necessarily pure)
over a number field introduced by Harder and Stuhler \cite{HarderStuhler2003} and Gaudron
\cite{Gaudron07,Gaudron08}. Moreover, a hermitian vector bundle over
$\overline{\mathbb Q}$ can be approximated by a sequence of
hermitian vector bundles in usual sense (but over more and more
large number fields). Most of constructions and results of the slope
theory can be generalized to the framework of hermitian vector
bundles over $\overline{\mathbb Q}$.

The sections 3 and 4 are devoted to diverse upper bounds for the
Arakelov degree of hermitian line subbundle of a tensor product
bundle. We shall prove Theorem A in \S4. The fifth section contains
a reminder on $\mathbb R$-filtrations of vector spaces. In the sixth
section, we describe the geometric semistability by the language of
$\mathbb R$-filtrations. In the seventh section, we discuss the
relationship between the geometric and arithmetic semistability.
Finally, we prove Theorem B in the last section of the article.

\subsection*{Acknowledgement} We would like to thank \'Eric Gaudron and Ga\"el R\'emond for having communicated to us their article \cite{Gaudron_Remond} and for their careful reading and  valuable remarks on  this work. Part of this project has been effectuated during the visit of both authors to Mathematical Sciences Center of Tsinghua University. We are grateful to the center for hospitality.

\section{Hermitian vector bundles over $\overline{\mathbb Q}$}

In this article, $\overline{\mathbb Q}$ denotes the algebraic
closure of $\mathbb Q$ in $\mathbb C$. Any subfield of
$\overline{\mathbb Q}$ which is finite over $\mathbb Q$ is called a
\emph{number field}. If $K$ is a number field, we denote by
$\Sigma_{K,f}$ (resp. $\Sigma_{K,\infty}$) the set of finite (resp.
infinite) places of $K$, and define $\Sigma_K$ as the disjoint union
of $\Sigma_{K,f}$ and $\Sigma_{K,\infty}$. To each place
$v\in\Sigma_{K}$ we associate an absolute value $|.|_{v}$ which
extends either the usual absolute value of $\mathbb Q$ or some
$p$-adic absolute value $|.|_p$ normalized such that
$|a|_p:=p^{-v_p(a)}$ for any $a\in\mathbb Q^{\times}$, where $v_p$
is the $p$-adic valuation. We denote by $K_v$ the completion of $K$
with respect to the absolute value $|.|_v$ and by $\mathbb C_v$
the completion of an algebraic closure of $K_v$. The absolute value
$|.|_v$ extends in a unique way to $\mathbb C_v$.

\subsection{Norm families}\label{Subsec:norm families}
Let $K$ be a number field. If $E$ is a finite dimensional vector
space over $K$, denote by $\mathcal N_E$ the set of norm families of
the form $h=(h_v)_{v\in\Sigma_K}$, where for any $v$, $h_v$ (which
is also denoted by $\|.\|_{h_v}$ or by $\|.\|_v$) is a norm on
$E\otimes_K\mathbb C_v$, which is invariant by the action of
$\mathrm{Gal}(\mathbb C_v/K_v)$, and is an ultranorm if $v$ is
finite. Denote by $\mathcal H_E$ the subset of $\mathcal N_E$
consisting of those families $h$ such that $h_v$ is hermitian for
any infinite place $v$. The norm families in $\mathcal H_E$ are
called hermtian families. Note that, when $\rang(E)=1$, one has
$\mathcal H_E=\mathcal N_E$.

Let $\mathbf{e}=(e_1,\ldots,e_n)$ be a basis of $E$. It determines a
norm family $h^{\mathbf{e}}$ such that
\[\|\lambda_1e_1
+\cdots+\lambda_ne_n\|_{h^{\mathbf{e}}_v}=
\begin{cases}
\displaystyle\max_{1\leqslant i\leqslant n}|\lambda_i|_v,&\text{$v$ is finite},\\
(|\lambda_1|_v^2+\cdots+|\lambda_n|_v^2)^{1/2},&\text{$v$ is
infinite}.
\end{cases}\]
It is a hermitian family.

\subsubsection{Induced norm families} Let $E$, $F$ and $G$ be finite dimensional vector spaces over $K$.
Any injective $K$-linear map $i:F\rightarrow E$ defines a mapping
$i^*:\mathcal N_E\rightarrow\mathcal N_F$ by restricting norms. It
sends $\mathcal H_E$ to $\mathcal H_F$. Similarly, any surjective
$K$-linear map $p:E\rightarrow G$ defines a mapping $p_*:\mathcal
N_E\rightarrow\mathcal N_G$ by projection of norms. One also has
$p_*(\mathcal H_E)\subset\mathcal H_G$.

\subsubsection{Dual norm families}Let $E$ be a finite dimensional vector space over $K$. Any norm family $h$ on $E$ induces by duality a norm family $h^\vee$ on $E^\vee$, which is in $\mathcal H_{E^\vee}$ if $h\in\mathcal H_E$.

\subsubsection{Direct sum}
Let $E$ and $F$ be two finite dimensional vector spaces over $K$,
and $h^E$ (resp. $h^F$) be a norm family in $\mathcal H_E$ (resp.
$\mathcal H_F$). We define a norm family $h^{E\oplus F}\in\mathcal
H_{E\oplus F}$ as follows. If $v$ is a finite place and $(x,y)\in
E_{\mathbb C_v}\oplus F_{\mathbb C_v}$, one has
\[\|(x,y)\|_{h^{E\oplus F}_v}=\max(\|x\|_{h^E_v},\|y\|_{h^F_v});\]
if $v$ is a infinite place, then
\[\|(x,y)\|_{h^{E\oplus F}_v}=(\|x\|_{h^E_v}^2+\|y\|_{h^F_v}^2)^{1/2}.\]
The norm family $h^{E\oplus F}$ is called the \emph{direct sum} of
$h^E$ and $h^F$, and is denoted by $h^E\oplus h^F$.
\subsubsection{Tensor product}
Let $E$ and $F$ be two finite dimensional vector spaces over $K$.
There are many ``natural ways'' to define a tensor product map from
$\mathcal N_E\times\mathcal N_F$ to $\mathcal N_{E\otimes F}$, see
for example \cite{Gaudron07} Definition 2.10. Here we introduce two
constructions which will be used in the sequel.

The first construction is given by operator norms. We identify
$E\otimes F$ with the space of linear operators $\Hom_K(E^\vee, F)$.
Assume that $h^E$ and $h^F$ are two norm families in $\mathcal N_E$
and in $\mathcal N_F$ respectively, we define
$h^E\otimes_{\varepsilon}h^F$ as the norm family
$(h_v^E\otimes_{\varepsilon}h_v^F)_{v\in\Sigma_K}$, where
$h_v^E\otimes_\varepsilon h_v^F$ denotes the norm of linear
operators from $(E_{\mathbb C_v}^\vee,h_v^{E,\vee})$ to $(F_{\mathbb
C_v},h_v^F)$. Thus one obtains a map $\mathcal N_E\times\mathcal
N_F\rightarrow\mathcal N_{E\otimes F}$. It can be shown that the
operator $\otimes_{\varepsilon}$ is associative (up to isomorphism).
Namely, if $E$, $F$ and $G$ are three finite dimensional vector
spaces over $K$ and $h^E$, $h^F$ and $h^G$ are norm families in
$\mathcal N_E$, $\mathcal N_F$ and $\mathcal N_G$ respectively, then
$h^E\otimes_{\varepsilon}(h^F\otimes_{\varepsilon}h^G)$ identifies
with $(h^E\otimes_{\varepsilon}h^F)\otimes_{\varepsilon}h^G$ under
the canonical isomorphism $E\otimes(F\otimes G)\cong (E\otimes
F)\otimes G$. This permits to define the tensor product for several
norm families. To any finite family $(E_i)_{i=1}^n$ of finite
dimensional vector spaces over $K$ and any
$(h^{E_i})_{i=1}^n\in\prod_{i=1}^n\mathcal N_{E_i}$ we can associate
an element
$h^{E_1}\otimes_{\varepsilon}\cdots\otimes_{\varepsilon}h^{E_n}
\in\mathcal N_{E_1\otimes\cdots\otimes E_n}$. We remind that
$h^{E_1}\otimes_{\varepsilon}\cdots\otimes_{\varepsilon}h^{E_n}$
need not be in $\mathcal H_{E_1\otimes\cdots\otimes E_n}$, even
though each $h^{E_i}$ is in $\mathcal H_{E_i}$. The operator
$\otimes_{\varepsilon}$ is also commutative (up to isomorphism) in
the sense that, for all finite dimensional vector spaces $E$ and $F$
over $K$ equipped with norm families $h^E$ and $h^F$, the norm
families $h^E\otimes_{\varepsilon}h^F$ and
$h^F\otimes_{\varepsilon}h^E$ coincide under the canonical
isomorphism $E\otimes F\cong F\otimes E$.

The second construction is the Hermitian tensor product. Let
$(E_i)_{i=1}^n$ be a family of finite dimensional vector spaces over
$K$. For $(h^{E_i})_{i=1}^n\in\prod_{i=1}^n\mathcal H_{E_i}$, denote
by $h^{E_1}\otimes\cdots\otimes h^{E_n}$ the norm family
$(h_v^{E_1}\otimes\cdots\otimes h_v^{E_n})_{v\in\Sigma_K}$ such
that,
\begin{enumerate}[1)]
\item if $v$ is finite, then $h_v^{E_1}\otimes\cdots\otimes
h_v^{E_n}:=h_v^{E_1}\otimes_{\varepsilon}\cdots\otimes_{\varepsilon}
h_v^{E_n}$,
\item if $v$ is infinite, then $h_v^{E_1}\otimes\cdots\otimes
h_v^{E_n}$ is the usual hermitian product norm on $E_{1,\mathbb
C_v}\otimes\cdots\otimes E_{n,\mathbb C_v}$.
\end{enumerate}
Note that one has $h^{E_1}\otimes\cdots\otimes h^{E_n}\in\mathcal
H_{E_1\otimes\cdots\otimes E_n}$, and for any infinite place $v$,
the following relation holds
\begin{equation}\label{Equ:epsilon}
h_v^{E_1}\otimes_{\varepsilon}\cdots\otimes_{\varepsilon}
h_v^{E_n}\leqslant h_v^{E_1}\otimes\cdots\otimes h_v^{E_n}.
\end{equation}

\subsubsection{Determinant} Let $E$ be a finite dimensional vector space over $K$ and $h\in\mathcal H_E$. Recall that the \emph{determinant} of $E$ is the maximal exterior product $\det E:=\Lambda^{\rang(E)}E$. We denote by $\det h$ the family of determinant norms. Recall that for any finite place $v$, $\det h_v$ coincides with the quotient of the tensor power norm. However, when $v$ is infinite, $\det h_v$ is defined as the hermitian norm such that
\[\langle x_1\wedge\cdots\wedge x_r,y_1\wedge\cdots\wedge y_r\rangle_{\det h_v}=\det (\langle x_i,y_j\rangle_{h_v})_{1\leqslant i,j\leqslant r},\]
where $r$ is the rank of $E$.

\subsubsection{Scalar extension}
Let $E$ be a finite dimensional vector space over $K$ and
$h\in\mathcal N_E$. Assume that $K'/K$ is a finite extension of
number fields. We define $h_{K'}$ as the family
$(h_{K',w})_{w\in\Sigma_{K'}}$ such that, for any place
$w\in\Sigma_{K'}$ lying over some $v\in\Sigma_K$, one has
$h_{K',w}=h_v$ on $E\otimes_K\mathbb C_v=(E\otimes_KK')
\otimes_{K'}\mathbb C_{w}$. Note that $h_{K'}$ is an element in
$\mathcal N_{E\otimes_KK'}$. If $h\in\mathcal H_{E}$, then
$h_{K'}\in\mathcal H_{E\otimes_KK'}$.

\subsubsection{Distance of two norm families}
Let $E$ be a finite dimensional vector space over $K$ and $h$, $h'$
be two norm families in $\mathcal N_E$. For any place $v$ of $K$,
the \emph{local distance} between $h$ and $h'$ at $v$ is defined as
\[d_v(h,h'):=\sup_{0\neq s\in E_{\mathbb C_v}}\big|\log\|s\|_{h_v}-\log\|s\|_{h_v'}\big|\]
The global distance of $h$ and $h'$ is defined as
\[d(h,h'):=\sum_{v\in\Sigma_K}\frac{[K_v:\mathbb Q_v]}{[K:\mathbb Q]}d_v(h,h').\]
Observe that the global distance is invariant by scalar extension.

\subsubsection{Admissible norm families}
Let $E$ be a finite dimensional vector space over $K$. If
$\mathbf{e}$ and $\mathbf{f}$ are two bases of $E$, one has
\[h^{\mathbf{e}}=h^{\mathbf{f}}\quad\text{for all but finitely many $v$}.\]
In particular, one has $d(h^{\mathbf{e}},h^{\mathbf{f}})<+\infty$.
In fact, for any matrix $A\in\mathrm{GL}_n(k)$, there exists a
finite subset $S_A$ of $\Sigma_K$ such that
$A\in\mathrm{GL}_n(\mathcal O_v)$ for any $v\in\Sigma_K\setminus
S_A$, where $\mathcal O_v$ denotes the valuation ring of $\mathbb
C_v$, and $n$ is the rank of $E$.

Denote by $\mathcal N_E^{\circ}$ the subset of $\mathcal N_E$
consisting of norm families $h$ such that there exists a basis (or
equivalently, for any basis) $\mathbf{e}$ of $E$ such that the
equality $h_v=h_v^{\mathbf{e}}$ holds for all but finitely many
places $v$. Define $\mathcal H_E^{\circ}:=\mathcal H_E\cap\mathcal
N_E^{\circ}$. The norm families in $\mathcal N_E^{\circ}$ are called
\emph{admissible families}.

\subsection{Hermtian vector bundles over $\overline{\mathbb Q}$}\label{SubSec:hermitian}

Given a vector space $V$ of finite rank $r$ over $\overline{\mathbb
Q}$. For any number field $K$, we call \emph{model} of $V$ over $K$
any $K$-vector subspace $\mathcal V_K$ of rank $r$ of $V$ which
generates $V$ as a vector space over $\overline{\mathbb Q}$. The
number field $K$ is called the field of definition of the model
$\mathcal V_K$. Note that if $\mathcal V_1$ and $\mathcal V_2$ are
two models of $V$, there exists a number field $K$ which contains
the \emph{fields of definition} of $\mathcal V_1$ and $\mathcal V_2$
and such that the $K$-vector subspaces of $V$ generated by $\mathcal
V_1$ and $\mathcal V_2$ are the same.

Let $V$ be a vector space of finite rank over $\overline{\mathbb
Q}$. We call \emph{norm family} on $V$ any pair $(\mathcal V,h)$,
where $\mathcal V$ is a model of $V$ and $h\in\mathcal N_{\mathcal
V}^\circ$. Note that $(\mathcal V,h)$ is actually an adelic vector
bundle in the sense of \cite{Gaudron07} (non-necessarily pure, see also \cite{Gaudron08}). We say that a norm family
$(\mathcal V,h)$ on $V$ is \emph{hermitian} if $h\in\mathcal
H_{\mathcal V}^\circ$. We say that two \emph{norm families}
$(\mathcal V_1,h_1)$ and $(\mathcal V_2,h_2)$ are equivalent if
there exists a number field $K$ containing the fields of definition
of $\mathcal V_1$ and $\mathcal V_2$ and such that $\mathcal
V_{1,K}=\mathcal V_{2,K}$ and $h_{1,K}=h_{2,K}$. This is an
equivalence relation on the set of all norm families on $V$.

We call \emph{normed vector bundle} over $\overline{\mathbb Q}$ any
pair $\overline V=(V,\xi)$, where $V$ is a finite dimensional vector
space over $\overline{\mathbb Q}$ and $\xi$ is an equivalence class
of norm families on $V$. Any element $(\mathcal V,h)$ in the
equivalence class $\xi$ is called a \emph{model} of $\overline V$. We
say that a normed vector bundle $\overline V$ over
$\overline{\mathbb Q}$ is a hermitian vector bundle if $\xi$ is an
equivalence class of hermitian norm families. A normed vector bundle
of rank $1$ over $\overline{\mathbb Q}$ is necessarily hermitian. It
is also called a \emph{hermitian line bundle}.

\begin{rema}\label{Rem:approx}
We observe immediately from the definition that any adelic vector bundle $(\mathcal V,h)$ over some number field $K$ defines naturally a normed vector bundle $\overline V=(\mathcal V_{\overline{\mathbb Q}},\xi)$ where $\xi$ is the equivalence class of $(\mathcal V,h)$. In particular, any normed (resp. hermitian) vector bundle on $\Spec\mathcal O_K$ determines a normed (resp. hermitian) vector bundle over $\overline{\mathbb Q}$. However, a normed vector bundle over $\overline{\mathbb Q}$ need not come from a normed vector bundle in usual sense over the spectrum of an algebraic integer ring, unless it admits a pure model (see \cite{Gaudron08}). In general, a normed vector bundle $\overline V$ can be approximated by usual normed vecotor bundles in the sense that, for any $\epsilon>0$, there exists a number field $K$, a pure adelic vector bundle $(\mathcal V,h')$, and a model $(\mathcal V,h)$ of $\overline V$ over $K$ whose underlying $K$-vector space is $\mathcal V$ such that $d(h,h')<\epsilon$.
\end{rema}

Let $\overline V$ and $\overline V{}'$ be two normed vector bundles
over $\overline{\mathbb Q}$ and $f:V\rightarrow V'$ be an
isomorphism of $\overline{\mathbb Q}$-vector spaces. We say that $f$
is an \emph{isomorphism} of normed vector bundles if there exists a
number field $K$ and a model $(\mathcal V,h)$ (resp. $(\mathcal
V',h')$) of $\overline V$ (resp. $\overline V{}'$) whose field of
definition is $K$ and such that
\begin{enumerate}[(1)]
\item $f$ descends to an isomorphism $\widetilde f:\mathcal V\rightarrow\mathcal V'$ of $K$-vector spaces,
\item $\widetilde f$ induces for each place $v\in\Sigma_K$ an isometry between $(\mathcal V_{\mathbb C_v},{h_v})$ and $(\mathcal V'_{\mathbb C_v},{h_v'})$.
\end{enumerate}

The operations on norm families discussed in \S\ref{Subsec:norm
families} preserve admissible norm families and the above equivalence
relation. Hence they lead to the following constructions of normed
(hermitian) vector bundles.

\subsubsection{Subbundles and quotient bundles} Let $\overline V=(V,\xi)$ be a normed vector bundle over $\overline{\mathbb Q}$ and $W$ be a vector subspace of $V$. For any model $(\mathcal V, h)$ of $\overline V$, $\mathcal W=W\cap\mathcal V$ is a model of the $\overline{\mathbb Q}$-vector space $W$ (with the same field of definition as that of $\mathcal V$). The induced norm family on $\mathcal W$ defines a structure of normed vector bundle on $V$ which does not depend on the choice of the model $(\mathcal V,h)$. The corresponding normed vector bundle $\overline W$ is called a \emph{normed vector subbundle} of $\overline V$.

Similarly, the normed vector bundle structure $\xi$ induces by
quotient for each quotient space of $V$ a structure of normed vector
bundle on it. The normed vector bundle thus obtained is called a
\emph{normed quotient bundle} of $\overline V$. Note that, if
$\overline V$ is hermitian, then also are its normed subbundles and
quotient bundles.

Let $\overline V{}'$, $\overline V$ and $\overline V{}''$ be three
normed vector bundles over $\overline{\mathbb Q}$ and
\begin{equation}\label{Equ:suiteex}\xymatrix{\relax
0\ar[r]&V'\ar[r]^-f&V\ar[r]^-{g}&V''\ar[r]&0}\end{equation} be a
diagram of $\overline{\mathbb Q}$-linear maps. We say that
\[\xymatrix{\relax 0\ar[r]&\overline V{}'\ar[r]^-f&\overline V\ar[r]^-{g}&\overline V{}''\ar[r]&0}\]
is a \emph{short exact sequence} of normed vector spaces if the
diagram \eqref{Equ:suiteex} is a short exact sequence in the
category of $\overline{\mathbb Q}$ vector spaces and if $f$ (resp.
$g$) defines an isomorphism between $\overline V{}'$ (resp.
$\overline V{}''$) and a normed vector subbundle (resp. quotient
bundle) of $\overline V$.

\subsubsection{Dual bundles} Let $\overline V=(V,\xi)$ be a normed vector bundle over $\overline{\mathbb Q}$. Denote by $\overline V^\vee$ the pair $(V^\vee,\xi^\vee)$, where $\xi^\vee$ is the equivalence class of the dual of a model in $\xi$, called the \emph{dual} of $\overline V$. It is a normed vector bundle over $\overline{\mathbb Q}$. Moreover, it is hermitian if $\overline V$ is hermitian.

\subsubsection{Direct sum} Let $\overline V=(V,\xi)$ and $\overline W=(W,\eta)$ be two hermitian vector bundles over $\overline{\mathbb Q}$. Denote by $\overline V\oplus\overline W$ the pair $(V\oplus W,\xi\oplus\eta)$, where $\xi\oplus\eta$ is the equivalence class of the direct sum of two models (with the same field of definition) of $\overline V$ and $\overline W$ respectively. It is a hermtian vector bundle over $\overline{\mathbb Q}$. We call it the \emph{direct sum} of $\overline V$ and $\overline W$.

\subsubsection{Tensor products} Let $\overline V=(V,\xi)$ and $\overline W=(W,\eta)$ be two normed vector bundles over $\overline{\mathbb Q}$. Denote by $\overline V\otimes_{\varepsilon}\overline W$ the pair $(V\otimes W,\xi\otimes_{\varepsilon}\eta)$, where $\xi\otimes_{\varepsilon}\eta$ denotes the equivalence class of the $\varepsilon$-tensor product of two models (with the same field of definition) of $\overline V$ and $\overline W$ respectively. It is a normed vector bundle over $\overline{\mathbb Q}$ which does not depend on the choice of models. We call it the $\varepsilon$-\emph{tensor product} of $\overline V$ and $\overline W$.

If in addition $\overline V$ and $\overline W$ are hermitian, the
\emph{hermitian tensor product} $\overline V\otimes\overline W$ is
defined in a similar way. It is a hermitian vector bundle over
$\overline{\mathbb Q}$.
\subsubsection{Determinant} Let $\overline V$ be a hermitian vector bundle over $\overline{\mathbb Q}$ and $(\mathcal V,h)$ be a model of it. Denote by $\det\overline V$ the hermitian vector bundle over $\overline{\mathbb Q}$ whose hermitian vector bundle structure is the equivalence class of $(\det \mathcal V,\det h)$, called the \emph{determinant} of $\overline V$. It is a hermitian line bundle over $\overline{\mathbb Q}$.

\subsection{Arakelov degree}

Let $\overline L$ be a hermitian line bundle over $\overline{\mathbb
Q}$ and \[\overline{\mathcal L}=(\mathcal
L,(\|.\|_{v})_{v\in\Sigma_K})\] be a model of $\overline L$,
where $K$ is its field of definition. We define the (normalized)
\emph{Arakelov degree} of $\overline L$ as
\begin{equation}\label{Equ:degree}\widehat{\deg}_{\mathrm{n}}(\overline L):=-\sum_{v\in\Sigma_K}\frac{[K_v:\mathbb Q_v]}{[K:\mathbb Q]}\log\|s\|_v,\end{equation}
where $s$ is a non-zero element in $\mathcal L$. The sum in
\eqref{Equ:degree} does not depend on the choice of $s$ (by the
product formula), and does not depend on the choice of the model
(being normalized, it is stable under scalar extensions).

More generally, the (normalized) \emph{Arakelov degree} of a
\emph{hermitian} vector bundle $\overline E$ over $\overline{\mathbb
Q}$ is defined as
\[\widehat{\deg}_{\mathrm{n}}(\overline E):=\widehat{\deg}_{\mathrm{n}}(\det\overline E).\]
If $\xymatrix{0\ar[r]&\overline F\ar[r]&\overline E\ar[r]&\overline
G\ar[r]&0}$ is a short exact sequence of hermitian vector bundles
over $\overline{\mathbb Q}$, then one has
\[\widehat{\deg}_{\mathrm{n}}(\overline E)=\widehat{\deg}_{\mathrm{n}}(\overline F)+\widehat{\deg}_{\mathrm{n}}(\overline G).\]
Moreover, if $\overline E$ is a hermitian vector bundle over
$\overline{\mathbb Q}$ and if $F_1$ and $F_2$ are two vector
subspaces of $E$, then the following relations hold~:
\begin{gather}
\rang(F_1\cap F_2)+\rang(F_1+F_2)=\rang(F_1)+\rang(F_2),\\
\label{Equ:suraddi}\ndeg(\overline{F_1\cap
F_2})+\ndeg(\overline{F_1+F_2})\geqslant\ndeg(\overline
F_1)+\ndeg(\overline F_2).
\end{gather}
In fact, by approximation (see Remark \ref{Rem:approx}) we can reduce the problem to the case where all hermitian vector bundles admit pure models. Then it suffices to apply \cite{Gaudron07}, Propositions 4.22
and 4.23 to suitable models.

\subsection{Slopes}\label{Subsec:slopes}
If $\overline E$ is a non-zero hermitian vector bundle over
$\overline{\mathbb Q}$, the \emph{slope} of $\overline E$ is defined
as the quotient
\[\hat{\mu}(\overline E):=\frac{\widehat{\deg}_{\mathrm{n}}(\overline E)}{\rang(E)}.\]
We say that $\overline E$ is \emph{semistable} (resp. \emph{stable})
if for any non-zero vector subspace $F$ which is strictly contained
in $E$ one has $\hat{\mu}(\overline F)\leqslant
\hat{\mu}(\overline E)$ (resp. $\hat{\mu}(\overline
F)<\hat{\mu}(\overline E)$).

\begin{prop}\label{Pro:maximal}
Let $\overline E$ be a non-zero hermitian vector bundle over
$\overline{\mathbb Q}$. The set \[\{\hat{\mu}(\overline
F)\,:\,\text{$F$ is a non-zero vector subspace of $E$}\}\] attains
its maximum, and there exists a non-zero subbudle $\overline
E_{\des}$ of $\overline E$ whose slope is maximal and which contains
all subbundles of $\overline E$ with the maximal slope.
\end{prop}
\begin{proof}
We prove the proposition by induction on the rank $r$ of $E$. The
case where $r=1$ is trivial. In the following, we assume $r>1$.

If for any $F\subset E$, one has
$\hat{\mu}(F)\leqslant\hat{\mu}(E)$, then $E_{\des}=E$
already verifies the condition. Otherwise we can choose
$E'\subsetneq E$ such that $\hat{\mu}(\overline
E{}')>\hat{\mu}(\overline E)$ and that $\rang(E')$ is as large
as possible. The induction hypothesis applied on $\overline E{}'$
shows that there exists $E{}_{\des}'\subset E'$ which verifies the
properties predicted by the conclusion of the Proposition. One has
$\hat{\mu}(\overline
E{}_{\des}')\geqslant\hat{\mu}(\overline
E{}')>\hat{\mu}(\overline E)$. We shall verify that actually
$E_{\des}:=E_{\des}'$ has the required properties relatively to
$\overline E$. Let $F$ be a non-zero subspace of $E$. If $F$ is
contained in $E'$, then $\hat{\mu}(\overline
F)\leqslant\hat{\mu}(\overline E{}_{\des}')$. Otherwise one has
$\rang(F+E')>\rang(E')$ and hence
$\hat{\mu}(\overline{F+E'})\leqslant\hat{\mu}(\overline E)$.
By \eqref{Equ:suraddi} one obtains
\[\widehat{\deg}_{\mathrm{n}}(\overline{F\cap E'})+\widehat{\deg}_{\mathrm{n}}
(\overline{F+E'})\geqslant\widehat{\deg}(\overline
F)+\widehat{\deg}(\overline E{}').\] Moreover, one has
\[\widehat{\deg}_{\mathrm{n}}(\overline{F+E'})\leqslant
\rang(F+E')\hat{\mu}(\overline
E)<\rang(F+E')\hat{\mu}(\overline E{}'),\] and
\[\widehat{\deg}_{\mathrm{n}}(\overline{F\cap E'})\leqslant\rang(F\cap E')\hat{\mu}(\overline E{}_{\des}')\]
since $F\cap E'\subset E'$. Thus
\[\widehat{\deg}_{\mathrm{n}}(\overline F)<\big(\rang(E'+F)-\rang(E')\big)\hat{\mu}(\overline E{}')+\rang(F\cap E')\hat{\mu}(\overline E{}_{\des}')\leqslant
\rang(F)\hat{\mu}(\overline E{}_{\des}').\] So the inequality
$\hat{\mu}(\overline F)\leqslant\hat{\mu}(\overline
E{}'_{\des})$ always holds, and the inequality is strict when
$F\not\subset E'$. By the induction hypothesis, if $F\subset E$ is
such that $\hat{\mu}(\overline F)=\hat{\mu}(\overline
E{}_{\des}')$, then one has $F\subset E_{\des}'$. The propositon is
thus proved.
\end{proof}

\begin{defi}
Let $\overline E$ be a non-zero hermitian vector bundle over
$\overline{\mathbb Q}$. Define \[\hat{\mu}_{\max}(\overline
E):=\sup\{\hat{\mu}(\overline F)\,:\,\text{$F$ is a non-zero
vector subspace of $E$}\},\] called the \emph{maximal slope} of
$\overline E$. By Proposition \ref{Pro:maximal}, we obtain that
$\hat{\mu}_{\max}(\overline E)$ is finite, and is attained by a vector
subspace of $E$.
\end{defi}

The following is an upper bound of the maximal slope of a quotient
hermitian vector bundle.
\begin{prop}\label{Pro:mumaxquot}
Let $\overline E$ be a hermitian vector bundle over
$\overline{\mathbb Q}$ and $M$ be a non-zero vector subspace of $E$
such that $E/M\neq 0$. Let $r$ be the rank of $M$. One has
\begin{equation}\label{Equ:mumaxquot}\hat{\mu}_{\max}(\overline E/\overline M)
\leqslant (r+1)\hat{\mu}_{\max}(\overline
E)-r\hat{\mu}(\overline M).\end{equation}
\end{prop}
\begin{proof}
Let $F$ be a vector subspace of $E$ containing $M$ and such that
$F/M=(E/M)_{\des}$. One has
\[\ndeg(\overline F/\overline M)=(\rang(F)-r)\hat{\mu}(\overline F/\overline M)=(\rang(F)-r)\hat{\mu}_{\max}(\overline E/\overline M).\]
Moreover,
\[\ndeg(\overline F/\overline M)=\ndeg(\overline F)-\ndeg(\overline M)\leqslant\rang(F)\hat{\mu}_{\max}(\overline E)-\ndeg(\overline M).\]
Therefore
\[\hat{\mu}_{\max}(\overline E/\overline M)\leqslant
\frac{\rang(F)}{\rang(F)-r}\hat{\mu}_{\max}(\overline E)-
\frac{\ndeg(\overline
M)}{\rang(F)-r}=\frac{\rang(F)}{\rang(F)-r}\big(\hat{\mu}_{\max}(\overline
E)-\hat{\mu}(\overline M)\big)+\hat{\mu}(\overline M).\]
This implies \eqref{Equ:mumaxquot} since $\rang(F)/(\rang(F)-r)$ is
bounded from above by $r+1$.
\end{proof}

Let $\overline E$ be a non-zero hermitian vector bundle over
$\overline{\mathbb Q}$. The proposition \ref{Pro:maximal} implies
the existence of an increasing flag
\begin{equation}\label{Equ:HN}0=E_0\subsetneq E_1\subsetneq\ldots\subsetneq E_n=E.\end{equation}
such that each subquotient $\overline E_i/\overline E_{i-1}$ is
semistable and that the successive slopes form a strictly decreasing
sequence:
\begin{equation}\label{Equ:succslope}\hat{\mu}(\overline E_1/\overline E_0)>\ldots>\hat{\mu}(\overline E_n/E_{n-1}).\end{equation}
In fact, it suffice to construct $E_i$ in a recursive way such that
$E_i/E_{i-1}=(E/E_{i-1})_{\des}$. Similarly to the classical theory
of Harder-Narasimhan filtrations for vector bundles on a regular
projective curve, the flag \eqref{Equ:HN} (called the
\emph{Harder-Narasimhan flag} of $\overline E$) is characterized by
the above two properties (semistable subquotients and strictly
decreasing successive slopes). The last slope in
\eqref{Equ:succslope} equals the minimal value of slopes of non-zero
quotient hermitian vector bundles of $\overline E$. It is denoted by
$\hat{\mu}_{\min}(\overline E)$.

\subsection{Successive degrees}\label{SubSec:successive deg}

Let $\overline E$ be a non-zero \emph{normed} vector bundle over
$\overline{\mathbb Q}$. Inspired by \cite{Bost_Kunnemann}, we
defined the \emph{first degree} of $\overline E$ as
\[\degi{1}(\overline E):=\sup\{\widehat{\deg}_{\mathrm{n}}(\overline L)\,|\,\text{$L$ is a vector subspace of rank $1$ of $E$}
\}.\] This definition could be compared to
\cite[(3.20)]{Bost_Kunnemann}: if $(\mathcal E,h)$ is a model of
$\overline E$, then one has
\[\degi{1}(\overline E)=\sup_{K'/K}\udeg(\mathcal E_{K'},h_{K'}),\]
where $K'/K$ runs over all finite extensions of $K$. In particular, if $\overline E$ comes from a hermitian vector bundle on the spectrum of an algebraic integer ring, then $\degi{1}(\overline E)$ is noting but the stable upper arithmetic degree of the hermitian vector bundle introduced in \S\ref{Subsec:sudegree}.


Note that the first degree can be interpreted as an absolute minimum.
In fact, $\degi{1}(\overline E)$ is just the opposite of the infimum
of heights (with respect to the universal line bundle equipped with
Fubini-Study metrics) of algebraic points in $\mathbb P(E^\vee)$.
Inspired by this observation, we propose the following notion of
\emph{successive degrees} as follows. Let $\overline E$ be a
non-zero normed vector bundle of rank $r$ over $\overline{\mathbb
Q}$. For any integer $i\in\{1,\ldots,r\}$, let
\begin{equation}\label{Equ:succ deg}
\degi{i}(\overline E):=\inf_{Z}\big(\sup\{\ndeg(\overline
L)\,|\,L\in(\mathbb P(E^\vee)\setminus Z)(\overline{\mathbb
Q})\}\big),\end{equation} where $Z$ runs over all closed subscheme
of codimension $\geqslant r-i+1$ in $\mathbb P(E^\vee)$ (subscheme
of codimension $r$ of $\mathbb P(E^\vee)$ refers to the empty
scheme). Note that $\degi{i}(\overline E)$ is the opposite of the
$(r-i+1)^{\mathrm{th}}$ logarithmic minimum of $\mathbb P(E^\vee)$.
We have the following relations
\begin{equation}
\degi{1}(\overline E)\geqslant\ldots\geqslant \degi{r}(\overline E).
\end{equation}


We recall below a result of Zhang (in a particular case) which
compares the sum of successive minima and the height of an
arithmetic variety (cf. \cite{Zhang95} Theorem 5.2).

\begin{theo}[Zhang]\label{Thm:Zhang}
Let $\overline E$ be a non-zero hermitian vector bundle of rank $r$
over $\overline{\mathbb Q}$. Then the following inequalities hold
\begin{equation}\label{Equ:encadrement}
0\leqslant\ndeg(\overline E)-\frac 12\sum_{i=1}^r\degi{i}(\overline
E)\leqslant \frac 12r\ell(r),
\end{equation}
where
\begin{equation}\label{Equ:alpha}\ell(r)=\sum_{2\leqslant j\leqslant r}\frac 1j.\end{equation}
\end{theo}

\begin{rema}
One has
\[\ell(r)<\log(r)\]
for any $r\geqslant 2$, and
\[\log(r)-\ell(r)=(1-\gamma)+O(r^{-1})\qquad (r\rightarrow\infty),\]
where $\gamma$ is the constant of Euler. The following are some
values of the functions $\ell(r)$ and
$\lceil\exp(r\ell(r))\rceil$.

\begin{center}
\begin{tabular}{|c|c|c|c|}
\hline $r$&$2$&$3$&$4$\\
\hline
$\ell(r)$&1/2&$5/6$&$13/12$\\
\hline $\lceil\exp(r\ell(r))\rceil$&$3$&$13$&$77$\\\hline
\end{tabular}\end{center}
\end{rema}

As an application, one obtains the following absolute version of
``Siegel's lemma'' (cf. \cite{Gaudron07} 4.13 and 4.14)
\begin{coro}\label{Cor:siegel}
For any  hermitian vector bundle $\overline E$ of rank $r>0$ over
$\overline{\mathbb Q}$ and any positive real number $\epsilon$,
there exist hermitian line subbundles $\overline
L_1,\ldots,\overline L_r$ of $\overline E$ such that
$E=L_1+\cdots+L_r$ and that
\begin{equation}
0\leqslant\widehat{\deg}_{\mathrm{n}}(\overline
E)-\sum_{i=1}^r\widehat{\deg}_{\mathrm{n}}(\overline L_i)\leqslant
\frac12 r\ell(r)+\epsilon.
\end{equation}
\end{coro}

\begin{rema}
In the special case where  $r=2$ and $\overline E$ admits a pure model
over $\mathbb Z$, the work of de Shalit et Parzanchevski implies
(see \cite{deShalit_Parzan} \S1.3) that there exist hermitian line
subbundles $\overline L_1$ and $\overline L_2$ of $\overline E$ such
that
\[\hat{\mu}({\overline E})-\frac 12\big(\widehat{\deg}_{\mathrm{n}}(\overline L_1)+\widehat{\deg}_{\mathrm{n}}(\overline L_2)\big)\leqslant\frac 12\log\frac{2}{\sqrt{3}}.\]
{Their method relies on the classical reduction theory of Euclidean
lattices, which goes back to Lagrange and Gauss.}
\end{rema}

By definition the first degree of a hermitian vector bundle is
bounded from above by the maximal slope. The theorem of Zhang leads
to a reverse inequality with a supplementary term which is
$\frac 12\ell(r)$.

\begin{prop}\label{Pro:sudeg encad}
For any hermitian vector bundle $\overline E$ of positive rank $r$
over $\overline{\mathbb Q}$, one has
\begin{equation}
\degi{1}(\overline E)\leqslant\hat{\mu}_{\max}(\overline
E)\leqslant\degi{1}(\overline E)+\frac 12\ell(r).
\end{equation}
\end{prop}
\begin{proof}
As explained above, the first degree $\degi{1}(\overline E)$ is
bounded from above by $\hat{\mu}_{\max}(\overline E)$. Consider
the destabilizing vector subbundle $\overline E_{\des}$ of
$\overline E$ and let $m$ denote its rank. By definition, one has
$\degi{1}(\overline E_{\des})\leqslant\degi{1}(\overline E)$.
Moreover, Theorem \ref{Thm:Zhang} applied to $\overline E_{\des}$
gives
\[\ndeg(\overline E_{\des})\leqslant \sum_{i=1}^m
\degi{i}(\overline E_{\des})+\frac 12\ell(m)\leqslant m\degi{1}(\overline
E_{\des})+\frac 12m\ell(m).\] Therefore,
\[\hat{\mu}_{\max}(\overline E)=\hat{\mu}(\overline E_{\des})\leqslant \degi{1}(\overline E_{\des})+\frac 12\ell(m)
\leqslant\degi{1}(\overline E)+\frac 12\ell(r).\]
\end{proof}

\subsection{Height of a linear map}

Let $\overline E$ be $\overline F$ be hermitian vector bundles over
$\overline{\mathbb Q}$, and $f:E\rightarrow F$ be a
$\overline{\mathbb Q}$-linear map. There exist models
$\overline{\mathcal E}_K$ and $\overline{\mathcal F}_K$ over the
same number filed $K$ such that $f$ descends to a $K$-linear map
$f_K:\mathcal E_K\rightarrow\mathcal F_K$. We define the
\emph{height} of $f$ as
\[h(f):=\sum_{v\in\Sigma_K}\frac{[K_v:\mathbb Q_v]}{[K:\mathbb Q]}\log\|f_{K,v}\|,\]
where $\|f_{K,v}\|$ denotes the operator norm of the $\mathbb
C_v$-linear map $\mathcal E_K\otimes_K\mathbb C_v\rightarrow\mathcal
F_K\otimes_K\mathbb C_v$ induced by $f_K$. Note that this definition
does not depend on the choice of the models $\overline{\mathcal
E}_K$ and $\overline{\mathcal F}_K$ (such that the linear map $f$
descends). Moreover, if $f$ is an isomorphism of vector spaces over
$\overline{\mathbb Q}$, one has
\[\ndeg(\overline E)=\ndeg(\overline F)+h(\Lambda^rf),\]
where $r$ is the rank of $E$. Since $h(\Lambda^rf)\leqslant rh(f)$,
one obtains
\[\hat{\mu}(\overline E)\leqslant\hat{\mu}(\overline F)+h(f)\]
provided that $f$ is a non-zero isomorphism of vector spaces over
$\overline{\mathbb Q}$. Therefore, we obtain the following
inequalities which should be considered a reformulation of classical slope inequalities
(see for example \cite[\S4.1.4]{Bost2001}).
\begin{prop}\label{Pro:slopeinequality}
Let $\overline E$ and $\overline F$ be non-zero hermitian vector
bundles over $\overline{\mathbb Q}$, and $f:E\rightarrow F$ be a
$\overline{\mathbb Q}$-linear map.
\begin{enumerate}[1)]
\item If $f$ is injective, then $\hat{\mu}_{\max}(\overline
E)\leqslant\hat{\mu}_{\max}(\overline F)+h(f)$.
\item If $f$ is surjective, then $\hat{\mu}_{\min}(\overline E)\leqslant
\hat{\mu}_{\min}(\overline F)+h(f)$.
\item If $f$ is non-zero, then $\hat{\mu}_{\min}(\overline E)\leqslant\hat{\mu}_{\max}(\overline F)+
h(f)$.
\end{enumerate}
\end{prop}

\section{Line subbundles of an $\varepsilon$-tensor product}

In this section, we prove an upper bound (Proposition \ref{Pro:upp2}) for the Arakelov degree
of a line subbundle in the $\varepsilon$-tensor product of two hermitian
vector bundles. It
leads to non-trivial applications on the study of the maximal slope
of tensor products (see Theorem \ref{Thm:mumax}).

\begin{prop}\label{Pro:upp2}
Let $\overline E$ and $\overline F$ be two non-zero hermitian vector
bundles over $\overline{\mathbb Q}$. One has
\begin{equation}\label{EQu:upp2}
\degi{1}(\overline E\otimes_{\varepsilon}\overline F)\leqslant
\hat{\mu}_{\max}(\overline E)+\hat{\mu}_{\max}(\overline
F).\end{equation}
\end{prop}
\begin{proof}
Let $L$ be a one-dimensional subspace of $E\otimes F$ and $f\in
L\setminus\{0\}$, considered as a $\overline{\mathbb Q}$-linear map
from $E^\vee$ to $F$. Choose a number field $K$ such that $f$ give
rise to a $K$-linear map $\widetilde f:\mathcal
E^\vee_K\rightarrow\mathcal F_K$, where $\overline{\mathcal E}_K$
and $\overline{\mathcal F}_K$ are respectively  models of $\overline
E$ and $\overline F$. One has
\[\widehat{\deg}_{\mathrm{n}}(\overline L)=-
\sum_{v\in\Sigma_{K}}\frac{[K_v:\mathbb Q_v]}{[K:\mathbb
Q]}\log\|\widetilde f_v\|_v=-h(f)\] Since $f$ is non-zero, by the
slope inequality (Proposition \ref{Pro:slopeinequality}), one has
\[
\hat{\mu}_{\min}(\overline E^\vee)\leqslant
\hat{\mu}_{\max}(\overline{F})+h(f)=\hat{\mu}_{\max}(\overline{F})-\widehat{\deg}_{\mathrm{n}}(\overline
L).
\]
Hence $\widehat{\deg}_{\mathrm{n}}(\overline
L)\leqslant\hat{\mu}_{\max}(\overline
F)-\hat{\mu}_{\min}(\overline
E^\vee)=\hat{\mu}_{\max}(\overline
E)+\hat{\mu}_{\max}(\overline F)$.
\end{proof}

\begin{defi}
Let $\overline E$ be a non-zero normed vector bundle of rank $r$
over $\overline{\mathbb Q}$. We define
\[\varsigma(\overline E):=\inf_{F\subsetneq E}\degi{1}(\overline E/\overline F),\]
where $F$ runs over all vector subspaces of $E$ such that
$F\subsetneq E$.
\end{defi}

Assume that $\overline E$ is hermitian. By definition, one always
has $\varsigma(\overline E)\leqslant\hat{\mu}_{\min}(\overline
E)$. Moreover Proposition \ref{Pro:sudeg encad} implies that
\[\varsigma(\overline E)\geqslant\hat{\mu}_{\min}(\overline E)-\frac 12\ell(r).\]
By passing to dual, one obtains
\[\hat{\mu}_{\max}(\overline E)\leqslant -\varsigma(\overline E^\vee)\leqslant\hat{\mu}_{\max}(\overline E)+\frac12\ell(r).\]

The following is a variant of the inequality \eqref{EQu:upp2} where
in the upper bound there appears the first degree of $\overline E$
rather than the maximal slope, and we need not assume that the
normed vector bundle $\overline E$ is hermitian. As a price paid,
the term $\hat{\mu}_{\max}(\overline F)$ figuring in the upper
bound is replaced by a larger term $-\varsigma(\overline F^\vee)$.
This permits to obtain an upper bound for the $\varepsilon$-tensor
product of \emph{several} hermitian vector bundles in a recursive
way (Corollary \ref{Cor:tensor}) and hence to establish a stronger
upper bound (see Theorem \ref{Thm:mumax}) for the maximal slope of
tensor product of hermitian vector bundles.
\begin{prop}\label{Pro:majr}
Let $\overline E$ and $\overline F$ be non-zero normed vector
bundles over $\overline{\mathbb Q}$. We have
\begin{equation}
\degi{1}(\overline E\otimes_{\varepsilon}\overline
F)\leqslant\degi{1}(\overline E)-\varsigma(\overline F^\vee).
\end{equation}
\end{prop}
\begin{proof}
Let $L$ be a one-dimensional subspace of $E\otimes F$, considered as
a subspace of $\Hom(F^\vee,E)$. Pick a non-zero element $f$ in $L$.
If $M$ is a subspace of rank one of $F^\vee/\mathrm{Ker}(f)$, one
has
\[\begin{split}\widehat{\deg}_{\mathrm{n}}(\overline M)&\leqslant\ndeg(\overline{f(M)})-\ndeg(\overline L)\\
&\leqslant\degi{1}(\overline E)-\ndeg(\overline L),\end{split}\]
where the first inequality comes from the slope inequality, in
considering $\ndeg(\overline L)$ as the height of
$f:F^\vee\rightarrow E$ (cf. \cite[Proposition 4.5]{Bost2001}).
Since $M$ is arbitrary, one obtains
\[\ndeg(\overline L)\leqslant\degi{1}(\overline E)-\degi{1}(\overline{F^\vee/\mathrm{Ker}(f)})
\leqslant\degi{1}(\overline E)-\varsigma(\overline F^\vee).\]
\end{proof}

\begin{rema}
Let $k$ be a field of characteristic $0$, and $K$ be the algebraic
closure of $k(t)$ (the field of rational functions of one variable
with coefficients in $k$). Similarly to \S\ref{SubSec:hermitian}, we
can introduce the notion of (ultra)normed vector bundles over $K$ as
an equivalence class of adelic vector bundle (in the sense of
Gaudron) over a function field defined over $k$; such objects have been studied by Hoffmann, Jahnel, and Stuhler \cite{HoffmannJahnelStuhler98}. It can be show
that, for any non-zero normed vector bundle $\overline E$ over $K$,
one has $\degi{1}(\overline E)=\hat{\mu}_{\max}(\overline E)$ (see Proposition \ref{mumsudegg}),
and consequently $\varsigma(\overline
E)=\hat{\mu}_{\min}(\overline E)$. Moreover, the analogue of
Proposition \ref{Pro:majr} also holds in this setting. Therefore one
obtains (compare to Theorem \ref{Thm:NS'} and Lemma \ref{Lemm:udegmax}) \[\hat{\mu}_{\max}(\overline E\otimes\overline F)
\leqslant\hat{\mu}_{\max}(\overline
E)+\hat{\mu}_{\max}(\overline F)\] (note that in function field
case the $\varepsilon$-tensor product is just the usual tensor
product). However, in number field case, the inequality
$\degi{1}(\overline E)\leqslant\hat{\mu}_{\max}(\overline E)$
could be strict. The $A_2$ lattice provides such a
counter-example (see Proposition \ref{propA2}).

The above discussion also shows that, even if the normed vector
bundles $\overline E$ and $\overline F$ (over $\overline{\mathbb
Q}$) are hermitian, one should not expect a result of the form
\[\degi{1}(\overline E\otimes_{\varepsilon}\overline F)
\leqslant\degi{1}(\overline E)+\hat{\mu}_{\max}(\overline F).\]
In fact, consider a hermitian vector bundle $\overline E$ which is
semistable and such that $\degi{1}(\overline
E)<\hat{\mu}_{\max}(\overline E)$ (as $A_2$ lattice for
example). One has $\degi{1}(\overline
E\otimes_{\varepsilon}\overline E^\vee)\geqslant 0$ since the line
in $E\otimes E^\vee$ generated by the trace element has Arakelov
degree $0$. However, \[\degi{1}(\overline
E)+\hat{\mu}_{\max}(\overline
E^\vee)<\hat{\mu}_{\max}(\overline
E)+\hat{\mu}_{\max}(\overline E^\vee)=0.\]
\end{rema}

The above theorem leads to the following corollary by induction.
\begin{coro}\label{Cor:tensor}
For any integer $N\geqslant 2$ and  $N$ non-zero hermitian vector
bundles $\overline E_1,\ldots,\overline E_N$over $\overline{\mathbb
Q}$, the following inequality holds:
\begin{equation}
\degi{1}(\overline
E_1\otimes_{\varepsilon}\cdots\otimes_{\varepsilon} \overline
E_N)\leqslant\hat{\mu}_{\max}(\overline
E_1)+\hat{\mu}_{\max}(\overline E_2)-\sum_{3\leqslant i\leqslant
N}\varsigma(\overline E_i^\vee).
\end{equation}
\end{coro}

\begin{theo}\label{Thm:mumax}
For non-zero hermitian vector bundles $\overline E$ and $\overline
F$ over $\overline{\mathbb Q}$, one has
\[\hat{\mu}_{\max}(\overline E\otimes\overline F)\leqslant\hat{\mu}_{\max}(\overline E)-\varsigma(\overline F^\vee)
\]
In particular,
\begin{equation}\label{Equ:maxleqsom}\hat{\mu}_{\max}(\overline E\otimes\overline F)\leqslant\hat{\mu}_{\max}(\overline E)+\hat{\mu}_{\max}(\overline F)+\frac12\ell\rang(F)).\end{equation}
\end{theo}
\begin{proof}
Let $V=(E\otimes F)_{\des}$. Consider the one dimensional vector
subspace $L$ of $V^{\vee}\otimes(E\otimes F)\cong \Hom(V,E\otimes
F)$ generated by the inclusion map of $V$ in $E\otimes F$. As a
normed subbundle of $\overline
V^{\vee}\otimes_{\varepsilon}(E\otimes F)$, one has $\ndeg(\overline
L)\geqslant 0$ since $L$ is generated by the inclusion map.
Moreover, one has
\[\ndeg(\overline L)\leqslant\degi{1}(\overline V^\vee\otimes_{\varepsilon}(\overline E\otimes\overline F))\leqslant\degi{1}(\overline V^\vee\otimes_{\varepsilon}\overline E\otimes_{\varepsilon}\overline F)\]
since the ${\varepsilon}$-tensor product norm is bounded from above
by the hermitian tensor product norm. By Corollary \ref{Cor:tensor},
one has
\[0\leqslant\ndeg(\overline L)\leqslant\hat{\mu}_{\max}(\overline V^\vee)+\hat{\mu}_{\max}(\overline E)-\varsigma(\overline F^\vee).\]
Since $\overline V$ is semi-stable, $\hat{\mu}_{\max}(\overline
V^\vee)=\hat{\mu}(\overline V^\vee)=-\hat{\mu}(\overline
V)=-\hat{\mu}_{\max}(\overline V)$. The proposition is thus
proved.
\end{proof}

\begin{rema} By the symmetry between $\overline E$ and $\overline F$, the inequality \eqref{Equ:maxleqsom} implies Theorem A. Moreover,
one obtains by induction from the previous theorem that, if
$(\overline E_i)_{i=1}^N$ is a finite family of non-zero hermitian
vector bundles over $\overline{\mathbb Q}$, then one has
\[\hat{\mu}_{\max}(\overline E_1\otimes\cdots
\otimes\overline E_N)\leqslant\hat{\mu}_{\max}(\overline
E_1)+\cdots+\hat{\mu}_{\max}(\overline
E_N)+\sum_{i=2}^N\frac12\ell\rang(E_i)).\]
\end{rema}

\section{Line subbundles of a hermitian tensor product}

In this section, we study the Arakelov degree of a line subbundle in the hermtian tensor product of two hermitian vector bundles over $\overline{\mathbb Q}$. Note that the hermitian tensor product metric is usually grater than the $\varepsilon$-tensor product (see \eqref{Equ:epsilon}). Therefore the upper bound \eqref{EQu:upp2} leads to a similar one for the hermitian tensor product case. As we shall show below, the upper bound obtained in such way can be refined where the tensorial rank of the line
subbundle appears (see Proposition \ref{Pro:majoration}). This method can be generalized to obtain an upper bound for the Arakelov degree of a vector subspace in the hermitian tensor product of two hermitian vector bundles, where the successive tensorial ranks of the vector subspace appear.

\subsection{Successive tensorial ranks of a line subbundle}

Let $K$ be a field and $E$ and $F$ be two vector spaces of finite
rank over $K$. We say that a vector in $E\otimes F$ is \emph{split}
if it can be written as the tensor product of a vector in $E$ and a
vector in $F$. For a non-zero vector $s\in E\otimes F$, the
\emph{tensorial rank} of $s$ is defined as the smallest integer $n\geqslant
1$ such that $s$ can be written as the sum of $n$ split vectors. The
tensorial rank of a non-zero vector $s\in E\otimes F$ is denoted by
$\rho(s)$. Note that the function $\rho(\cdot)$ is invariant by
dilations. Namely, for any non-zero element $a\in K$, one has
$\rho(as)=\rho(s)$. If $M$ is a vector subspace of rank one of
$E\otimes F$, we denote by $\rho(M)$ the tensorial rank of an arbitrary
non-zero element in $M$, called the \emph{tensorial rank} of $M$.

Let $s$ be a non-zero vector in $E\otimes F$. The tensorial rank of $s$ is
equal to the rank of $s$ considered as a $K$-linear map from
$E^\vee$ to $F$. If $M$ is a one-dimensional vector subspace of
$E\otimes F$, the tensorial rank of $M$ coincides with the rank of the image
of $M$ in $E$ (namely the smallest vector subspace $E_1$ of $E$ such
that $M\subset E_1\otimes F$), and also the rank of the image of $M$
in $F$.

The tensorial rank function is a geometric invariant which measures the
algebraic complexity of lines in a tensor product. For general
subspaces, we propose the following notion of successive tensorial ranks for
vector spaces over $\overline{\mathbb Q}$ (similar definition also
makes sense for vector spaces over a general algebraically closed
field, but we only need the restricted case in this article).

Let $E$ and $F$ be two vector spaces of finite rank over
$\overline{\mathbb Q}$, and $V$ be a non-zero subspace of rank $r$
of $E\otimes F$. For each integer $i\in\{1,\ldots,r\}$, let
\begin{equation}\rho_{i}(V):=\sup_{Z}\big(\inf\{\rho(M)\,|\,
M\in(\mathbb P(V^\vee)\setminus Z)(\overline{\mathbb Q})\}\big),
\end{equation}
where $Z$ runs over all closed subvarieties of codimension\footnote{By convention, in the case where $i=1$, the
condition $\mathrm{codim}(Z)=r$ means that the scheme $Z$ is empty.} $r-i+1$
in $\mathbb P(E^\vee)$. The integers $(\rho_{i}(V))_{i=1}^r$ are called the successive tensorial ranks of $V$.

The successive tensorial ranks can also be interpreted via the dimensions of the intersections of $V$ with the determinantal subvarieties of $E\otimes F$ (considered as an affine variety defined over $\overline{\mathbb Q}$). Recall that the $k$-th determinantal subvariety of $E\otimes F$ is the closed subvariety $D_k$ of $E\otimes F$ classifying the all tensor vectors which can be written as the sum of $k$ split vectors. Clearly one has
\[\{0\}=D_0\subset D_1\subset\ldots\subset D_k\subset D_{k+1}\subset\ldots,\]
and for any $k\geqslant\min(\rang(E),\rang(F))$, one has $D_k=E\otimes F$. With this notation, for any vector subspace $V$ of rank $r\geqslant 1$ of $E\otimes F$, one has
\[\rho_i(V)=\min\{k\,|\,\dim(V\cap D_k)\geqslant i\},\quad
\forall\,i\in\{1,\ldots,r\}.\]

\begin{rema}\label{Rem:ssl}
Let $\overline E$ and $\overline F$ be hermitian vector bundles over
$\overline{\mathbb Q}$, $\overline V$ be a hermitian vector
subbundle of rank $r\geqslant 1$ of $\overline E\otimes\overline F$
and $\epsilon$ be a positive real number. We can choose in a
recursive way one dimensional subspaces $L_1,\ldots,L_r$ in $V$
which are linearly independent and such that $\ndeg(\overline
L_i)\geqslant\degi{r}(\overline V)-\epsilon/r$ and
$\rho(L_i)\geqslant\rho_{i}(V)$. Note that by Theorem
\ref{Thm:Zhang}, one has
\[\ndeg(\overline E)\leqslant\sum_{i=1}^r\ndeg(\overline L_i)+\frac 12r\ell(r)+\epsilon.\]
This construction will be useful further in \S\ref{Sec:small rank}
for the study of upper bounds for the Arakelov degree of a vector
subbundle in the tensor product of two hermitian vector bundles.
\end{rema}

\subsection{An upper bound for the Arakelov degrees of line subbudles}

We begin by an upper bound for the Arakelov degree of an arbitrary
line subbundle of the tensor product of two hermitian vector bundles
which can be considered as a reformulation of Hadamard's inequality.
{This upper bound is a variant of \cite{Bost_Kunnemann} Proposition
3.4.1.}

\begin{prop}\label{Pro:majoration}
Let $\overline E$ and $\overline F$ be two hermitian vector bundles
over $\overline{\mathbb Q}$, and $M$ be a non-zero vector subspace
of rank $1$ of $E\otimes F$. Let $E_1$ and $F_1$ be the images of
$M$ in $E$ and in $F$ respectively. One has
\begin{equation}\label{Equ:sudeg1}\widehat{\deg}_{\mathrm{n}}(\overline M)
\leqslant\hat{\mu}(\overline E_1)+\hat{\mu}(\overline
F_1)-\frac 12\log\rho(M).\end{equation} In particular, one has
\begin{equation}\label{Equ:sudeg2}\widehat{\deg}_{\mathrm{n}}(\overline M)
\leqslant\hat{\mu}_{\max}(\overline
E)+\hat{\mu}_{\max}(\overline F)-\frac
12\log\rho(M).\end{equation}
\end{prop}
\begin{proof}
The inequality \eqref{Equ:sudeg2} is a direct consequence of
\eqref{Equ:sudeg1}. In the following, we prove the first one. By
definition, $E_1$ and $F_1$ are respectively vector subspaces of
rank $\rho(M)$ of $E$ and $F$. Consider the $K$-linear maps
\begin{equation}\label{Equ:compos}M^{\otimes r}\longrightarrow E_1^{\otimes r}\otimes F_1^{\otimes r}\longrightarrow \Lambda^r(E_1)\otimes\Lambda^r(F_1),\end{equation}
where $r=\rho(M)$. Denote by $\varphi$ the composed map. Let
$(e_i)_{i=1}^r$ and $(f_i)_{i=1}^r$ be respectively a basis of $E_1$
and $F_1$ such that
\[\alpha=e_1\otimes f_1+\cdots+e_r\otimes f_r\]
is a non-zero element in $M$. The image of $\alpha^{\otimes r}$ by
the composed map \eqref{Equ:compos} is just
\begin{equation}\label{Equ:varphi}\varphi(\alpha^{\otimes r})=r!(e_1\wedge\cdots\wedge e_r)\otimes(f_1\wedge\cdots\wedge f_r).\end{equation}
Let $K$ be a number field such that $(e_i)_{i=1}^r$ (resp.
$(f_i)_{i=1}^r$) gives rise to a basis of a model $\mathcal E_{1,K}$
(resp. $\mathcal F_{1,K}$) of $E_1$ (resp. $F_1$) over $K$. Thus the
homomorphism $\varphi$ gives rise to a $K$-linear map
$\widetilde{\varphi}$ from $K\alpha^{\otimes r}$ to
$\Lambda^r(\mathcal E_{1,K})\otimes\Lambda^r(\mathcal F_{1,K})$.

By \eqref{Equ:varphi}, for any finite place $\mathfrak p$ of $K$,
one has
\[\|\widetilde{\varphi}\|_{\mathfrak p}\leqslant|r!|_{\mathfrak p}.\]
In fact, if $\|\alpha\|_{\mathfrak p}<1$, then we can choose $(e_i)_{i=1}^r$ and $(f_i)_{i=1}^r$ such that $\displaystyle\max_{1\leqslant i\leqslant r}\|e_i\|_{\mathfrak p}\leqslant 1$ and $\displaystyle\max_{1\leqslant i\leqslant r}\|f_i\|_{\mathfrak p}\leqslant 1$ since $\|\alpha\|_{\mathfrak p}$ is the operator norm of $\alpha_{\mathbb C_{\mathfrak p}}:E_{1,\mathbb C_{\mathfrak p}}^\vee\rightarrow F_{1,\mathbb C_p}$.
Moreover, for any infinite place $\sigma$ of $K$, one has
\[\|\alpha\|_\sigma^{2r}=\bigg(\sum_{1\leqslant i,j\leqslant r}\langle e_i,e_j\rangle_\sigma\langle f_i,f_j\rangle_\sigma\bigg)^r\geqslant
r^r\|e_1\wedge\cdots\wedge
e_r\|_\sigma^2\cdot\|f_1\wedge\cdots\wedge f_r\|_\sigma^2.\] Hence
\[\|\widetilde{\varphi}\|_{\sigma}=|r!|_\sigma\cdot r^{-r}.\]
By the slope inequality (cf. Proposition \ref{Pro:slopeinequality})
and the product formula,
\[r\,\widehat{\deg}_{\mathrm{n}}(\overline M)
\leqslant\widehat{\deg}_{\mathrm{n}}(\overline E_1)+
\widehat{\deg}_{\mathrm{n}}(\overline F_1)-\frac{r}{2}\log(r).\]
Therefore,
\[\widehat{\deg}_{\mathrm{n}}(\overline M)\leqslant\hat{\mu}(\overline E_1)+\hat{\mu}(\overline F_1)-\frac 12\log\rho(M).\]
\end{proof}

We obtain from \eqref{Equ:sudeg2} that, if $V$ is a non-zero vector
subspace of $E\otimes F$, then one has
\begin{equation}\label{Equ:upp}\degi{1}(\overline V)+\frac 12\log\rho_{1}(\overline V)\leqslant\hat{\mu}_{\max}(\overline E)+\hat{\mu}_{\max}(\overline F).\end{equation}
By Theorem \ref{Thm:Zhang}, we obtain the following variant of
\eqref{Equ:upp}.

\begin{prop}\label{Pro:majo de mu}
Let $\overline E$ and $\overline F$ be two hermitian vector bundle
over $K$, and $V$ be a non-zero vector subspace of $E\otimes F$. One
has
\begin{equation}\label{Equ:majoV}
\hat{\mu}(\overline V)\leqslant\hat{\mu}_{\max}(\overline
E)+\hat{\mu}_{\max}(\overline
F)+\frac 12\ell(r)-\frac{1}{2r}\sum_{i=1}^r\log\rho_{i}(V),
\end{equation}
where $r$ is the rank of $V$. In particular, the inequality
\[\hat{\mu}(\overline V)\leqslant\hat{\mu}_{\max}(\overline E)+\hat{\mu}_{\max}(\overline F)\]
holds as soon as
\[\prod_{i=1}^r\rho_i(V)\geqslant \exp\big(r\ell(r)\big).\]
\end{prop}
\begin{proof}
By the previous proposition, one obtains that
\[\degi{i}(\overline V)\leqslant\hat{\mu}_{\max}(\overline E)+\hat{\mu}_{\max}(\overline F)-\frac 12\log\rho_{i}(V).\]
Therefore Theorem \ref{Thm:Zhang} implies that
\[\ndeg(\overline V)\leqslant\sum_{i=1}^r\degi{i}(\overline V)+\frac 12r\ell(r)\leqslant r(\hat{\mu}_{\max}(\overline E)+\hat{\mu}_{\max}(\overline F))+\frac 12r\ell(r)-\frac 12\sum_{i=1}\log\rho_{i}(V)\]
which gives \eqref{Equ:majoV}.
\end{proof}

\section{Filtrations}\label{Filtr}

In this section, the expression $K$ denotes an arbitrary field. Let
$W$ be a vector space of finite rank over $K$. We call
\emph{filtration} of $W$ any decreasing family $(\mathcal
F^tW)_{t\in\mathbb R}$ of vector subspaces of $W$.  We assume in convention that any filtration $\mathcal F$ is
separated ($\mathcal F^tW=0$ for $t$ sufficiently positive),
exhaustive ($\mathcal F^tW=W$ for $t$ sufficiently negative), and
left continuous (the function $t\mapsto\rang(\mathcal F^tW)$ is
locally constant on left). The set of all filtrations of $W$ is
denoted by $\mathbf{Fil}(W)$. If $\mathcal F$ is a filtration of
$W$, we denote by $\lambda_{\mathcal F}:W\rightarrow\mathbb
R\cup\{+\infty\}$ the map which sends $x\in W$ to $\sup\{t\in\mathbb
R\,|\,x\in\mathcal F^tW\}$. This function takes finite values on
$W\setminus\{0\}$.

\subsection{Filtration as a weighted flag}\label{SubSec:weighted} Let $W$ be a finite dimensional vector space over ${K}$.
A filtration $\mathcal F$ of $W$ can be considered as an increasing
flag of $W$~:
\begin{equation}\label{Equ:flag}0=W_0\subsetneq W_1\subsetneq\ldots\subsetneq W_n=W\end{equation}
together with a strictly decreasing sequence of real numbers
\[a_1>\ldots>a_n\]
which describes the jumps of the filtration. One has $\mathcal
F^tW=W$ if $t\leqslant a_n$, $\mathcal F^tW=0$ if $t>a_1$; and if
$i\in\{1,\ldots,n-1\}$ and $t\in\,]a_{i+1},a_{i}]$, one  has
$\mathcal F^tW=W_i$. We say that a basis $\mathbf{e}$ of $W$ is
\emph{compatible} with the filtration $\mathcal F$ if for any
$t\in\mathbb R$ one has
\[\rang(\mathcal F^tW)=\#(\mathbf{e}\cap\mathcal F^tW).\]
Note that this definition only depends on the flag associated to the
filtration. In fact, the basis $\mathbf{e}$ is compatible with the
filtration $\mathcal F$ if and only if it is \emph{compatible} with
the flag \eqref{Equ:flag}, namely for any $i\in\{1,\ldots, n\}$ one
has $\#(\mathbf{e}\cap W_i)=\rang(W_i)$. Moreover, by  Bruhat
decomposition for general linear groups, we obtain that, for all
filtrations $\mathcal F_1$ and $\mathcal F_2$ of $W$, there always
exists a basis of $W$ which is simultaneously compatible with
$\mathcal F_1$ and $\mathcal F_2$.

Given a basis $\mathbf{e}=(e_j)_{j=1}^r$ of a finite dimensional
vector space $W$ over ${K}$ and a map
$\varphi:\mathbf{e}\rightarrow\mathbb R$, one can construct a
filtration $\mathcal F_{\mathbf{e},\varphi}$ of $W$ as follows. For
any $t\in\mathbb R$, $\mathcal F_{\mathbf{e},\varphi}^t(W)$ is taken
as the vector subspace of $W$ generated by the elements
$e_j\in\mathbf{e}$ such that $\varphi(e_j)\geqslant t$. The basis
$\mathbf{e}$ is compatible with the filtration $\mathcal
F_{\mathbf{e},\varphi}$, and the restriction of $\lambda_{\mathcal
F_{\mathbf{e},\varphi}}$ on $\mathbf{e}$ coincides with $\varphi$.
The filtration $\mathcal F_{\mathbf{e},\varphi}$ is said to be
\emph{associated} to the basis $\mathbf{e}$ and the function
$\varphi$.

Conversely, given a finite dimensional vector space $W$ over $K$
equipped with a filtration $\mathcal F$, and a basis
$\mathbf{e}=(e_j)_{j=1}^r$ of $W$ which is compatible with $\mathcal
F$. If we denote by $\varphi$ the restriction of $\lambda_{\mathcal
F}$ on $\mathbf{e}$, then the filtration associated to the basis
$\mathbf{e}$ and the function $\varphi$ coincides with $\mathcal F$.
In particular, for any element $x=a_1e_1+ \cdots+a_re_r\in W$, one
has
\begin{equation}\label{Equ:lambdaFx}\lambda_{\mathcal F}(x)=\min\{\lambda_{\mathcal F}(e_j)
\,|\,1\leqslant j\leqslant r,\,a_j\neq 0\}.\end{equation}

\subsection{Filtration as a norm} Let $W$ be a finite dimensional vector space over $K$. If $\mathcal F$ is a filtration of $W$, one has
\[\lambda_{\mathcal F}(x+y)\geqslant\min(\lambda_{\mathcal F}(x),\lambda_{\mathcal F}(y)),\qquad\lambda_{\mathcal F}(ax)=
\lambda_{\mathcal F}(x)\] for any $a\in K\setminus\{0\}$ and all
$x,y\in W$. Therefore the function \[(x\in W)\mapsto
\|x\|:=\exp(-\lambda_{\mathcal F}(x))\] is actually a norm on $W$,
where we have considered the trivial absolute value $|.|$ on $K$
such that $|a|=1$ for any $a\in K\setminus\{0\}$.

Conversely, given a norm $\|.\|$ on the vector space $W$, one
obtains a filtration $\mathcal F_{\|.\|}$ such that
\[\mathcal F^t_{\|.\|}(W)=\{x\in W\,:\,\|x\|\leqslant \exp(-t)\}.\]
For any $x\in W$, one has
\[\|x\|=\exp(-\lambda_{\mathcal F_{\|.\|}}(x)).\]

Given a filtration $\mathcal F$ of $W$ which corresponds to the norm
$\|.\|$. By \eqref{Equ:lambdaFx} we obtain that a basis
$\mathbf{e}=(e_j)_{j=1}^r$ is compatible with the filtration
$\mathcal F$ if and only if it is an orthogonal basis of $W$ with
respect to the norm $\|.\|$, namely for any
$x=a_1e_1+\cdots+a_re_r\in W$ one has \[\|x\|=\max_{1\leqslant
i\leqslant r}\|a_ie_i\|=\max_{\begin{subarray}{c}
1\leqslant i\leqslant r\\
a_i\neq 0
\end{subarray}}\|e_i\|.\]

\subsection{Expectation}
Let $W$ be a finite dimensional vector space over ${K}$ and
$\mathcal F$ be a filtration of $W$. For any $t\in\mathbb R$, we
denote by $\mathrm{sq}_{\mathcal F}^tW$ (or simply $\mathrm{sq}^tW$)
the sub-quotient $\mathcal F^tW/\mathcal F^{t+}W$, where $\mathcal
F^{t+}W$ is defined as $\bigcup_{\varepsilon>0}\mathcal
F^{t+\varepsilon}W$. Note that there exists a finite set such that,
for any $t\in\mathbb R$ outside of this set, one has
$\mathrm{sq}_{\mathcal F}^tW=0$. Hence the sum
\[\sum_{t\in\mathbb R}t\rang(\mathrm{sq}_{\mathcal F}^tW)\]
is well defined (since it is actually a finite sum). If $W$ is
non-zero, we define the \emph{expectation} of $\mathcal F$ as
\begin{equation}\label{Equ:exp}\mathbb E[\mathcal F]:=\frac{1}{\rang(W)}\sum_{t\in\mathbb R}t\rang(\mathrm{sq}_{\mathcal F}^tW).\end{equation}
Assume that the filtration $\mathcal F$ corresponds to the flag
\[0=W_0\subsetneq W_1\subsetneq\ldots\subsetneq W_n=W\]
together with the sequence
\[a_1>\ldots>a_n.\]
One has
\[\mathbb E[\mathcal F]=\frac{1}{\rang(W)}\sum_{i=1}^na_i\rang(W_i/W_{i-1}).\]
We can also write $\mathbb E[\mathcal F]$ as an integral
\begin{equation}\mathbb E[\mathcal F]=-\frac{1}{\rang(W)}\int_{\mathbb R}
t\,\mathrm d\rang(\mathcal F^tW)=a+\int_a^{+\infty}
\frac{\rang(\mathcal F^tW)}{\rang(W)}\,\mathrm{d}t,\end{equation}
where $a$ is a sufficient negative number (such that $\mathcal
F^aW=W$).

Assume that $\mathbf{e}=(e_j)_{j=1}^r$ is a basis of $W$ which is
compatible with the filtration $\mathcal F$. The projection of
$\mathcal F^tW$ onto $\mathrm{sq}^t(W)$ induces a bijection between
$\{e_j\,|\,\lambda_{\mathcal F}(e_j)=t\}$ and its image. Moreover,
the image of $\{e_j\,|\,\lambda_{\mathcal F}(e_j)=t\}$ in the
quotient space $\mathrm{sq}^t(W)$ forms a basis of the latter.
Therefore, one has
\[\mathbb E[\mathcal F]=\frac{1}{r}\sum_{j=1}^r\lambda_{\mathcal F}(e_j).\]

If $\mathbf{e}'=(e_j')_{j=1}^r$ is another basis of $W$ which is not
necessarily compatible with the filtration $\mathcal F$, one has
\begin{equation}\label{Equ:ineqliat}
\mathbb E[\mathcal F]\geqslant
\frac{1}{r}\sum_{j=1}^r\lambda_{\mathcal F}(e_j').
\end{equation}

The zero vector space over ${K}$ has only one filtration. Its
expectation is defined to be zero by convention.

\subsection{Random variable associated to a filtration}
Let $W$ be a non-zero vector space of finite rank over ${K}$ and
$\mathcal F$ be a filtration of $W$. We associate to the filtration
$\mathcal F$ a random variable $Z_{\mathcal F}$ on $\{1,\ldots,r\}$
equipped with the equidistributed probability measure such that
\[Z_{\mathcal F}(i):=\sup\{t\in\mathbb R\,|\,\rang(\mathcal F^tW)\geqslant i\}.\]
If the filtration $\mathcal F$ corresponds to the flag
\[0=W_0\subsetneq W_1\subsetneq\ldots\subsetneq W_n=W\]
together with the decreasing sequence $a_1>\ldots>a_n$, then
$\{a_1,\ldots,a_n\}$ is just the image of the random variable
$Z_{\mathcal F}$. By definition $\mathbb E[\mathcal F]$ coincides
with the expectation of the random variable $Z_{\mathcal F}$. Denote
by $\|\mathcal F\|$ the number $\mathbb E[Z_{\mathcal F}^2]^{1/2}$.
We say that $\mathcal F$ is \emph{trivial} if $\|\mathcal F\|=0$,
namely the filtration $\mathcal F$ has only a jump point which is
$0$. We say that the filtration $\mathcal F$ is \emph{degenerated}
if the random variable $Z_{\mathcal F}$ is constant. If $\mathcal F$
and $\mathcal G$ are two filtrations of $W$, we denote by
$\langle\mathcal F,\mathcal G\rangle$ the expectation of
$Z_{\mathcal F}Z_{\mathcal G}$. By definition one has
$\langle\mathcal F,\mathcal F\rangle=\|\mathcal F\|^2$.

\subsection{Construction of filtrations} We explain below how to construct filtrations from the existing ones.

\subsubsection{Restricted filtration}
Let $W$ be a finite dimensional vector space over ${K}$ and $V$ be a
vector subspace of $W$. If $\mathcal F$ is a filtration of $W$, we
denote by $\mathcal F|_{V}$ the filtration of $V$ such
that\[\forall\,t\in\mathbb R,\quad(\mathcal F|_{V})^tV:=V\cap
\mathcal F^tW,\] called the \emph{restriction} of $\mathcal F$ on
$V$. The norm corresponding to $\mathcal F|_V$ is just the induced
norm on $V$ (from that corresponding to $\mathcal F$).

\subsubsection{Quotient filtration}
Let $W$ be a finite dimensional vector space over $K$, $V$ be a
vector subspace of $W$ and $\pi:W\rightarrow W/V$ be the quotient
map. Any filtration $\mathcal F$ of $W$ leads to a \emph{quotient
filtration} $\pi(\mathcal F)$ of $W/V$ such that
\[\forall\,t\in\mathbb R,\quad \pi(\mathcal F)^t(W/V)=\pi(\mathcal F^t(W)).\]
The norm on $W/V$ corresponding to $\pi(\mathcal F)$ is the quotient
norm.

\subsubsection{Tensor product}Let $E$ and $F$ be two finite dimensional vector spaces over ${K}$. Suppose given a filtration $\mathcal F$ of $E$ and a filtration $\mathcal G$ of $F$. We define their \emph{tensor product}  as the filtration $\mathcal F\otimes\mathcal G$ of $E\otimes F$ such that
\[(\mathcal F\otimes\mathcal G)^t(E\otimes F):=
\sum_{a+b=t}\mathcal F^aE\otimes\mathcal G^bF.\]

\begin{prop}\label{Pro:subquotient tp}
For any $t\in\mathbb R$, the sub-quotient $\mathrm{sq}^t_{\mathcal
F\otimes\mathcal G}(E\otimes F)$ of the tensor product filtration is
isomorphic to
\[\bigoplus_{a+b=t}\mathrm{sq}^a_{\mathcal F}(E)\otimes
\mathrm{sq}^b_{\mathcal G}(F).\] In particular, the following
equality holds
\begin{equation}\label{Equ:expect}\mathbb E[\mathcal F\otimes\mathcal G]=\mathbb E[\mathcal F]+\mathbb E[\mathcal G]\end{equation}
provided that neither $E$ nor $F$ is the zero vector space.
\end{prop}
\begin{proof}
Let $\mathbf{e}=(e_i)_{i=1}^n$ and $\mathbf{f}=(f_j)_{j=1}^m$ be
respectively bases of $E$ and $F$ which are compatible with the
filtrations $\mathcal F$ and $\mathcal G$. Let $\varphi$ and $\psi$
be respectively the restriction of $\lambda_{\mathcal F}$ and
$\lambda_{\mathcal G}$ on $\mathbf{e}$ and on $\mathbf{f}$. Then the
filtration $\mathcal F\otimes\mathcal G$ is associated to the basis
\[\mathbf{e}\otimes\mathbf{f}:=\{e_i\otimes f_j\,|\,1\leqslant i\leqslant n,\;1\leqslant j\leqslant m\}\]
of $E\otimes F$ and the function $\varphi\otimes\psi$ which sends
$e_i\otimes f_j$ to $\varphi(e_i)+\psi(f_j)$. For any $t\in\mathbb
R$, let $\mathbf{e}_t=\{e_i\,|\,\lambda_{\mathcal F}(e_i)=t\}$ and
$\mathbf{f}_t=\{f_j\,|\,\lambda_{\mathcal G}(f_j)=t\}$. The
bijection between $\bigcup_{a+b=t}\mathbf{e}_a\times\mathbf{f}_b$
and
\[(\mathbf{e}\otimes\mathbf{f})_t=\{e_i\otimes f_j\,|\,
\lambda_{\mathcal F}(e_i)+\lambda_{\mathcal G}(f_j)=t\}=
\{e_i\otimes f_j\,|\,\lambda_{\mathcal F\otimes\mathcal
G}(e_i\otimes f_j)=t\}\] which sends $(e_i,f_j)$ to $e_i\otimes f_j$
induces an isomorphism between the $K$-vector spaces
$\bigoplus_{a+b=t}\mathrm{sq}_{\mathcal
F}^a(E)\otimes\mathrm{sq}_{\mathcal G}^b(F)$ and
$\mathrm{sq}_{\mathcal F\otimes\mathcal G}^t(E\otimes F)$.

The equality \eqref{Equ:expect} is a direct consequence of the
relation
\[\rang(\mathrm{sq}^t_{\mathcal F\otimes\mathcal G}(E\otimes F))=\sum_{a+b=t}\rang(\mathrm{sq}^a_{\mathcal F}(E))\rang(\mathrm{sq}^b_{\mathcal G}(F))\]
and the definition of the expectation \eqref{Equ:exp}.
\end{proof}
From the proof of the previous proposition, we observe that, if
$\mathbf{e}$ (resp. $\mathbf{f}$) is a basis of $E$ (resp. $F$)
which is compatible with the filtration $\mathcal F$ (resp.
$\mathcal G$), then $\mathbf{e}\otimes\mathbf{f}$ is a basis of
$E\otimes F$ which is compatible with the tensor product filtration
$\mathcal F\otimes\mathcal G$. From the metric point of view, this
means that the norm on $E\otimes F$ corresponding to the tensor
product filtration $\mathcal F\otimes\mathcal G$ is the tensor
product norm.

\subsubsection{Exterior product} Let $E$ be a finite dimensional vector space over $K$ and $n\geqslant 1$ be an integer. Denote by $\pi_n:E^{\otimes n}\rightarrow\Lambda^nE$ the natural projection. If $\mathcal F$ is a filtration of $E$, we denote by $\Lambda^n\mathcal F$ the quotient filtration $\pi_*(\mathcal F^{\otimes n})$, called the \emph{exterior product filtration}. If $\mathbf{e}=(e_i)_{i=1}^r$ is a basis of $E$ which is compatible with the filtration $\mathcal F$, then $(e_{i_1}\wedge\cdots\wedge e_{i_n})_{1\leqslant i_1<\ldots<i_n\leqslant r}$ is a basis of $\Lambda^nE$ which is compatible with $\Lambda^n\mathcal F$. Moreover, one has
\[\lambda_{\Lambda^n\mathcal F}(e_{i_1}\wedge\cdots\wedge e_{i_n})=\lambda_{\mathcal F}(e_{i_1})+\cdots+\lambda_{\mathcal F}(e_{i_n}).\]
In particular, one has
\[\binom{r}{n}\mathbb E[\Lambda^n\mathcal F]=\sum_{1\leqslant i_1<\ldots<i_n\leqslant r}\sum_{j=1}^n\lambda_{\mathcal F}(e_{i_j})=\binom{r-1}{n-1}\sum_{i=1}^r\lambda_{\mathcal F}(e_i)=r\binom{r-1}{n-1}\mathbb E[\mathcal F],\]
where the second equality can be proved by induction on $n$. Hence
we obtain
\begin{equation}\label{Equ:explamabda}
\mathbb E[\Lambda^n\mathcal F]=n\mathbb E[\mathcal F].
\end{equation}

\subsubsection{Direct sum}
Let $E$ and $F$ be two finite dimensional vector spaces over ${K}$.
Suppose given a filtration $\mathcal F$ of $E$ and a filtration
$\mathcal G$ of $F$. We define the {\it direct sum} of $\mathcal F$
and $\mathcal G$ as the filtration $\mathcal F\oplus\mathcal G$ of
$E\oplus F$ such that
\[(\mathcal F\oplus\mathcal G)^t(E\oplus F)=\mathcal F^tE\oplus\mathcal G^tF.\]
One has
\[(\rang(E)+\rang(F))\mathbb E[\mathcal F\oplus\mathcal G]=\rang(E)\mathbb E[\mathcal F]+\rang(F)\mathbb E[
\mathcal G].\]

\subsubsection{Refinement}\label{Subsubsec:refinement}
Let $W$ be a finite dimensional vector space over ${K}$ which is
non-zero. Let $\mathcal F$ be a filtration of $W$. Suppose given,
for any $t\in\mathbb R$, a filtration $\mathcal G_{(t)}$ of
$\mathrm{sq}_{\mathcal F}^t(W)$. Let $\widetilde{\mathbf{e}}_{t}$ be
a basis of $\mathrm{sq}_{\mathcal F}^t(W)$ which is compatible with
the filtration $\mathcal G_{(t)}$. The vector family
$\widetilde{\mathbf{e}}_{t}$ gives rise to a linearly independent
family $\mathbf{e}_t$ in $\mathcal F^t(W)$. It turns out that
$\mathbf{e}:=\bigcup_{t\in\mathbb R}\mathbf{e}_t$  is a basis of
$W$. For any $t\in\mathbb R$ and any $x\in\mathbf{e}_t$, let
$\varphi(x)=\lambda_{\mathcal G_{(t)}}(\bar{x})$, where $\bar{x}$
denotes the class of $x$ in $\mathrm{sq}_{\mathcal F}^t(W)$. Let
$\mathcal G$ be the filtration associated to the basis $\mathbf{e}$
and the function $\varphi:\mathbf{e}\rightarrow\mathbb R$. By
definition one has
\begin{equation}\label{Equ:EG}\mathbb E[\mathcal G]=\frac{1}{\rang(W)}\sum_{t\in\mathbb R}\sum_{x\in\mathbf{e}_t}\varphi(x)=\sum_{t\in\mathbb R}\frac{\rang(\mathrm{sq}_{\mathcal F}^t(W))}{\rang(W)}\mathbb E[\mathcal G_{(t)}].\end{equation}
Moreover, the basis $\mathbf{e}$ is simultaneously compatible with
the filtrations $\mathcal F$ and $\mathcal G$. In particular, one
has
\begin{equation}\label{Equ:innerFG}\langle\mathcal F,\mathcal G\rangle=\frac{1}{\rang(W)}
\sum_{t\in\mathbb R}t\sum_{x\in\mathbf{e}_t}\varphi(x)
=\sum_{t\in\mathbb R}\frac{\rang(\mathrm{sq}_{\mathcal
F}^t(W))}{\rang(W)}t\mathbb E[\mathcal G_{(t)}].\end{equation}

\subsubsection{Translation and dilation}

Let $W$ be a finite dimensional vector space over $K$ and $\mathcal
F$ be a filtration of $W$. For any real number $a$, let
$\tau_a\mathcal F$ be the filtration of $W$ such that
\[(\tau_a\mathcal F)^tW=\mathcal F^{a+t}(W).\]
One has $\mathbb E[\tau_a\mathcal F]=\mathbb E[\mathcal F]+a$. For
any positive real number $\varepsilon$, let $\varepsilon\mathcal F$
be the filtration of $W$ such that
\[(\varepsilon\mathcal F)^tW=\mathcal F^{\varepsilon t}W.\]
One has $\mathbb E[\varepsilon\mathcal F]=\varepsilon\mathbb
E[\mathcal F]$. Moreover, the following relation between the
translation and the dilation holds. For all $a\in\mathbb R$ and
$\varepsilon>0$, one has
\begin{equation}\label{Equ:taua}\tau_{a\varepsilon}(\varepsilon\mathcal F)=\varepsilon(\tau_a\mathcal F).\end{equation}
If $V$ is a vector subspace of $W$, one has
\begin{equation}\label{Equ:tauv}\tau_a(\mathcal F|_V)=(\tau_a\mathcal F)|_V\quad\text{and}\quad
\varepsilon(\mathcal F|_V)=(\varepsilon\mathcal F)|_V.\end{equation}

Let $E$ and $F$ be two finite dimensional vector spaces
over $K$, equipped with filtrations $\mathcal F$ and $\mathcal G$
respectively. For any real number $a$ one has
\begin{equation}\label{Equ:tauprten}\tau_a(\mathcal F\oplus\mathcal G)=(\tau_a\mathcal F)\oplus(\tau_a\mathcal G),\quad
\tau_a(\mathcal F\otimes\mathcal G)=(\tau_a\mathcal
F)\otimes\mathcal G=\mathcal F\otimes(\tau_a\mathcal
G).\end{equation} For any $\varepsilon>0$, one has
\begin{equation}\label{Equ:varepsion}\varepsilon(\mathcal F\oplus\mathcal G)=(\varepsilon\mathcal F)\oplus(\varepsilon\mathcal G),\quad
\varepsilon(\mathcal F\otimes\mathcal G)=(\varepsilon\mathcal
F)\otimes(\varepsilon\mathcal G).\end{equation}

Let $W$ be a finite dimensional vector space over $K$, and
$\mathcal F$ and $\mathcal G$ be two filtrations of $W$. For
any real number $a$ and any $\varepsilon>0$, one has
\begin{gather}
\label{Equ:correl}Z_{\tau_a\mathcal F}=Z_{\mathcal F}+a,\quad\langle\tau_a\mathcal F,\mathcal G\rangle=\langle\mathcal F,\mathcal G\rangle+a\mathbb E[\mathcal G],\\
\label{Equ:epscoro} Z_{\varepsilon\mathcal F}=\varepsilon
Z_{\mathcal F},\quad\langle\varepsilon\mathcal F,\mathcal
G\rangle=\varepsilon\langle\mathcal F,\mathcal G\rangle.
\end{gather}

\subsubsection{Dual filtration}

Let $W$ be a finite dimensional vector space over $K$ and $\mathcal
F$ be a filtration of $W$. We define a filtration $\mathcal F^\vee$
of the dual space $W^\vee$ as follows. For any $t\in\mathbb R$,
\[(\mathcal F^\vee)^tW^\vee:=(\mathcal F^{-t}W)^{\perp},\]
where for any vector subspace $V$ of $W$, $V^\perp$ denotes the
subset of $W^\vee$ of linear forms $\varphi$ such that
$\varphi|_{V}$ is zero. By definition, for any element
$\varphi\in W^\vee$, one has
\[\lambda_{\mathcal F^\vee}(\varphi)=-\inf_{\begin{subarray}{c}x\in W\\\varphi(x)\neq 0\end{subarray}} \lambda_{\mathcal F}(x).\]
Moreover, one has $\mathbb E[\mathcal F^\vee]=-\mathbb E[\mathcal
F]$.

\subsection{An expectation inequality}

Let $E$ and $F$ be two non-zero finite dimensional vector space over
$K$, and $\mathcal F$ and $\mathcal G$ be respectively filtrations
of $E$ and $F$. Let $\mathbf{e}=(e_i)_{i=1}^n$ and
$\mathbf{f}=(f_j)_{j=1}^m$ be bases of $E$ and $F$ which are
respectively compatible with filtrations $\mathcal F$ and $\mathcal
G$. By virtue of the proof of Proposition \ref{Pro:subquotient tp},
the set
\[\mathbf{e}\otimes\mathbf{f}=\{e_i\otimes f_j\,|\,1\leqslant i\leqslant n,\;1\leqslant j\leqslant m\}\]
is a basis of $E\otimes F$ which is compatible with $\mathcal
F\otimes\mathcal G$. Consider a vector subspace $V$ of $E\otimes F$
which is of dimension one. Let $\varphi$ be a non-zero element in
$V$, which is written as
\[\varphi=\sum_{i,j}a_{ij}e_i\otimes f_j,\]
where $a_{ij}$ are elements in $K$. By \eqref{Equ:lambdaFx}, one has
\[\mathbb E[(\mathcal F\otimes\mathcal G)|_V]=\lambda_{\mathcal F\otimes\mathcal G}(\varphi)=
\min\{\lambda_{\mathcal F}(e_i)+\lambda_{\mathcal
G}(f_j)\,|\,a_{ij}\neq 0\}.\] Note that the tensor vector $\varphi$
corresponds to a $K$-linear map from $E^\vee$ to $F$, which we
denote by $T_{\varphi}$.

\begin{prop}\label{Pro:semist}
With the notation above, if $T_\varphi$ is an isomorphism of
$K$-vector spaces, then the following inequality holds
\begin{equation}\label{Equ:semistabilit}
\mathbb E[(\mathcal F\otimes\mathcal G)|_V]\leqslant\mathbb
E[\mathcal F]+\mathbb E[\mathcal G].
\end{equation}
\end{prop}
\begin{proof} Let $\mathbf{e}=(e_i)_{i=1}^n$ be a basis of $W$ of $E$ which is compatible with the filtration $\mathcal F$.
Denote by $\mathbf{e}^\vee=(e_i^\vee)_{i=1}^n$ the dual basis of
$\mathbf{e}$. For any $i\in\{1,\ldots,n\}$, one has
\begin{equation}\label{Equ:lambdai}\lambda_{\mathcal F}(e_i)+\lambda_{\mathcal G}(T_{\varphi}(e_i^\vee))=
\lambda_{\mathcal F}(e_i)+\min\{\lambda_{\mathcal
G}(f_j)\,|\,a_{ij}\neq 0\}\geqslant \lambda_{\mathcal
F\otimes\mathcal G}(\varphi).\end{equation} Since $\mathbf{e}$ is a
basis of $E$ which is compatible with the filtration $\mathcal F$,
one has
\[\frac{1}{n}\sum_{i=1}^n\lambda_{\mathcal F}(e_i)=\mathbb E[\mathcal F].\]
Moreover, since $T_{\varphi}$ is an isomorphism,
$(T_{\varphi}(e_i^\vee))_{i=1}^n$ is a basis of $F$. By
\eqref{Equ:ineqliat} one has
\[\frac{1}{n}\sum_{i=1}^n\lambda_{\mathcal G}(T_{\varphi}(e_i^\vee))\leqslant\mathbb E[\mathcal G].\]
By taking the average of \eqref{Equ:lambdai} with respect to
$i\in\{1,\ldots,n\}$, one obtains \eqref{Equ:semistabilit}.
\end{proof}

\section{Geometric semistability} In this section, we introduce several (semi)stability conditions for vector subspaces in the tensor product of two vector spaces, and discuss their properties.
\subsection{Conditions of semistability}
Let $E$ and $F$ be two finite dimensional vector spaces and $V$ be a
non-zero vector subspace of $E\otimes F$.
\begin{enumerate}[1)]
\item We say that $V$ is \emph{left  semistable} (resp. \emph{left  stable}) if for any non-degenerated filtration $\mathcal F$ of $E$, one has
\[\mathbb E[(\mathcal F\otimes\mathcal G_0)|_V]\leqslant\mathbb E[\mathcal F]\quad (\text{resp. }\mathbb E[(\mathcal F\otimes\mathcal G_0)|_V]<\mathbb E[\mathcal F]),\]
where $\mathcal G_0$ denotes the trivial filtration of $F$.
\item We say that
$V$ is \emph{right  semistable} (resp. \emph{right stable}) if for
any non-degenerated filtration $\mathcal G$ of $F$, one has
\[\mathbb E[(\mathcal F_0\otimes\mathcal G)|_V]\leqslant\mathbb E[\mathcal G]\quad (\text{resp. }\mathbb E[(\mathcal F_0\otimes\mathcal G)|_V]<\mathbb E[\mathcal G]),\]
where $\mathcal F_0$ denotes the trivial filtration of $E$.
\item We say that $V$ is \emph{both-sided  semistable} (resp. \emph{both-sided  stable}) if for any filtration $\mathcal F$ of $E$ and any filtration $\mathcal G$ of $F$ such that at least one filtration between them is non-degenerated, one has
    \[\mathbb E[(\mathcal F\otimes\mathcal G)|_V]\leqslant\mathbb E[\mathcal F]+\mathbb E[\mathcal G]\quad (\text{resp. }\mathbb E[(\mathcal F\otimes\mathcal G)|_V]<\mathbb E[\mathcal F]+\mathbb E[\mathcal G]),\]
\end{enumerate}
Remind that a filtration is said to be non-degenerated if its
associated random variable does not reduce to a constant function.

\begin{rema}
Denote by $\mathbb P(V,\mathcal F,\mathcal G)$ the relation $\mathbb
E[(\mathcal F\otimes\mathcal G)|_V]\leqslant\mathbb E[\mathcal
F]+\mathbb E[\mathcal G]$. By \eqref{Equ:taua} and
\eqref{Equ:tauprten}, for any $a\in\mathbb R$, one has
\begin{equation}\mathbb P(V,\mathcal F,\mathcal G)\Longleftrightarrow\mathbb P(V,\tau_a\mathcal F,\mathcal G)\Longleftrightarrow\mathbb P(V,\mathcal F,\tau_a\mathcal G)
\end{equation}
Therefore, $V$ is both-sided semistable if and only if, for any
filtration $\mathcal F$ of $E$ and any filtration $\mathcal G$ of
$F$ satisfying $\mathbb E[\mathcal F]=\mathbb E[\mathcal G]=0$, one
has $\mathbb E[(\mathcal F\otimes\mathcal G)|_V]\leqslant 0$. Note
that the inclusion map of $V$ into $E\otimes F$ defines a rational
point $x$ of the projective space $\mathbb
P(\Lambda^{\rang(V)}(E\otimes F)^\vee)$ (with the convention of
Grothendieck). The argument above shows that the vector subspace $V$
is both-sided semistable if and only if the point $x$ is semistable
in the sense of the geometric invariant theory under the action of
$\mathrm{SL}(E)\times\mathrm{SL}(F)$ (with respect to the universal
line bundle).
\end{rema}

\begin{prop}\label{Pro:semistable}
Let $E$ and $F$ be two finite dimensional vector spaces over ${K}$,
and $V$ be a non-zero vector subspace of $E\otimes F$. Then $V$ is
left  semistable (resp. left stable) if and only if, for any vector
subspace $E_1$ of $E$ such that $0\subsetneq E_1\subsetneq E$, one
has
\begin{equation}\label{Equ:stable}\frac{\rang(V\cap(E_1\otimes F))}{\rang(V)}\leqslant\frac{\rang(E_1)}{\rang(E)}\qquad
(\text{resp. }\frac{\rang(V\cap(E_1\otimes
F))}{\rang(V)}<\frac{\rang(E_1)}{\rang(E)}).\end{equation} A similar
result also holds for the right (semi)stability.
\end{prop}
\begin{proof}
``$\Longrightarrow$'': Let $\mathcal F$ be the filtration of $E$
corresponding to the flag $0\subsetneq E_1\subsetneq E$ together
with the jump points $\{1,0\}$, and let $\mathcal G_0$ be the
trivial filtration of $F$. One has
\[\mathbb E[(\mathcal F\otimes\mathcal G_0)|_V]=\frac{\rang(V\cap(E_1\otimes F))}{\rang(V)}\qquad\text{and}\qquad\mathbb E[\mathcal F]=\frac{\rang(E_1)}{\rang(E)}.\]
Therefore, if $V$ is left (semi)stable, then the inequality
\eqref{Equ:stable} holds.

``$\Longleftarrow$'': Assume that the inequality \eqref{Equ:stable}
holds for any vector subspace $E_1$ of $E$ such that $0\subsetneq
E_1\subsetneq E$. Let $\mathcal F$ be a filtration of $E$. One has,
for sufficient negative number $a$,
\[\mathbb E[(\mathcal F\otimes\mathcal G_0)|_V]=a+\int_{a}^{+\infty}
\frac{\rang(V\cap(\mathcal F^tE\otimes
F))}{\rang(V)}\,\mathrm{d}t\leqslant a+\int_a^{+\infty}
\frac{\rang(\mathcal F^tE)}{\rang(E)}\,\mathrm{d}t=\mathbb
E[\mathcal F].\] The inequality is strict when $V$ is left stable
since there then exists an interval $I$ of positive length such that
$0\subsetneq \mathcal F^tE\subsetneq E$ for any $t\in I$.
\end{proof}

\begin{rema}
We observe from the above definition that, if $V$ is both-sided
(semi)stable, then it is left and right (semi)stable. However, the
converse is not true in general. The following is a counter-example.
Let $E$ and $F$ be two vector spaces of dimension $3$ over ${K}$.
Let $(x_1,x_2,x_3)$ and $(y_1,y_2,y_3)$ be respectively a basis of
$E$ and of $F$. Denote by $E_1$ (resp. $F_1$) the vector subspace of
$E$ (resp. $F$) generated by $x_1$ (resp. $y_1$), and let $V$ be the
two-dimensional vector subspace of $E\otimes F$ generated by
$x_1\otimes y_2+x_2\otimes y_1$ and $x_1\otimes y_3+x_3\otimes y_1$.
Consider the filtration $\mathcal F$ of $E$ corresponding to the
flag $0\subsetneq E_1\subsetneq E$ and the jump points $\{1,0\}$,
and the filtration $\mathcal G$ of $F$ corresponding to $0\subsetneq
F_1\subsetneq F$ and $\{1,0\}$. By definition, the filtration
$\mathcal F\otimes\mathcal G$ corresponds to $0\subsetneq E_1\otimes
F_1\subsetneq E_1\otimes F+E\otimes F_1\subsetneq E\otimes F$ and
$\{2,1,0\}$. Therefore $(\mathcal F\otimes\mathcal G)|_V$ is the
filtration with only one jump point at $1$. Hence
\[\mathbb E[(\mathcal F\otimes\mathcal G)|_V]=1>\mathbb E[\mathcal F]+\mathbb E[\mathcal G]=2/3,\]
which shows that $V$ is not both-sided semistable. However, for any
subspace $E'$ of $E$ such that $0\subsetneq E'\subsetneq E$, if
$\rang(E')=1$, then $V\cap(E'\otimes F)=0$; if $\rang(E')=2$, then
$\rang(V\cap(E'\otimes F))\leqslant 1$. Therefore, we always have
\[\frac{\rang(V\cap(E'\otimes F))}{\rang(V)}\leqslant\frac{\rang(E')}{\rang(E)}.\]
By Proposition \ref{Pro:semistable}, we obtain that $V$ is left
stable. Similarly we can verify that it is also right stable.
\end{rema}

The following propositions describe the link between the geometric
(semi)stability and the tensorial rank of lines in the tensor product.

\begin{prop}
Let $E$ and $F$ be two finite dimensional vector spaces over ${K}$.
Let $n$ and $m$ be respectively the rank of $E$ and $F$. Let $V$ be
a one-dimensional vector subspace of $E\otimes F$. Denote by
$\rho(V)$ the tensorial rank of $V$, namely the rank of an arbitrary
non-zero vector in $V$, considered as a $K$-linear map from $E^\vee$
to $F$.
\begin{enumerate}[1)]
\item The following conditions are equivalent~:
\begin{enumerate}[i)]
\item $V$ is left (resp. right) stable,
\item $V$ is left (resp. right) semistable,
\item $\rho(V)=n$ (resp. $\rho(V)=m$).
\end{enumerate}
\item If $V$ is left and right semistable, then it is both-sided semistable.
\end{enumerate}
\end{prop}
\begin{proof}
1) The implication ``i)$\Rightarrow$ii)'' is trivial.

``ii)$\Rightarrow$iii)'': let $E_1$ be the image of $V$ in $E$,
namely $E_1$ is the smallest vector subspace of $E$ such that
$V\subset E_1\otimes F$. If $\rho(V)<n$, then $0\subsetneq
E_1\subsetneq E$. One has
\[\frac{\rang(V\cap(E_1\otimes F))}{\rang(V)}=1>\frac{\rang(E_1)}{\rang(E)}.\]
By Proposition \ref{Pro:semistable}, $V$ is not left semistable.

``iii)$\Rightarrow$i)'': Assume that $\rho(V)=n$. For any vector
subspace $E_1$ of $E$ such that $0\subsetneq E_1\subsetneq E$, one
has $V\cap(E_1\otimes F)=0$. Hence
\[\frac{\rang(V\cap(E_1\otimes F))}{\rang(V)}=0<\frac{\rang(E_1)}{\rang(E)}.\]
Again by Proposition \ref{Pro:semistable}, we obtain that $V$ is
left stable.

2) Let $\varphi$ be a non-zero element in $V$. If $V$ is left and
right semistable, then $\varphi$ is an isomorphism (as a $K$-linear
map from $E^\vee$ to $F$). By Proposition \ref{Pro:semist}, $V$ is
both-sided semistable.
\end{proof}

\begin{rema}
Note that even in the case where $V$ is one-dimensional, the
both-sided semistability is in general \emph{not} equivalent to the
both-sided stability. Consider the following example. Let $E$ and
$F$ be vector spaces of rank two over ${K}$, and let $(x_1,x_2)$ and
$(y_1,y_2)$ be respectively a basis of $E$ and $F$. Consider the
subspace $V$ of $E\otimes F$ generated by $x_1\otimes y_1+x_2\otimes
y_2$. Let $\mathcal F$ be the filtration of $E$ corresponding to the
flag $0\subsetneq E_1\subsetneq E$ together with the jump set
$\{1,0\}$, where $E_1$ is the vector subspace of $E$ generated by
$x_1$. Let $\mathcal G$ be the filtration of $F$ corresponding to
the flag $0\subsetneq F_1\subsetneq F$ together with the jump set
$\{1,0\}$, where $F_1$ is the vector subspace of $F$ generated by
$y_1$. The filtration $(\mathcal F\otimes\mathcal G)|_V$ has only
one jump at $1$ since $V$ is contained in $E_1\otimes F+E\otimes
F_1$. Hence $\mathbb E[(\mathcal F\otimes\mathcal G)|_V]= 1=\mathbb
E[\mathcal F]+\mathbb E[\mathcal G]$.
\end{rema}

\begin{prop}
Let $E$ and $F$ be two finite dimensional vector spaces over $K$ and
$V$ be a non-zero vector subspace of $E\otimes F$. If $V$ is left
semistable, then the image of $V$ in $E$ is equal to $E$. Similarly,
if $V$ is right semistable, then the image of $V$ in $F$ is equal to
$F$.
\end{prop}
\begin{proof}
Let $E_1$ be the image of $V$ in $E$ (namely the smallest vector
subspace of $E$ such that $V\subset E_1\otimes F$). One has
$V\subset E_1\otimes F$. Hence by the left semistability of $V$ one
has (see Proposition \ref{Pro:semistable}) $\rang(E_1)\geqslant
\rang(E)$. Therefore $E_1=E$.
\end{proof}

\begin{prop}
Let $E$ and $F$ be two finite dimensional vector spaces over $K$ and
$V$ be a non-zero vector subspace of $E\otimes F$. Let $E_1$ and
$F_1$ be respectively vector subspaces of $E$ and $F$ such that
$0\subsetneq E_1\subsetneq E$ and $0\subsetneq F_1\subsetneq F$, and
\[
r_1=\rang(V\cap(E_1\otimes F_1)),\quad r_2=\rang(V\cap(E_1\otimes
F+E\otimes F_1)).
\]
Assume that $V$ is both-sided stable, then one has
\begin{equation}
\frac{r_1+r_2}{\rang(V)}<\frac{\rang(E_1)}{\rang(E)}+\frac{\rang(F_1)}{\rang(F)}.
\end{equation}
\end{prop}
\begin{proof}
Let $\mathcal F$ be the filtration of $E$ corresponding to the flag
$0\subsetneq E_1\subsetneq E$ and the sequence $1>0$. Let $\mathcal
G$ be the filtration of $F$ corresponding to the flag $0\subsetneq
F_1\subsetneq F$ and the sequence $1>0$. Then the tensor product
filtration $\mathcal F\otimes\mathcal G$ corresponds to the flag
\[0\subsetneq E_1\otimes F_1\subsetneq E_1\otimes F+E\otimes F_1\subsetneq E\otimes F\]
and the sequence $2>1>0$. Therefore, the expectation of the
restricted filtration $(\mathcal F\otimes\mathcal G)|_V$ is
\[\frac{2r_1+(r_2-r_1)}{\rang(V)}=\frac{r_1+r_2}{\rang(V)}.\]
Since $V$ is both-sided stable, one has
\[\mathbb E[(\mathcal F\otimes\mathcal G)|_V]<\mathbb E[\mathcal F]+\mathbb E[\mathcal G]=\frac{\rang(E_1)}{\rang(E)}+\frac{\rang(F_1)}{\rang(F)}.\]
The proposition is thus proved.
\end{proof}

\begin{coro}\label{Cor:leng}
Let $E$ and $F$ be finite dimensional vector spaces over $K$ and $V$
be a non-zero vector subspace of $E\otimes F$. Assume that $V$ is
both-sided stable and contains a line $M$ of tensorial rank $\rho$, then one
has
\[\frac{2}{\rang(V)}<\frac{\rho}{\rang(E)}+\frac{\rho}{\rang(F)}.\]
\end{coro}
\begin{proof}
Let $E_1$ and $F_1$ be respectively the images of $M$ in $E$ and in
$F$. One has $\rang(V\cap(E_1\otimes F_1))\geqslant 1$ and
$\rang(V\cap(E_1\otimes F+E\otimes F_1))\geqslant1$. Note that
$\rang(E_1)=\rang(F_1)=\rho$. By the previous proposition, one
obtains the required inequality.
\end{proof}

\begin{prop}\label{Pro:ell}
Let $E$ and $F$ be two finite dimensional vector spaces over $K$ and
$V$ be a non-zero vector subspace of $E\otimes F$. Assume that $V$
is left stable and contains a line of tensorial rank $\rho$, then one has
\[\frac{1}{\rang(V)}<\frac{\rho}{\rang(E)}.\]
Similarly, if $V$ is right stable and contains a line of tensorial rank
$\rho$, then one has
\[\frac{1}{\rang(V)}<\frac{\rho}{\rang(F)}.\]
\end{prop}
\begin{proof}
By symmetry it suffices to prove the first inequality. Let $M$ be a
one-dimensional subspace of $V$ which has tensorial rank $\rho$, and let
$E_1$ be its image in $E$. Since $V$ is left stable, by Proposition
\ref{Pro:semistable} one obtains
\[\frac{\rang(V\cap(E_1\otimes F))}{\rang(V)}<\frac{\rang(E_1)}{\rang(E)}=\frac{\rho}{\rang(E)}.\]
Since $V\cap(E_1\otimes F)$ contains $M$, it has rank $\geqslant 1$.
The required inequality is thus proved.
\end{proof}

\subsection{Kempf's filtration pair}

We recall below a result of Totaro \cite[Proposition 2]{Totaro96}
which proves the existence of Kempf's filtration pair for a vector
subspace of the tensor product of two vector spaces which is not
both-sided semistable.

\begin{prop}\label{Pro:destabilite}
Let $E$ and $F$ be two finite dimensional vector spaces over ${K}$,
and $\mathbf{S}$ be the subset of
$\mathbf{Fil}(E)\times\mathbf{Fil}(F)$ consisting of filtration
pairs $(\mathcal F,\mathcal G)$ such that $\mathbb E[\mathcal
F]=\mathbb E[\mathcal G]=0$ and that both of filtrations $\mathcal
F$ and $\mathcal G$ are not trivial. Assume that $V$ is a non-zero
vector subspace of $E\otimes F$ which is not both-sided stable. Then
the function $\Theta:\mathbf{S}\rightarrow\mathbb R$ which sends
$(\mathcal F,\mathcal G)$ to $-{\mathbb E[(\mathcal F\otimes\mathcal
G)|_V]}{(\|\mathcal F\|^2+\|\mathcal G\|^2)^{-1/2}}$ attains its
minimum value $c$ which is non-positive. Moreover, if $(\mathcal
F_1,\mathcal G_1)$ is a point in $\mathbf{S}$ at which the function
$\Theta$ takes its minimum value, then for all filtrations $\mathcal
F\in\mathbf{Fil}(E)$ and $\mathcal G\in\mathbf{Fil}(F)$, one has
\begin{equation}\mathbb E\mathbb[(\mathcal F\otimes\mathcal G)|_V]\leqslant \mathbb E[\mathcal F]+\mathbb E[\mathcal G]-c\frac{\langle\mathcal F_1,\mathcal F\rangle+\langle\mathcal G_1,\mathcal G\rangle}{(\|\mathcal F_1\|^2+\|\mathcal G_1\|^2)^{1/2}}.\end{equation}
\end{prop}

The key point of the proof is a convexity lemma which has been
introduced by Ramanan and Ramanathan (cf. \cite[Lemma
1.1]{Ramanan_Ramanathan}, see also \cite[Lemma 3]{Totaro96}). Note
that in \cite[Proposition 2]{Totaro96}, Totaro has assumed that $V$
is not both-sided semistable. In this case, the minimizing
filtration pair $(\mathcal F_1,\mathcal G_1)$ is unique up to
dilation, and the minimum value $c$ of the function $\Theta$ is
negative. In the case where $V$ is not both-sided stable but
both-sided semistable, the proposition also holds and the proof is
the same as in \cite{Totaro96}. However, in this case, the minimum
value $c$ equals $0$, and the minimizing filtration pair need not be
unique up to dilation.

\section{Applications in arithmetic semistability}

In this section, we discuss the relationship between the geometric
and arithmetic (semi)stability conditions.

\subsection{A criterion of the arithmetic semistability}

Let $\overline E$ be a non-zero hermitian vector bundle over
$\overline{\mathbb Q}$. For any filtration $\mathcal F$ of $E$, we
define
\[\mathbb E_{\hat{\mu}}[\mathcal F]:=-\frac{1}{\rang(E)}\int_{\mathbb R}t\,\mathrm{d}\ndeg(\mathcal F^t\overline E).\]
The following is a criterion of semistablity which is essentially
due to Bogomolov (see \cite{Raynaud81}).

\begin{prop}
A non-zero hermitian vector bundle $\overline E$ over
$\overline{\mathbb Q}$ is semistable (resp. stable) if and only if,
for any non-degenerate filtration $\mathcal F$ of $E$, one has
\[\mathbb E_{\hat{\mu}}[\mathcal F]\leqslant \hat{\mu}(\overline E) \mathbb E[\mathcal F]\qquad(\text{resp. } \mathbb E_{\hat{\mu}}[\mathcal F]<\hat{\mu}(\overline E)\mathbb E[\mathcal F]).\]
\end{prop}
\begin{proof}
``$\Longleftarrow$'': Let $F$ be a vector subspace of $E$ such that
$0\subsetneq F\subsetneq E$. Denote by $\mathcal F$ the filtration
of $E$ corresponding to the flag $0\subsetneq F\subsetneq E$ and the
jump points $1>0$. One has
\[\mathbb E_{\hat{\mu}}[\mathcal F]=\frac{\ndeg(\overline F)}{\rang(E)}\qquad\text{and}\qquad
\mathbb E[\mathcal F]\hat{\mu}(\overline
E)=\frac{\rang(F)}{\rang(E)}\hat{\mu}(\overline E).\] Hence we
obtain the desired result.

``$\Longrightarrow$'': Let $\mathcal F$ be a non-degenerate
filtration of $E$. If $\overline E$ is semistable, for sufficiently
negative number $a$, one has
\[\begin{split}&\quad\;\mathbb E_{\hat{\mu}}[\mathcal F]=-a\hat{\mu}(\overline E)+\int_{a}^{+\infty}\frac{\ndeg(\mathcal F^t\overline E)}{\rang(E)}\,\mathrm{d}t
\\&\leqslant -a\hat{\mu}(\overline E)+\hat{\mu}(\overline E)\int_a^{+\infty}\frac{\rang(\mathcal F^t(E))}{\rang(E)}\,\mathrm{d}t=\hat{\mu}(\overline E)\mathbb E[\mathcal F],
\end{split}\]
where the inequality is strict if $\overline E$ is stable.
\end{proof}

\subsection{Harder-Narasimhan filtration}

The $\mathbb R$-indexed Harder-Narasimhan filtration has been
introduced  \cite[\S 2.2.2]{Chen10b}. Let $\overline E$ be a
non-zero hermitian vector bundle over $\overline{\mathbb Q}$ and
$0=E_0\subsetneq E_1\subsetneq\ldots\subsetneq E_n=E$ be its
Harder-Narasimhan flag (see \S\ref{Subsec:slopes}). We define a
filtration $\mathcal F$ of $E$ as follows:
\[\mathcal F^tE=\bigcup_{\begin{subarray}{c}1\leqslant i\leqslant n\\ \mu_i\geqslant t
\end{subarray}}E_i,\]
where for each $i\in\{1,\ldots,n\}$, $\mu_i:=\hat{\mu}(\overline
E_i/\overline E_{i-1})$. It can be shown that (see \cite{Chen10b}
Corollary 2.2.3)
\[\mathcal F^tE=\sum_{\begin{subarray}{c}
0\neq F\subset E\\
\hat{\mu}_{\min}(\overline F)\geqslant t
\end{subarray}}F.\]
Moreover, we observe from the definition that the expectation of
$\mathcal F$ is just $\hat{\mu}(\overline E)$. The filtration
$\mathcal F$ is called the ($\mathbb R$-indexed) Harder-Narasimhan
filtration of $\overline E$.

Note that, if $\overline E$ is a non-zero hermitian vector bundle
over $\overline{\mathbb Q}$ and if $\mathcal F$ is the
Harder-Narasimhan filtration of it, then for any $t\in\mathbb R$,
either the subquotient \[\mathrm{sq}_{\mathcal F}^t(\overline
E)=\mathcal F^t\overline E/\mathcal F^{t+}\overline E\] is either
zero or  semistable of slope $t$.

\subsection{Strong slope inequalities under geometric (semi)stability}

Let $r,n,m$ be three integers which are $\geqslant 1$. We say that
\emph{ the condition $B(r,n,m)$ holds} if for all hermitian vector bundles
$\overline E$ and $\overline F$ of rank $n$ and $m$ on
$\Spec{\overline{\mathbb Q}}$ respectively, and any subspace $V$ of
rank $r$ of $E\otimes F$, one has
\[\hat{\mu}(V)\leqslant\hat{\mu}_{\max}(\overline E)+\hat{\mu}_{\max}(\overline F).\]
Denote by $B(r)$ the condition
\[\forall\,n\geqslant 1,\;\forall\,m\geqslant1,\quad B(r,n,m).\]
With this notation, the general validity of the conjectural inequality (\ref{equamaxbis}) --- or equivalently, a positive answer to  Problems  \ref{mumaxconj}-\ref{HNtens} --- may be rephrased as: 
\[\forall\,r\geqslant 1,\quad B(r).\]

\begin{lemm}\label{Lem:semstable}
Let $p\geqslant 1$ and $r\geqslant 1$ be integers, $(\overline
E_i)_{i=1}^{p}$ and $(\overline F_i)_{i=1}^{p}$ be two families of
hermitian vector bundles over $\overline{\mathbb Q}$, and $V$ be a
vector subspaces of rank $r$ of
\[(\overline E_1\otimes\overline F_1)\oplus\cdots\oplus(\overline E_p\otimes\overline F_p).\]
Assume that the condition $B(s,n,m)$ holds for all
integers $s$, $n$ and $m$ such that
\[1\leqslant s\leqslant r,\quad n\leqslant\max_{1\leqslant i\leqslant p}\rang(E_i),\quad
m\leqslant\max_{1\leqslant i\leqslant p}\rang(F_i).\] For any
$i\in\{1,\ldots, p\}$, let $\mathcal F_i$ and $\mathcal G_i$ be
respectively the Harder-Narasimhan filtrations of $\overline E_i$
and $\overline F_i$. Let $\mathcal H=(\mathcal F_1\otimes\mathcal
G_1)\oplus\cdots\oplus(\mathcal F_p\otimes\mathcal G_p)$. Then one
has $\hat{\mu}(\overline V)\leqslant\mathbb E[\mathcal H|_V]$.
\end{lemm}
\begin{proof}
For any $t\in\mathbb R$, the subquotient $\mathrm{sq}_{\mathcal
H|_V}^t(V)$ identifies with a vector subspace of
\[\mathrm{sq}_{\mathcal H}^t((E_1\otimes F_1)\oplus\cdots\oplus(E_p\otimes F_p))=\bigoplus_{i=1}^p\bigoplus_{a+b=t}\mathrm{sq}_{\mathcal F_i}^a(E_i)\otimes\mathrm{sq}_{\mathcal G_i}^{b}(F_i).\]
Moreover, the height of the inclusion map is non-positive. Since
$\mathrm{sq}_{\mathcal F_i}^a(\overline E_i)$ (resp.
$\mathrm{sq}_{\mathcal G_i}^{b}(\overline F_i)$) is either zero or
semistable of slope $a$ (resp. $b$), by the hypothesis of the lemma,
for any non-zero vector subspace $W$ of $\mathrm{sq}_{\mathcal
F_i}^aE\otimes\mathrm{sq}_{\mathcal G_i}^bF$ such that
$\rang(W)\leqslant r$, one has $\hat{\mu}(\overline W)\leqslant
t$. Therefore, under the hypothesis of the lemma, one has
\[\widehat{\deg}_{\mathrm{n}}(\mathrm{sq}_{\mathcal H|_V}^t(\overline V))\leqslant t\rang(\mathrm{sq}_{\mathcal H|_V}^t(V)).\]
By taking the sum with respect to $t$ one obtains
$\widehat{\deg}_{\mathrm{n}}(\overline V)\leqslant \rang(V)\mathbb
E[\mathcal H|_V]$ and hence $\hat{\mu}(\overline
V)\leqslant\mathbb E[\mathcal H|_V]$.
\end{proof}

\begin{theo}\label{Thm:semi}
Let $r\geqslant 1$ be an integer, $\overline E$ and $\overline F$ be
two hermitian vector bundles over $\overline{\mathbb Q}$, and $V$ be
a subspace of rank $r$ of $E\otimes F$. Assume that $B(s)$ holds for
any $s\in\{1,\ldots,r\}$.
\begin{enumerate}[1)]
\item If $V$ is  left semistable, then $\hat{\mu}(\overline V)\leqslant\hat{\mu}(\overline E)+\hat{\mu}_{\max}(\overline F)$.
\item If $V$ is right semistable, then $\hat{\mu}(\overline V)\leqslant\hat{\mu}_{\max}(\overline E)+\hat{\mu}(\overline F)$.
\item If $V$ is both-sided semistable, then $\hat{\mu}(\overline V)\leqslant\hat{\mu}(\overline E)+\hat{\mu}(\overline F)$.
\end{enumerate}
\end{theo}
\begin{proof}
1) Let $\mathcal F$ be the Harder-Narasimhan filtration of $E$
(indexed by $\mathbb R$) and $\mathcal G_0$ be the trivial
filtration of $F$. For any $t\in\mathbb R$, the subquotient
$\mathrm{sq}^t(V)$ with respect to the induced filtration $(\mathcal
F\otimes\mathcal G_0)|_V$ can be considered as a vector subspace of
$\mathrm{sq}_{\mathcal F}^t(E)\otimes F$. Moreover, the height of
the inclusion map $\mathrm{sq}^t(V)\rightarrow\mathrm{sq}_{\mathcal
F}^t(E)\otimes F$ is non-positive. Since the rank of
$\mathrm{sq}^t(V)$ is no more than $r$, by the hypothesis of the
theorem we obtain
\[\ndeg(\mathrm{sq}^t(\overline V))\leqslant \rang(\mathrm{sq}^t(V))(t+\hat{\mu}_{\max}(\overline F)),\]
where we have used the fact that either $\mathrm{sq}_{\mathcal
F}^t(E)=0$ or $\mathrm{sq}_{\mathcal F}^t(\overline E)$ is
semistable of slope $t$. By taking the sum with respect to $t$, we
obtain
\[\ndeg(\overline V)\leqslant\rang(V)\mathbb E[(\mathcal F\otimes\mathcal G_0)|_V]
+\rang(V)\hat{\mu}_{\max}(\overline F)\leqslant\rang(V)(\mathbb
E[\mathcal F]+\hat{\mu}_{\max}(\overline F)),\] where the second
inequality comes from the condition that $V$ is left semistable.
Since $\mathcal F$ is the Harder-Narasimhan filtration, one has
$\mathbb E[\mathcal F]=\hat{\mu}(\overline E)$. So we obtain the
desired inequality.

2) The proof is quite similar to that of 1). We need to consider the
Harder-Narasimhan filtration of $\overline F$.

3) Let $\mathcal F$ be the Harder-Narasimhan filtration of
$\overline E$ and $\mathcal G$ be that of $\overline F$. Denote by
$\mathcal H$ the tensor product filtration $\mathcal
F\otimes\mathcal G$. By Lemma \ref{Lem:semstable}, one has
\[\hat{\mu}(\overline V)\leqslant\mathbb E[\mathcal H|_V]
\leqslant\mathbb E[\mathcal F]+\mathbb E[\mathcal G]=
\hat{\mu}(\overline E)+\hat{\mu}(\overline F),\] where the
second inequality comes from the both-sided semistability of $V$.
\end{proof}

\begin{prop}\label{Pro:stable}
Let $r\geqslant 1$ be an integer, $\overline E$ and $\overline F$ be
two hermitian vector bundles over $\overline{\mathbb Q}$ which are
semistable, and $V$ be a subspace of rank $r$ of $E\otimes F$.
Assume that condition $B(s,n',m')$ holds for\footnote{Namely, $n'\leqslant n$,
$m'\leqslant m$, and $n'+m'<n+m$.} $(n',m')<(n,m)$ and $1\leqslant s\leqslant r$. If $V$ is not
both-sided stable, then
\begin{equation}\label{Equ:majoration}\hat{\mu}(\overline V)\leqslant\hat{\mu}(\overline E)+\hat{\mu}(\overline F).\end{equation}
\end{prop}
\begin{proof}
Let $\mathbf{S}$ be the subset of
$\mathbf{Fil}(E)\times\mathbf{Fil}(F)$ consisting of filtration
pairs $(\mathcal F,\mathcal G)$ such that $\mathbb E[\mathcal
F]=\mathbb E[\mathcal G]=0$ and that both filtrations $\mathcal F$
and $\mathcal G$ are not trivial. Let
$\Theta:\mathbf{S}\rightarrow\mathbb R$ be the function defined as
\[\Theta(\mathcal F,\mathcal G)=-\frac{\mathbb E[(\mathcal F\otimes\mathcal G)|_V]}{(\|\mathcal F\|^2+\|\mathcal G\|^2)^{1/2}}.\] By Proposition \ref{Pro:destabilite}, there exists a filtration pair $(\mathcal F_1,\mathcal G_1)$ such that $c:=\Theta(\mathcal F_1,\mathcal G_1)$ is the minimum value of the function $\Theta$. One has $c\leqslant 0$. Moreover, for any filtration $\mathcal F$ of $E$ and any filtration $\mathcal G$ of $F$ one has
\begin{equation}\label{Equ:FG-}\mathbb E[\mathcal F]+\mathbb E[\mathcal G]-\mathbb E[(\mathcal F\otimes\mathcal G)|_V]\geqslant c\frac{\langle\mathcal F_1,\mathcal F\rangle+\langle\mathcal G_1,\mathcal G\rangle}{
(\|\mathcal F_1\|^2+\|\mathcal G_1\|^2)^{1/2}}.\end{equation}

We construct filtrations $\mathcal F$ and $\mathcal G$ as in
\S\ref{Subsubsec:refinement} such that, for any $t\in\mathbb R$, the
filtration $\mathcal F$ (resp. $\mathcal G$) induces by subquotient
the Harder-Narasimhan filtration of $\mathrm{sq}_{\mathcal
F_1}^t(\overline E)$ (resp. $\mathrm{sq}_{\mathcal G_1}^t(\overline
F)$). Let $\mathcal H_1$ be the restriction of $\mathcal
F_1\otimes\mathcal G_1$ on $V$. For any $t\in\mathbb R$, the
subquotient $\mathrm{sq}_{\mathcal H_1}^t(V)$ identifies with a
vector subspace of
\[\mathrm{sq}_{\mathcal F_1\otimes\mathcal G_1}^t(E\otimes F)=
\bigoplus_{a+b=t}\mathrm{sq}_{\mathcal F_1}^a(E)\otimes
\mathrm{sq}_{\mathcal G_1}^{b}(F).\] Moreover, the height of the
inclusion map $\mathrm{sq}^t_{\mathcal
H_1}(V)\rightarrow\mathrm{sq}^t_{\mathcal F_1\otimes\mathcal
G_1}(V)$ is non-positive. Since both filtrations $\mathcal F_1$ and
$\mathcal G_1$ are not degenerated, by Lemma \ref{Lem:semstable},
one has
\[\widehat{\deg}_{\mathrm{n}}(\mathrm{sq}_{\mathcal H_1}^t(\overline V))
\leqslant \rang(\mathrm{sq}_{\mathcal H_1}^t(V))\mathbb
E[\widetilde{\mathcal H}_t],\] where $\widetilde{\mathcal H}_t$ is
the restriction of the subquotient filtration of
$\mathrm{sq}_{\mathcal F_1\otimes\mathcal G_1}^t(E\otimes F)$
(induced by $\mathcal F\otimes\mathcal G$) on
$\mathrm{sq}^t_{\mathcal H_1}(V)$. By taking the sum with respect to
$t$, one obtains
\[\widehat{\deg}_{\mathrm{n}}(\overline V)\leqslant\rang(V)
\mathbb E[(\mathcal F\otimes\mathcal G)|_V].\] Hence
\[\hat{\mu}(\overline V)\leqslant\mathbb E[(\mathcal F\otimes\mathcal G)|_V]\leqslant -c\frac{\langle\mathcal F_1,\mathcal F\rangle+\langle\mathcal G_1,\mathcal G\rangle}{
(\|\mathcal F_1\|^2+\|\mathcal G_1\|^2)^{1/2}}+\mathbb E[\mathcal
F]+\mathbb E[\mathcal G].\] For any $t\in\mathbb R$, the filtration
$\mathcal F$ induces by subquotient a filtration on
$\mathrm{sq}_{\mathcal F_1}^t(E)$ which coincides with the
Harder-Narasimhan filtration of $\mathrm{sq}_{\mathcal
F_0}^t(\overline E)$. By \eqref{Equ:EG}, one has
\[\mathbb E[\mathcal F]=\sum_{t\in\mathbb R}\frac{\rang(\mathrm{sq}_{\mathcal F_1}^t(E))}{\rang(E)}\hat{\mu}(\mathrm{sq}_{\mathcal F_1}^t(\overline E))=\hat{\mu}(\overline E).\]
Similarly, one has $\mathbb E[\mathcal G]=\hat{\mu}(\overline
F)$. Moreover, by \eqref{Equ:innerFG}, one has
\[\begin{split}\langle\mathcal F_1,\mathcal F\rangle&=\sum_{t\in\mathbb R}
\frac{\rang(\mathrm{sq}_{\mathcal F_1}^t(E))}{\rang(E)}t\hat{\mu}(\mathrm{sq}_{\mathcal F_0}^t(\overline E))=\frac{1}{\rang(E)}\sum_{t\in\mathbb R}t\widehat{\deg}_{\mathrm{n}}(\mathrm{sq}_{\mathcal F_1}^t(\overline E))\\
&=-\frac{1}{\rang(E)}\int_{\mathbb
R}t\,\mathrm{d}\widehat{\deg}_{\mathrm{n}}(\mathcal F_1^t(\overline
E))=a\hat{\mu}(\overline
E)+\frac{1}{\rang(E)}\int_a^{+\infty}\widehat{\deg}_{\mathrm{n}}(\mathcal
F_1^t(\overline E))\,\mathrm{d}t,
\end{split}\]
where $a$ is a sufficiently negative number (such that $\mathcal
F_1^{a}(E)=E$). Since $\overline E$ is semistable, one obtains that
\[\langle\mathcal F_1,\mathcal F\rangle\leqslant a\hat{\mu}(E)+\hat{\mu}(E)\int_a^{\infty}\frac{\rang(\mathcal F_1^t(E))}{\rang(E)}\,\mathrm{d}t=\hat{\mu}(\overline E)\mathbb E[\mathcal F_1]=0.\]
Similarly, one has $\langle\mathcal G_1,\mathcal G\rangle \leqslant
0$. Therefore, $\hat{\mu}(\overline
V)\leqslant\hat{\mu}(\overline E)+\hat{\mu}(\overline F)$.
\end{proof}

\begin{theo}\label{Thm:recurrence}
Let $r\geqslant 1$ be an integer. Assume that
\begin{enumerate}[(1)]
\item $B(s)$ holds for any $s\in\{1,\ldots,r-1\}$,
\item for any stable hermitian vector bundles $\overline E$ and $\overline F$ over $\overline{\mathbb Q}$ such that $\rang(E)\geqslant 2$, $\rang(F)\geqslant 2$ and $\rang(E)\rang(F)\geqslant 2r$, and any vector subspace $V\subset E\otimes F$ of rank $r$ which is both-sided stable, one has
    \[\hat{\mu}({\overline V})\leqslant\hat{\mu}(\overline E)+\hat{\mu}(\overline F).\]
\end{enumerate}
Then the condition $B(r)$ holds.
\end{theo}
\begin{proof}
Let $\overline E$ and $\overline F$ be two hermitian vector bundles
over $\overline{\mathbb Q}$, and $V$ be a subspace of rank $r$ of
$E\otimes F$. Let $\mathcal A=(\mathbb N\setminus\{0\})^2$, equipped
with the order
\[(a_1,b_1)\leqslant (a_2,b_2)\text{ if and only if }a_1\leqslant a_2\text{ and }b_1\leqslant b_2.\] We shall prove the inequality \begin{equation}\label{Equ:muV1}\hat{\mu}(\overline V)\leqslant\hat{\mu}_{\max}(\overline E)+\hat{\mu}_{\max}(\overline F)\end{equation}
by induction on $(\rang(E),\rang(F))\in\mathcal A$. The case where
$\rang(E)=1$ or $\rang(F)=1$ is trivial. Assume that
\eqref{Equ:muV1} is verified for vector subspaces of rank $r$ in a
tensor product of hermitian vector bundles of ranks
$<(\rang(E),\rang(F))$. It suffices to treat the case where the
image of $V$ in $E$ (resp. $F$) equals $E$ (resp. $F$) since
otherwise the inequality \eqref{Equ:muV1} follows from the induction
hypothesis by replacing $\overline E$ and $\overline F$ by the
images of $V$ in $E$ and in $F$ (equipped with induced norms)
respectively.

If $\overline E$ is not stable, there exists a non-zero vector
subspace $E_1$ of $E$ such that $\overline E/\overline E_1$ is
stable and of slope $\leqslant\hat{\mu}(\overline E)$. Let
$V_1=V\cap (E_1\otimes F)$. One has $\rang(V_1)<\rang(V)$ (since $V$
is left stable, see Proposition \ref{Pro:semistable}). Hence (by the
hypothesis (1) of the theorem)
\[\ndeg(\overline V_1)\leqslant\rang(V_1)(\hat{\mu}_{\max}(\overline E_1)+\hat{\mu}_{\max}(\overline F))\leqslant\rang(V_1)(\hat{\mu}_{\max}(\overline E)+\hat{\mu}_{\max}(\overline F)).\]
Moreover, the quotient space $V/V_1$ identifies with a vector
subspace of $(E/E_1)\otimes F$, and the height of the inclusion map
is non-positive. Hence
\[\ndeg(\overline V/\overline V_1)\leqslant\rang(\overline V/\overline V_1)(\hat{\mu}_{\max}(\overline E/\overline E_1)+\hat{\mu}_{\max}(\overline F))\leqslant\rang(\overline V/\overline V_1)(\hat{\mu}_{\max}(\overline E)+\hat{\mu}_{\max}(\overline F)),\]
where we have used the hypothesis (1) of the theorem if $V_1\neq 0$,
and the induction hypothesis if $V_1=0$. The sum of the above two
inequalities gives
\[\ndeg(\overline V)\leqslant\rang(V)(\hat{\mu}_{\max}(\overline E)+\hat{\mu}_{\max}(\overline F)).\]
In a similar way, we can prove the same inequality in the case where
$\overline F$ is not stable.

In the following, we assume that both $\overline E$ and $\overline
F$ are stable. If $V$ is not both-sided
stable, by Proposition \ref{Pro:stable} we obtain the
inequality \eqref{Equ:muV1}.
By hypothesis (2) of the theorem, it remains the case where $V$ is both-sided stable and $r>\frac 12\rang(E)\rang(F)$. Let $W$ be the quotient space $(E\otimes F)/V$. One has $\rang(W)<\frac 12\rang(E)\rang(F)<r$. The dual space $W^\vee$ identifies with a subspace of $E^\vee\otimes F^\vee$. Thus induction hypothesis leads to the inequality $\hat{\mu}(\overline W^\vee)\leqslant\hat{\mu}(\overline E^\vee)+\hat{\mu}(\overline F^\vee)$ since $\overline E^\vee$ and $\overline F^\vee$ are stable. Consequently
\[\ndeg(\overline V)=\ndeg(\overline E\otimes\overline F)
-\ndeg(\overline W)\leqslant\ndeg(\overline E\otimes\overline
F)-\rang(W)(\hat{\mu}(\overline E)+\hat{\mu}(F)),\] which
implies $\hat{\mu}(\overline V)\leqslant \hat{\mu}(\overline
E)+\hat{\mu}(F)$.
\end{proof}

\section{Subbundles of small rank}\label{Sec:small rank}

This section is devoted to the proof of Theorem B. We will firstly
prove the assertion $B(r)$ for $r\leqslant 4$. The main tool is the
recursive argument that we have established in Theorem
\ref{Thm:recurrence}. Note that the assertion $B(1)$ has been proved
in Proposition \ref{Pro:majoration}.

\subsection{Subbundles of rank $2$}\label{SubSec:rang2}
In this subsection, we study planes in the tensor product of two
hermitian vector bundles over $\overline{\mathbb Q}$. The theorem
\ref{Thm:recurrence} permits us to reduce the problem to the
situation where the plane is both-sided stable in the tensor
product, which impose strong condition on the successive tensorial ranks of
the plan. We then apply the theorem of Zhang to obtain the desired
upper bound.

\begin{theo}\label{Thm:r2}
Let $\overline E$ and $\overline F$ be two hermitian vector bundles
over $\overline{\mathbb Q}$, and $V$ be a subspace of $E\otimes F$
which has rank $2$. Then one has \[\hat{\mu}(\overline V)
\leqslant\hat{\mu}_{\max}(\overline
E)+\hat{\mu}_{\max}(\overline F).\]
\end{theo}
\begin{proof}
By Theorem \ref{Thm:recurrence}, it suffices to treat the case where
$E$ and $F$ have ranks $\geqslant 2$, $V$ is both-sided stable, and
$\overline E$ and $\overline F$ are stable.

We claim that $\rho_{1}(V)\geqslant 2$. In fact, if
$\rho_{1}(V)=1$, then by Corollary \ref{Cor:leng} one obtains
\[1=\frac{2}{\rang(V)}<\frac{1}{\rang(E)}+\frac{1}{\rang(F)}\leqslant 1,\]
which leads to a contradiction. Thus one has
\[\rho_{1}(V)\rho_{2}(V)=4\geqslant \exp(2\ell(2))=\mathrm{e}.\]
By Proposition \ref{Pro:majo de mu}, we obtain the result.
\end{proof}

\subsection{Subbundles of rank $3$}

In this subsection, we study subbundles $V$ of rank $3$ in a tensor
product of two hermitian vector bundles over $\overline{\mathbb Q}$.
Just as the rank two case, the both-sided stability condition
imposes constraints to the successive tensorial ranks of $V$. In most
situations the theorem of Zhang permits us to conclude except two
cases where the successive tensorial ranks of $V$ are respectively $2,2,2$
and $2,2,3$. We shall treat these cases separately. The main idea is
to consider a suitably weighted average of the Zhang's upper bound
(\eqref{Equ:a2} and \eqref{Equ:aa}) for the Arakelov degree of
$\overline V$ and the upper bounds (\eqref{Equ:b1} and
\eqref{Equ:b}) obtained from the induction hypothesis.

\begin{theo}
Let $\overline E$ and $\overline F$ be two  hermitian vector bundles
over $\overline{\mathbb Q}$, and $V$ be a vector subspace of
$E\otimes F$ which has rank $3$. Then one has
\[\hat{\mu}(\overline V) \leqslant\hat{\mu}_{\max}(\overline
E)+\hat{\mu}_{\max}(\overline F).\]
\end{theo}
\begin{proof}
By Theorem \ref{Thm:recurrence}, it suffices to treat the case where
$E$ and $F$ have ranks $\geqslant 2$, $\rang(E)\rang(F)\geqslant 6$, $V$ is both-sided stable, and
$\overline E$ and $\overline F$ are stable.
In particular, at least one of the vector spaces $E$
and $F$ has rank $\geqslant 3$. By symmetry we may assume further
that $\rang(E)\geqslant\rang(F)$ (and hence $\rang(E)\geqslant 3$).
We claim that $\rho_{1}(V)\geqslant 2$, otherwise by Proposition
\ref{Pro:ell} one obtains
\[\frac{1}{3}=\frac{1}{\rang(V)}<\frac{1}{\rang(E)}\leqslant\frac 13,\]
which leads to a contradiction. Moreover, by Proposition
\ref{Pro:majo de mu}, the inequality
\[\hat{\mu}(\overline V)\leqslant\hat{\mu}_{\max}(\overline E)+\hat{\mu}_{\max}(\overline F).\]
holds once
\[\rho_{1}(V)\rho_{2}(V)\rho_{3}(V)\geqslant \exp(3\ell(3))=\mathrm{e}^{5/2}\approx 12.182493\ldots.\]
Hence it remains the following two cases.

\noindent 1) $\rho_{1}(V)=\rho_{2}(V)=\rho_{3}(V)=2$.

We fix a constant $\epsilon>0$ which is sufficiently small. By
Theorem \ref{Thm:Zhang} (see also Remark \ref{Rem:ssl}) we obtain
that there exist three lines $V_1$, $V_2$ and $V_3$ of tensorial rank $2$ in
$V$ such that $V=V_1+V_2+V_3$ and
\begin{equation}\label{Equ:a1}\ndeg(\overline V_1)+\ndeg(\overline V_2)+\ndeg(\overline V_3)\geqslant\ndeg(\overline V)-\frac{3}{2}\ell(3)-\epsilon.\end{equation}
Without loss of generality, we may assume $\ndeg(\overline
V_1)\geqslant\ndeg(\overline V_2)\geqslant\ndeg(\overline V_3)$. Let
$E_1$ and $F_1$ be the image of $V_1$ in $E$ and $F$ respectively.
The inequality \eqref{Equ:a1} implies
\begin{equation}\label{Equ:a2}\ndeg(\overline V)\leqslant 3\ndeg(\overline V_1)
+\frac32\ell(3)+\epsilon\leqslant 3\hat{\mu}(\overline E_1)
+3\hat{\mu}(\overline F_1)+\frac32\ell(3)-\frac
12\log(8)+\epsilon,\end{equation} where the second inequality comes
from Proposition \ref{Pro:majoration}.

The vector subspace $V$ being both-sided stable (hence it is left
stable), by Proposition \ref{Pro:semistable} one obtains
\[\frac{\rang(V\cap(E_1\otimes F))}{\rang(V)}<\frac{\rang(E_1)}{\rang(E)}\leqslant\frac{2}{3},\]
which implies that $\rang(V\cap(E_1\otimes F))=1$ (and therefore
$V\cap(E_1\otimes F)=V_1$). Thus we can identify $V/V_1$ with a
vector subspace of $(E/E_1)\otimes F$. Hence (by Theorem
\ref{Thm:r2})
\begin{equation*}\begin{split}
\ndeg(\overline V/\overline V_1)&\leqslant 2\hat{\mu}_{\max}(\overline E/\overline E_1)+2\hat{\mu}_{\max}(\overline F)\\
&\leqslant 6\hat{\mu}_{\max}(\overline
E)-4\hat{\mu}(\overline E_1)+2\hat{\mu}_{\max}(\overline F),
\end{split}
\end{equation*}
where the second inequality comes from Proposition
\ref{Pro:mumaxquot}. Therefore, by using the inequality (which
results from Proposition \ref{Pro:majoration})
\[\ndeg(\overline V_1)\leqslant\hat{\mu}(\overline E_1)+\hat{\mu}(\overline F_1)-\frac 12\log(2),\]
one obtains
\begin{equation}\label{Equ:b1}\ndeg(\overline V)\leqslant 6\hat{\mu}_{\max}(\overline E)-3\hat{\mu}(\overline E_1)+3\hat{\mu}_{\max}(\overline F)-\frac 12\log(2).\end{equation}
\eqref{Equ:a2}+\eqref{Equ:b1} gives
\[2\ndeg(\overline V)\leqslant 6\hat{\mu}_{\max}(\overline E)+6\hat{\mu}_{\max}(\overline F)+\frac32\ell(3)-\frac 12\log(8)-\frac 12\log(2)+\epsilon\]
Since $\frac{3}{2}\ell(3)-2\log(2)=-0.13629\ldots$, we obtain the result.
\vspace{3mm}

\noindent 2) $\rho_{1}(V)=\rho_{2}(V)=2$, $\rho_{3}(V)=3$.

We fix a constant $\epsilon>0$ which is sufficiently small. By
Theorem \ref{Thm:Zhang} (see also Remark \ref{Rem:ssl}) we obtain
that there exists three lines $V_1$, $V_2$ and $V_3$ in $V$ such
that $V=V_1+V_2+V_3$, $\rho(V_i)\geqslant\rho_{i}(V)$
($i\in\{1,2,3\}$) and
\begin{equation}\label{Equ:a}\ndeg(\overline V_1)+\ndeg(\overline V_2)+\ndeg(\overline V_3)\geqslant\ndeg(\overline V)-\frac32\ell(3)-\epsilon.\end{equation}
We may assume further that $\rho(V_1)=\rho(V_2)=2$ and $\rho(V_3)=3$
since the case where $\rho(V_1)=\rho(V_2)=\rho(V_3)=2$ has been
treated above.

Let $E_1$ and $F_1$ be the image of $V_1$ in $E$ and $F$
respectively. By Proposition \ref{Pro:majoration}, one has
\begin{gather*}
\ndeg(\overline V_1)\leqslant\hat{\mu}(\overline E_1)+\hat{\mu}_{\max}(\overline F)-\frac 12\log(2),\\
\ndeg(\overline V_2)\leqslant\hat{\mu}_{\max}(\overline E)+\hat{\mu}_{\max}(\overline F)-\frac 12\log(2),\\
\ndeg(\overline V_3)\leqslant\hat{\mu}_{\max}(\overline
E)+\hat{\mu}_{\max}(\overline F)-\frac 12\log(3).
\end{gather*}
Hence the inequality \eqref{Equ:a} implies
\begin{equation}\label{Equ:aa}\ndeg(\overline V)\leqslant \hat{\mu}(\overline E_1)
+2\hat{\mu}_{\max}(\overline E)+3\hat{\mu}_{\max}(\overline
F)+\frac32\ell(3)-\frac 12\log(12)+\epsilon.\end{equation} Similarly to
the case 1), the both-sided stability of $V$ implies that
$V\cap(E_1\otimes F)=V_1$ and thus $V/V_1$ identifies with a vector
subspace of $(E/E_1)\otimes F$. Hence
\begin{equation*}\begin{split}
\ndeg(\overline V/\overline V_1)&\leqslant 2\hat{\mu}_{\max}(\overline E/\overline E_1)+2\hat{\mu}_{\max}(\overline F)\\
&\leqslant 6\hat{\mu}_{\max}(\overline
E)-4\hat{\mu}(\overline E_1)+2\hat{\mu}_{\max}(\overline F),
\end{split}
\end{equation*}
where the second inequality comes from Proposition
\ref{Pro:mumaxquot}. Hence
\begin{equation}\label{Equ:b}\ndeg(\overline V)\leqslant 6\hat{\mu}_{\max}(\overline E)-3\hat{\mu}(\overline E_1)+3\hat{\mu}_{\max}(\overline F)-\frac 12\log(2).\end{equation}
Note that 3$\times$\eqref{Equ:aa}+\eqref{Equ:b} gives
\[4\ndeg(\overline V)\leqslant 12\hat{\mu}_{\max}(\overline E)+12\hat{\mu}_{\max}(\overline F)+\frac92\ell(3)-\frac{3}{2}\log(12)-\frac 12\log(2)+3\epsilon.\]
Since
\[\frac92\ell(3)-\frac{3}{2}\log(12)-\frac 12\log(2)=\frac{15}{4}-\frac{3}{2}\log(12)-\frac 12\log(2)=-0.323933\ldots,
\]
we obtain the desired result.
\end{proof}

\begin{rema}\label{Rem:conclure}
By Lemma \ref{Lem:dualite}, we obtain from the result of this subsection that, if $\overline E$ and $\overline F$ are hermitian vector bundles over $\overline{\mathbb Q}$ such that $\rang(E).\rang(F)\leqslant 6$, then for any non-zero subspace $V$ of $E\otimes F$, one has $\hat{\mu}(\overline V)\leqslant\hat{\mu}_{\max}(\overline E)+\hat{\mu}_{\max}(\overline F)$. In other words, the condition $B(r,m,n)$ holds for $mn\leqslant 6$ and for any $r$.
\end{rema}

\subsection{Subbundles of rank $4$}

Let $\overline E$ and $\overline F$ be two hermitian vector bundles
over $\overline{\mathbb Q}$, and $V$ a subspace of $E\otimes F$
which has rank $4$. Denote by $r=\rang(E)$ and $s=\rang(F)$. The aim of this section is to prove
$\hat{\mu}(\overline V)\leqslant\hat{\mu}_{\max}(\overline
E)+\hat{\mu}_{\max}(\overline F)$. By Theorem
\ref{Thm:recurrence}, we may assume in what follows that  $\overline
E$ and $\overline F$ are stable and have ranks $\geqslant 2$, $\rang(E)\rang(F)\geqslant 8$ and $V$ is both-sided stable.
Moreover, by Proposition \ref{Pro:majo de mu}, the required inequality holds
once
\[\rho_{1}(V)\rho_{2}(V)\rho_{3}(V)\rho_{4}(V)\geqslant \exp(4\ell(4))=\mathrm{e}^{13/3}\approx 76.19785\ldots.\]

In the following, we assume in addition that
\[
\rho_{1}(V)\rho_{2}(V)\rho_{3}(V)\rho_{4}(V)\leqslant 76.\]
In particular, one has $\rho_{1}(V)\leqslant 2$ (since
$3^4=81>76$). Without loss of generality, we assume
$\rang(E)\geqslant\rang(F)$. Hence $\rang(E)\geqslant 3$.

Compared to the rank $2$ or $3$ cases, the situation here is more
complicated. Firstly, the both-sided stability condition of $V$ does
not permit us to eliminate the case where $\rho_{1}(V)=1$. This
difficulty can be recovered by using the symmetry between $\overline
E$ and $\overline F$. Secondly, in the case where $\rho_{2}(V)=2$
and $r=s=3$, the devissage technic that we have applied in the
previous subsection does not provide the desired upper bound. Our
strategy is to prove the upper bound in a particular rank $5$ case
(see Lemma \ref{Lem:rank5}) and then use a duality argument to
reduced to the rank $4$ case.

The rest of this subsection is organized as follows. We firstly
treat the case where the vector space $V$ contains a line $V_1$
which has tensorial rank $1$. The crucial argument is to show that the
quotient space $V/V_1$ still have the geometric semistability. In
the second step, we  consider the case where $\rho_{1}(V)=2$. We
begin with an upper bound \eqref{Equ:4.0} of the Arakelov degree of
$\overline V$ by using the theorem of Zhang and the upper bound
\eqref{Equ:sudeg1} of the Arakelov degree of hermitian line
subbundle in a tensor product bundle. Then we establish upper bounds
\eqref{Equ:4.1}---\eqref{Equ:4.4} of $\ndeg(\overline V)$ by using
the induction hypothesis under diverse geometric conditions
(presented in items {\bf 1.} and {\bf 2.}). It turns out that a
suitable average of these two types of upper bounds give the desired
inequality in most situations (presented in items {\bf a.}, {\bf b.}
and {\bf c.}). The exceptional cases is where $r=s=3$, treated in
the item {\bf d.}, where we use the rank 5 trick, as explained
above.

\subsection*{Case 1: $\rho_{1}(V)=1$}
By Proposition \ref{Pro:ell} one obtains
$\rang(V)>\max(\rang(E),\rang(F))$. Hence the only possibility is
$\rang(E)=\rang(F)=3$. Let $V_1$ be a line of tensorial rank $1$ in $V$,
$E_1$ and $F_1$ be respectively the image of $V_1$ in $E$ and $F$.
Since $V$ is both-sided stable, by Proposition \ref{Pro:semistable}
one obtains
\[\rang(V\cap(E_1\otimes F))<\frac{\rang(V)\rang(E_1)}{\rang(E)}=\frac{4}{3}.\]
Hence $V\cap(E_1\otimes F)=V_1$. For the same reason, one has
$V\cap(E\otimes F_1)=V_1$. The quotient space $V/V_1$ then
identifies either with a vector subspace of $(E/E_1)\otimes F$, or a
vector subspace of $E\otimes (F/F_1)$. Moreover, by Proposition
\ref{Pro:semistable} one obtains that, for any subspace $E'$ of $E$
which has rank $2$, one has \[\rang(V\cap(E'\otimes
F))<\frac{\rang(V)\rang(E')}{\rang(E)}=8/3.\] Hence
$\rang(V\cap(E'\otimes F))\leqslant 2$. If in addition $E'$ contains
$E_1$, then \[\rang\big((V/V_1)\cap((E'/E_1)\otimes F)\big)\leqslant
1<\frac{\rang(V/V_1)\rang(E'/E_1)}{\rang(E/E_1)}=3/2.\] Hence (by
Proposition \ref{Pro:semistable}) $V/V_1$ is left stable as a vector
subspace of $(E/E_1)\otimes F$. For the same reason, $V/V_1$ is
right stable as a vector subspace of $E\otimes (F/F_1)$.

By the results obtained in previous subsections together with
Theorem \ref{Thm:semi} one obtains
\begin{gather}\label{Equ:v/v11}\ndeg(\overline V/\overline V_1)
\leqslant 3(\hat{\mu}(\overline E/\overline
E_1)+\hat{\mu}(\overline F))=
\frac{9}{2}\hat{\mu}(\overline E)+3\hat{\mu}(\overline F)-\frac32\hat{\mu}(\overline E_1),\\
\label{Equ:v/v12}\ndeg(\overline V/\overline V_1)\leqslant
3(\hat{\mu}(\overline E)+\hat{\mu}(\overline F/\overline
F_1))=\frac 92\hat{\mu}(\overline F)+3\hat{\mu}(\overline
E)-\frac 32\hat{\mu}(\overline F_1).\end{gather} Moreover, by
Proposition \ref{Pro:majoration} one has
\begin{equation}\label{Equ:muE1}
\ndeg(\overline V_1)\leqslant\hat{\mu}(\overline
E_1)+\hat{\mu}(\overline F_1).
\end{equation}
\eqref{Equ:v/v11}+\eqref{Equ:v/v12}+2$\times$\eqref{Equ:muE1} gives
\[2\ndeg(\overline V)\leqslant\frac{15}{2}(\hat{\mu}(\overline E)+\hat{\mu}(\overline F))+\frac 12(\hat{\mu}(\overline E)+\hat{\mu}(\overline F))\leqslant 8(\hat{\mu}(\overline E)+\hat{\mu}(\overline F)),\]
as required.

\subsection*{Case 2: $\rho_{1}(V)=2$.}
We fix a real number $\varepsilon>0$ which is sufficiently small. By
Theorem \ref{Thm:Zhang} (see also Remark \ref{Rem:ssl}), there
exists four lines $V_i$ ($i=1,2,3,4$) in $V$ which span $V$ as a
linear space over $\overline{\mathbb Q}$ and such that
\begin{enumerate}[(1)]
\item $\rho(V_1)\leqslant\rho(V_2)\leqslant\rho(V_3)\leqslant\rho(V_4)$,
\item for any $i\in\{1,2,3,4\}$, $\rho(V_i)\geqslant\rho_{i}(V)$,
\item $\ndeg(\overline V_1)+\ndeg(\overline V_2)+\ndeg(\overline V_3)+\ndeg(\overline V_4)\geqslant
    \ndeg(\overline V)-2\ell(4)-\varepsilon$.
\end{enumerate}
Let $a$ be the last index in $\{1,2,3,4\}$ such that $\rho(V_a)=2$.
By permuting the lines, we may assume that $\ndeg(\overline
V_1)\geqslant\ldots\geqslant\ndeg(\overline V_a)$.

Denote by $r$ and $s$ the rank of $E$ and $F$ over
$\overline{\mathbb Q}$ respectively. Remind that we have assumed
$r\geqslant s$ (and hence $r\geqslant 3$ since $rs\geqslant 8$).

We choose a line $V'$ in $V$ of tensorial rank $2$ such that
\begin{equation}\label{Equ:ndegV'}\ndeg(\overline V{}')\geqslant \sup_{M}\ndeg(\overline M)-\varepsilon,\end{equation}
where $M$ runs over all lines of tensorial rank $2$ in $V$. Let $E_1$ and
$F_1$ be respectively the images of $V'$ in $E$ and $F$. By
Proposition \ref{Pro:majoration}, one has
\begin{equation}\label{Equ:deg V'}\ndeg(\overline V{}')\leqslant\hat{\mu}(\overline E_1)
+\hat{\mu}(\overline F_1)-\frac 12\log(2).\end{equation} As a
consequence,
\begin{equation}\label{Equ:4.0}\begin{split}&\quad\;\ndeg(\overline V)\leqslant \sum_{i=1}^4\ndeg(\overline V_i)+2\ell(4)+\varepsilon\\
&\leqslant a(\hat{\mu}(\overline E_1)+\hat{\mu}(\overline F_1)-\log\sqrt{2}+\varepsilon)+(4-a)(\hat{\mu}(\overline E)+\hat{\mu}(\overline F)-\log\sqrt{3})+2\ell(4)+\varepsilon\\
&=a(\hat{\mu}(\overline E_1)+\hat{\mu}(\overline F_1))
+(4-a)(\hat{\mu}(\overline E)+\hat{\mu}(\overline F))
+a\log\sqrt{3/2}+4(\frac 12\ell(4)-\log\sqrt{3})+(a+1)\varepsilon\\
&\leqslant a(\hat{\mu}(\overline E_1)+\hat{\mu}(\overline
F_1)) +(4-a)(\hat{\mu}(\overline E)+\hat{\mu}(\overline F))
+a\log\sqrt{3/2}
\end{split}\end{equation}
provided that $\varepsilon$ is sufficiently small since
$\frac12\ell(4)-\log\sqrt{3}=-0.007639...<0$.

The rank of $E_1$ being $2$, $E$ contains $E_1$ strictly. Hence (by
Proposition \ref{Pro:ell}) one has
\begin{equation}\label{Equ:r}1\leqslant
\rang(V\cap(E_1\otimes
F))<\frac{\rang(V)\rang(E_1)}{\rang(E)}=\frac{8}{r},\end{equation}
which implies that $3\leqslant r\leqslant 7$. For the same reason,
one has $2\leqslant s\leqslant 7$. Moreover, \eqref{Equ:r} also
implies $\rang(V\cap(E_1\otimes F))\leqslant 2$, and
$\rang(V\cap(E_1\otimes F))=2$ only if $r=3$. The following are
several upper bounds of $\ndeg(\overline V)$ according to the values
of $\rang(V\cap(E_1\otimes F))$ and $\rang(V\cap(E\otimes F_1))$.

\vspace{2mm}\noindent{\bf 1.} Assume that $\rang(V\cap(E_1\otimes
F))=2$. Let $W$ be the intersection of $E_1\otimes F$ with $V$,
which has rank $2$. Since $\rho_{1}(\overline V)=2$, by Corollary
\ref{Cor:siegel} one has
\[\ndeg(\overline W)\leqslant 2(\ndeg(\overline V{}')+\varepsilon)+\ell(2)\leqslant 2(\hat{\mu}(\overline E_1)+
\hat{\mu}(\overline F_1)-\log\sqrt{2}+\frac 12\ell(2))+2\varepsilon,\]
where in the first inequality we have used \eqref{Equ:ndegV'}, and
the second inequality comes from \eqref{Equ:deg V'}. Note that
$\log\sqrt{2}-\frac 12\ell(2)=0.09657\ldots>0$. The quotient space $V/W$
identifies with a subspace of $(E/E_1)\otimes F$. Hence
\[\ndeg(\overline V/\overline W)\leqslant 2(\hat{\mu}(\overline E/\overline E_1)+\hat{\mu}(\overline F))=6\hat{\mu}(\overline E)
+2\hat{\mu}(\overline F)-4\hat{\mu}(\overline E_1),\] where
we have used the fact that $r=3$ (and hence $\rang(E/E_1)=1$).
Therefore
\begin{equation}\label{Equ:4.1}
\ndeg(\overline V)\leqslant 6\hat{\mu}(\overline
E)+2\hat{\mu}(\overline F)-2\hat{\mu}(\overline E_1)+2
\hat{\mu}(\overline F_1)-(\log(2)-\ell(2))+2\varepsilon.
\end{equation}
The same argument shows that, if $s\geqslant 3$ and if
$\deg_V(F_1)=2$ (and hence $s=3$), then
\begin{equation}\label{Equ:4.2}
\ndeg(\overline V)\leqslant 6\hat{\mu}(\overline
F)+2\hat{\mu}(\overline E)-2\hat{\mu}(\overline F_1)+2
\hat{\mu}(\overline E_1)-(\log(2)-\ell(2))+2\varepsilon.
\end{equation}

\vspace{2mm}\noindent{\bf 2.} Assume that $\rang(V\cap(E_1\otimes
F))=1$. Then $V\cap (E_1\otimes F)=V'$ and hence $V/V'$ identifies
with a subspace of $(E/E_1)\otimes F$. For any subspace $E'$ of $E$
such that $E_1\subsetneq E'\subsetneq E$, by Proposition
\ref{Pro:ell} one obtains
\[\rang(V\cap(E'\otimes F))<\frac{\rang(V)\rang(E')}{\rang(E)}=\frac{4r'}{r},\] where $r'$ is the rank of $E'$. Therefore \[\rang\big((V/V')\cap((E/E_1)\otimes F)\big)=\rang(V\cap(E'\otimes F))-1\leqslant \lceil 4r'/r\rceil-2,\]
which is no larger than
\[\frac{\rang(V/V')\rang(E'/E_1)}{\rang(E/E_1)}
=\frac{3(r'-2)}{r-2}.\]
Hence $V/V_1$ is left semistable as a vector subspace of
$(E/E_1)\otimes F$. By the results that we have obtained in the
previous subsection, together with Theorem \ref{Thm:semi}, we obtain
\[\ndeg(\overline V/\overline V{}')\leqslant 3(\hat{\mu}(\overline E/\overline E_1)+\hat{\mu}(\overline F))=\frac{3r}{r-2}\hat{\mu}(\overline E)+3\hat{\mu}(\overline F)-\frac{6}{r-2}\hat{\mu}(\overline E_1).\]
By \eqref{Equ:deg V'}, we obtain
\begin{equation}\label{Equ:4.3}
\ndeg(\overline V)\leqslant\frac{3r}{r-2}\hat{\mu}(\overline
E)+3\hat{\mu}(\overline
F)-\frac{8-r}{r-2}\hat{\mu}(\overline
E_1)+\hat{\mu}(\overline F_1)-\log\sqrt{2}.
\end{equation}
Similarly, if $s\geqslant 3$ and if $\rang(V\cap (E\otimes F_1))=1$,
then
\begin{equation}\label{Equ:4.4}
\ndeg(\overline V)\leqslant\frac{3s}{s-2}\hat{\mu}(\overline
F)+3\hat{\mu}(\overline
E)-\frac{8-s}{s-2}\hat{\mu}(\overline
F_1)+\hat{\mu}(\overline E_1)-\log\sqrt{2}.
\end{equation}

We now show that these upper bounds imply the inequality
\[\ndeg(V)\leqslant 4(\hat{\mu}(\overline
E)+\hat{\mu}(\overline F)).\]

\vspace{2mm}\noindent{\bf a.} We first treat the case where
$\rang(V\cap (E_1\otimes F))=2$, which implies $r=3$. Since we have
assumed $rs\geqslant 8$ and $r\geqslant s$, one has $s=3$. If $\rang(V\cap(E\otimes
F_1))=2$, then the sum of \eqref{Equ:4.1} and \eqref{Equ:4.2} gives
\[2\ndeg(\overline V)\leqslant 8(\hat{\mu}(\overline E)+\hat{\mu}(\overline F))-2(\log(2)-\ell(2))+4\varepsilon
<8(\hat{\mu}(\overline E)+\hat{\mu}(\overline F))\] provided
that $\varepsilon$ is sufficiently small. If $\rang(V\cap(E\otimes
F_1))=1$, then
$3a\times$\eqref{Equ:4.1}+$2a\times$\eqref{Equ:4.4}+$4\times$\eqref{Equ:4.0}
gives
\[\begin{split}
&(5a+4)\ndeg(\overline V)\leqslant (20a+16)(\hat{\mu}(\overline
E)+\hat{\mu}(\overline F))+
4a\log\sqrt{3/2}\\&\qquad-3a(\log(2)-\ell(2))-a\log(2)+6a\varepsilon.
\end{split}\]
Since
\[4\log\sqrt{3/2}-3(\log(2)-\ell(2))-\log(2)=-0.4617\ldots<0,\]
we obtain the required result.

\vspace{2mm}\noindent{\bf b.} We then treat the case where
$r\geqslant 5$. Note that in this case $\rang(V\cap(E_1\otimes
F))=1$. Hence
$((8-r)/(r-2))\times$\eqref{Equ:4.0}+$a\times$\eqref{Equ:4.3} gives
\[\Big(\frac{8-r}{r-2}+a\Big)\ndeg(V)
\leqslant 4\Big(\frac{8-r}{r-2}+a\Big)(\hat{\mu}(\overline
E)+\hat{\mu}(\overline F))+\frac{a(8-r)}{2(r-2)}\log(3/2)
-\frac{a}{2}\log(2).\] Since $(8-r)/(r-2)\leqslant 1$ and
$\log(3/2)<\log(2)$, we obtain the required result.

\vspace{2mm}\noindent{\bf c.} Consider the case where $r\leqslant 4$
and $\rang(V\cap(E_1\otimes F))=1$. Note that in this case $s$
cannot be $2$ since otherwise $\rho_{1}(V)=1$ by a dimension
argument\footnote{The locus of lines of tensorial rank $1$ in $E\otimes F$
is a subvariety of dimension $(r-1)+(s-1)$ (since it is isomorphic
to $\mathbb P(E^\vee)\times\mathbb P(F^\vee)$) in $\mathbb
P(E^\vee\otimes F^\vee)$, which has dimension $rs-1$.  Any vector
subspace of rank $4$ in $E\otimes F$ defines a subvariety of
dimension $3$ in $\mathbb P(E^\vee\otimes F^\vee)$, which should
have non-empty intersection with the locus of lines of tensorial rank $1$ if
$r\leqslant 4$ and $s=2$.}.

If $\rang(V\cap(E\otimes F_1))=2$, then
$3a\times$\eqref{Equ:4.2}+$2(r-2)a\times$\eqref{Equ:4.3}+$2(5-r)
\times$\eqref{Equ:4.0} gives
\[\begin{split}
((2r-1&)a+(10-2r))\ndeg(\overline V)\leqslant 4((2r-1)a+(10-2r))
\\
&+2(5-r)a\log\sqrt{3/2}-3a(\log
2-\ell(2))-2(r-2)a\log\sqrt{2}+2a\varepsilon.
\end{split}\]
Since $r\geqslant 3$,
\[2(5-r)a\log\sqrt{3/2}-3a(\log 2-\ell(2))-2(r-2)a\log\sqrt{2}\leqslant (2\log(3/2)-2\log(2))a<0,\]
as required.

If $\rang(V\cap(E\otimes F_1))=1$, then
$(r-2)a\times$\eqref{Equ:4.3}+$(s-2)a\times$\eqref{Equ:4.4}+
$(10-r-s)\times$\eqref{Equ:4.0} gives
\[\begin{split}(a(r+s-&4)+10-r-s)\ndeg(\overline V)
\leqslant 4(a(r+s-4)+10-r-s)(\hat{\mu}(\overline E)+\hat{\mu}(\overline F))\\
&+a(10-r-s)\log\sqrt{3/2}-a(r+s-4)\log\sqrt{2}.
\end{split}\]
If $r+s\geqslant 7$, one has
\[(10-r-s)\log\sqrt{3/2}-(r+s-4)\log\sqrt{2}\leqslant 3\log\sqrt{3/2}-3\log(2)<0\]
as required.

\vspace{2mm}\noindent{\bf d.} It remains the case where $r=s=3$ and
$\rang(V\cap(E_1\otimes F))=\rang(V\cap (E\otimes F_1))=1$. Consider
the following lemma.

\begin{lemm}\label{Lem:rank5}
Let $\overline A$ and $\overline B$ be two stable hermitian vector
bundles of rank $3$, and $C$ be a vector subspace of rank $5$ of
$A\otimes B$. Then one has
\[\hat{\mu}(\overline C)\leqslant \hat{\mu}(\overline A)+\hat{\mu}(\overline B).\]
\end{lemm}
\begin{proof}
By Proposition \ref{Pro:stable} and Remark \ref{Rem:conclure}, we may assume that $C$ is
both-sided stable. Moreover, the locus of lines of tensorial rank in
$A\otimes B$ is a subvariety of codimension $4$ in $\mathbb
P(A^\vee\otimes B^{\vee})$, hence has a non-empty intersection with
$C$. Namely the vector space $C$ contains a line $C_1$ with tensorial rank
$1$. Let $A_1$ and $B_1$ be respectively the image of $C_1$ in $A$
and $B$. The stability of $C$ in $A\otimes B$ implies (see
Proposition \ref{Pro:semistable})
\[\rang(C\cap(A_1\otimes B))<\frac{\rang(C)\rang(A_1)}{\rang(A)}=\frac{5}{3},\]
and hence $C\cap(A_1\otimes B)=C_1$. Thus $C/C_1$ can be considered
as a vector subspace of $(A/A_1)\otimes B$. Moreover, if $A'$ is a
vector subspace of $A$ such that $A_1\subsetneq A'\subsetneq A$,
then (by Proposition \ref{Pro:semistable})
\[\rang(C\cap(A'\otimes B))<\frac{\rang(C)\rang(A')}{\rang(A)}=\frac{10}{3},\]
namely $\rang(C\cap(A'\otimes B))\leqslant 3$. This implies that
\[\rang\big((C/C_1)\cap((A'/A_1)\otimes B)\big)\leqslant 2=\frac{\rang(C/C_1)\rang(A'/A_1)}{\rang(A/A_1)}.\]
By Proposition \ref{Pro:semistable}, $C/C_1$ is left semistable.
Hence the result obtained above together with Theorem \ref{Thm:semi}
implies that
\[\ndeg(\overline C/\overline C_1)\leqslant 4(\hat{\mu}(\overline A/\overline A_1)+\hat{\mu}(\overline B))=6\hat{\mu}(\overline A)+
4\hat{\mu}(\overline B)-2\hat{\mu}(\overline A_1).\]
Combined with the estimate
\[\ndeg(\overline C_1)\leqslant\hat{\mu}(\overline A_1)+\hat{\mu}(\overline B_1)\]
this inequality leads to
\begin{equation}\label{Equ:degC1}\ndeg(\overline C)\leqslant 6\hat{\mu}(\overline A)+4
\hat{\mu}(\overline B)-\hat{\mu}(\overline
A_1)+\hat{\mu}(\overline B_1).\end{equation} Finally, by
symmetry one has
\begin{equation}\label{Equ:degC2}\ndeg(\overline C)\leqslant 4\hat{\mu}(\overline A)+6
\hat{\mu}(\overline B)+\hat{\mu}(\overline
A_1)-\hat{\mu}(\overline B_1).\end{equation} The average of
\eqref{Equ:degC1} and \eqref{Equ:degC2} gives the required
inequality.
\end{proof}

We apply the above lemma to $\overline A=\overline E^\vee$,
$\overline B=\overline F^\vee$ and $\overline C=((\overline
E\otimes\overline F)/V)^\vee$ to get
\[-\ndeg(\overline E\otimes\overline F)+\ndeg(\overline V)\leqslant -5(\hat{\mu}(\overline E)+\hat{\mu}(\overline F)).\]
Since $\ndeg(\overline E\otimes\overline
F)=9(\hat{\mu}(\overline E)+\hat{\mu}(\overline F))$, we
obtain the required result.

We conclude the article by the following result, which follows from
the first part of Theorem B and the above lemma and completes the
proof of Theorem B.

\begin{theo}
Let $\overline E$ and $\overline F$ be hermitian vector bundles of
rank $\leqslant 3$ over $\overline{\mathbb Q}$. One has
\[\hat{\mu}_{\max}(\overline E\otimes\overline F)=\hat{\mu}_{\max}(\overline E)+\hat{\mu}_{\max}(\overline F).\]
\end{theo}

\backmatter

\bibliography{chen,complJBB}
\bibliographystyle{smfplain}

\end{document}